\documentclass[10pt]{amsart}
      \usepackage[mathscr]{eucal}
      \usepackage{amsmath,amsfonts}

\def\rg{\hbox to 30pt{\rightarrowfill}}
\def\lg{\hbox to 30pt{\leftarrowfill}}

      \parskip\smallskipamount
          \newtheorem{theorem}{Theorem}[section]
      
      \newtheorem{proposition}[theorem]{Proposition}
      \newtheorem{corollary}[theorem]{Corollary}
      \newtheorem{lemma}[theorem]{Lemma}
      \newtheorem{example}[theorem]{Example}
      
      \newtheorem{remark}[theorem]{Remark}
      \makeatletter
      \@addtoreset{equation}{section}
      \makeatother

      \newcommand{\BB}{{\mathbb B}}
      \newcommand{\CC}{{\mathbb C}}
      \newcommand{\NN}{{\mathbb N}}

      \newcommand{\DD}{{\mathbb D}}
      
      \newcommand{\FF}{{\mathbb F}}

      \newcommand{\cA}{{\mathcal A}}
      
      \newcommand{\cC}{{\mathcal C}}
      \newcommand{\cD}{{\mathcal D}}
      \newcommand{\cE}{{\mathcal E}}
      
      \newcommand{\cG}{{\mathcal G}}
      \newcommand{\cH}{{\mathcal H}}
      \newcommand{\cK}{{\mathcal K}}
      \newcommand{\cL}{{\mathcal L}}

      \newcommand{\cP}{{\mathcal P}}
      \newcommand{\cR}{{\mathcal R}}
      \newcommand{\cS}{{\mathcal S}}

      \newcommand{\cY}{{\mathcal Y}}
      \newcommand{\cX}{{\mathcal X}}

      \newdimen\expt
      \def\boxit#1{\setbox0\hbox{$\displaystyle{#1}$}
            \hbox{\lower.4\expt
       \hbox{\lower3\expt\hbox{\lower\dp0
            \hbox{\vbox{\hrule height.4\expt
       \hbox{\vrule width.4\expt\hskip3\expt
            \vbox{\vskip3\expt\box0\vskip2\expt}%
       \hskip3\expt\vrule width.4\expt}\hrule height.4\expt}}}}}}
      
\begin{document}
       \pagestyle{myheadings}
      \markboth{ Gelu Popescu}{ Free holomorphic functions on the unit ball
       of $B(\cH)^n$. II.}


      \title [ Free holomorphic functions on the unit ball
       of $B(\cH)^n$. II.]
      { Free holomorphic functions on the unit ball
       of $B(\cH)^n$. II.}
        \author{Gelu Popescu}
\date{January 27, 2008}
      \thanks{Research supported in part by an NSF grant}

      \address{Department of Mathematics, The University of Texas
      at San Antonio \\ San Antonio, TX 78249, USA}
      \email{\tt gelu.popescu@utsa.edu}

\begin{abstract}
In this paper we continue the study of
 free holomorphic functions
on the noncommutative ball
$$
[B(\cH)^n]_1:=\left\{ (X_1,\ldots, X_n)\in B(\cH)^n: \
\|X_1X_1^*+\cdots + X_nX_n^*\|^{1/2}<1\right\},
$$
where $B(\cH)$ is the algebra of all bounded linear operators on a
Hilbert space $\cH$, and $n=1,2,\ldots$ or $n=\infty$. Several
classical results from complex analysis have free analogues in our
noncommutative setting.

We prove a maximum principle, a Naimark type representation
theorem, and a Vitali
 convergence theorem,
 for free holomorphic functions with
operator-valued coefficients.
 We introduce the class of free holomorphic functions with the
radial infimum property  and study it in connection with
factorizations and noncommutative generalizations of  some
classical inequalities obtained by Schwarz and  Harnack. The
Borel-Carath\'eodory theorem is extended to our noncommutative
setting.

 Using a noncommutative generalization of Schwarz's lemma and
basic facts concerning the  free holomorphic  automorphisms of the
noncommutative ball $[B(\cH)^n]_1$, we obtain an analogue of Julia's
lemma for free holomorphic functions $F:[B(\cH)^n]_1\to
[B(\cH)^m]_1$. We also obtain Pick-Julia theorems for free
holomorphic functions with operator-valued coefficients.

 We provide  a noncommutative generalization of a classical
 inequality due to Lindel\"of, which turns out to be
  sharper then the noncommutative von Neumann inequality.

Finally, we introduce a pseudohyperbolic metric   on $[B(\cH)^n]_1$
which  is  invariant under the action of  the free holomorphic
automorphism group of $[B(\cH)^n]_1$ and turns out to be a
noncommutative extension of the pseudohyperbolic distance on
$\BB_n$, the open unit ball of $\CC^n$. In this setting, we  obtain
a Schwarz-Pick type lemma.

We also provide commutative versions of these results for
operator-valued multipliers of the Drury-Arveson space.

\end{abstract}

 \maketitle

\section*{Contents}
{\it

\quad Introduction

\begin{enumerate}
   \item[1.]    Free holomorphic functions: fractional transforms, maximum principle,  and geometric  structure
   \item[ 2.]    Vitali convergence and identity theorem for free holomorphic
functions
 \item[3.]    Free holomorphic functions with the radial infimum
 property
\item[ 4.]   Factorizations and  free  holomorphic versions of classical inequalities
\item[ 5.] Noncommutative Borel-Carath\'eodory theorems
 \item[ 6.]   Julia's lemma for holomorphic functions  on
 noncommutative balls
  \item[7.]   Pick-Julia  theorems for free holomorphic functions with
   operator-valued coefficients
\item[8.]  Lindel\"of inequality and sharpened forms of the noncommutative von Neumann
inequality
\item[9.] Pseudohyperbolic metric on the unit ball of $B(\cH)^n$ and an invariant Schwarz-Pick lemma
   \end{enumerate}

\quad References

}

\bigskip

\section*{Introduction}

In this paper,  we continue  our program to develop  a {\it
noncommutative  analytic function theory} on the unit ball of
$B(\cH)^n$,  where $B(\cH)$ is the algebra of all bounded linear
operators on a Hilbert space $\cH$. Initiated in
\cite{Po-holomorphic}, the theory of free holomorphic (resp.
pluriharmonic) functions on the unit ball of $B(\cH)^n$,  with
operator-valued coefficients,  has been developed   very recently
(see \cite{Po-free-hol-interp}, \cite{Po-pluri-maj},
\cite{Po-hyperbolic2}, \cite{Po-pluriharmonic},
\cite{Po-automorphism}, \cite{Po-unitary}, and \cite{Po-hyperbolic})
in the attempt  to provide a framework for the study of arbitrary
 $n$-tuples of operators on a Hilbert space.
  Several classical
 results from complex analysis, hyperbolic geometry,  and interpolation theory have
 free analogues in
 this  noncommutative multivariable setting. Related to our work, we
 mention  the papers \cite{HKMS}, \cite{K}, \cite{MuSo2},
 \cite{MuSo3}, and  \cite{V}, where several aspects of the theory
 of noncommutative analytic functions are considered in various
 settings.

To put our work in perspective, we need   to set up  some notation
and recall some definitions. Let $\FF_n^+$ be the unital free
semigroup on $n$ generators $g_1,\ldots, g_n$ and the identity
$g_0$.  The length of $\alpha\in \FF_n^+$ is defined by
$|\alpha|:=0$ if $\alpha=g_0$ and $|\alpha|:=k$ if
 $\alpha=g_{i_1}\cdots g_{i_k}$, where $i_1,\ldots, i_k\in \{1,\ldots, n\}$.
If $(X_1,\ldots, X_n)\in B(\cH)^n$,    we set $X_\alpha:=
X_{i_1}\cdots X_{i_k}$  and $X_{g_0}:=I_\cH$, the identity on $\cH$.

 We defined the algebra $H_{{\bf ball}_\gamma}$~ of
free holomorphic functions on the open operatorial  $n$-ball of
radius $\gamma>0$,
$$
[B(\cH)^n]_\gamma:=\left\{ (X_1,\ldots, X_n)\in B(\cH)^n: \
\|X_1X_n^*+\cdots + X_nX_n^*\|^{1/2}<\gamma\right\},
$$
 as the set of all power series $\sum_{\alpha\in
\FF_n^+}a_\alpha Z_\alpha$ with radius of convergence $\geq \gamma$,
i.e.,
 $\{a_\alpha\}_{\alpha\in \FF_n^+}$ are complex numbers  with
$\limsup_{k\to\infty} (\sum\limits_{|\alpha|=k}
|a_\alpha|^2)^{1/2k}\leq \frac{1}{\gamma}. $
 A free holomorphic function on
$ [B(\cH)^n]_\gamma $
 is
the representation of an element $F\in H_{{\bf ball}_\gamma}$~ on
the Hilbert space $\cH$, that is,  the mapping
$$
[B(\cH)^n]_\gamma\ni (X_1,\ldots, X_n)\mapsto F(X_1,\ldots,
X_n):=\sum_{k=0}^\infty \sum_{|\alpha|=k}
 a_\alpha X_\alpha\in
B(\cH),
$$
where  the convergence is in the operator norm topology.
   Due to the fact that a free holomorphic function is
uniquely determined by its representation on an infinite dimensional
Hilbert space, we  identify, throughout this paper,  a free
holomorphic function   with its representation on a   separable
infinite dimensional Hilbert space.

We recall that a free holomorphic function $F$ on $[B(\cH)^n]_1$ is
bounded if $
 \|F\|_\infty:=\sup  \|F(X)\|<\infty,
  $
where the supremum is taken over all $X\in [B(\cH)^n]_1$ and $\cH$
is an infinite dimensional Hilbert space. Let $H^\infty_{\bf ball}$
be the set of all bounded free holomorphic functions and  let
$A_{\bf ball}$ be the set of all elements $F$ such that the mapping
$$[B(\cH)^n]_1\ni (X_1,\ldots, X_n)\mapsto F(X_1,\ldots, X_n)\in B(\cH)$$
 has a continuous extension to the closed unit ball $[B(\cH)^n]^-_1$.
We  showed in \cite{Po-holomorphic} that $H^\infty_{\bf ball}$  and
$A_{\bf ball}$ are Banach algebras under pointwise multiplication
and the norm $\|\cdot \|_\infty$, which can be identified with the
noncommutative analytic Toeplitz algebra $F_n^\infty$ and the
noncommutative disc algebra $\cA_n$, respectively.

In Section 1,  we present new results concerning the composition of
free holomorphic functions with operator-valued coefficients and the
behavior of their model boundary functions, which will play an
important role throughout this paper.

Fractional maps of the operatorial unit ball $[B(\cE,\cG)]_1^-$ are
 due to Siegel \cite{Si} and Phillips \cite{Ph} (see also \cite{Y}).
 We should mention that the noncommutative ball $[B(\cH)^n]_1$ can be
identified with the open unit ball of $B(\cH^n,\cH)$, which is one
of the infinite-dimensional Cartan domains studied  by L. Harris
(\cite{Ha1}). He has obtained several results, related to our topic,
in the   setting of $JB^*$-algebras. We also remark that the group
of all free holomorphic automorphisms of $[B(\cH)^n]_1$
(\cite{Po-automorphism}), can be identified with  a subgroup of the
group of automorphisms of $[B(\cH^n,\cH)]_1$ considered by
R.S.~Phillips  \cite{Ph} (see also \cite{Y}).

 Following these
  ideas, fractional
transforms of free holomorphic functions were recently considered in
\cite{Po-hyperbolic2}, \cite{Po-hyperbolic}, and \cite{HKMS}. In
Section 1, we continue to investigate these transforms   and work
out several of their properties. A fractional transform $\Psi_A$ is
associated with each strict contraction $A=I\otimes A_0$, $A_0\in
B(\cE,\cG)$. We show that $ \Psi_A:\cS_{\text{\bf
ball}}(B(\cE,\cG))\to \cS_{\text{\bf ball}}(B(\cE,\cG)) $ defined by
$$
  \Psi_A[F]:=A-D_{A^*}(I-FA^*)^{-1} F D_{A}
  $$
 is a homeomorphism of
 the noncommutative Schur class $\cS_{\text{\bf ball}}(B(\cE,\cG))$
  of all free holomorphic functions
 $F$ on $[B(\cH)^n]_1$ with coefficients in $B(\cE,\cG)$ such that
 $\|F\|_\infty\leq 1$.
 Among other properties, we prove that $ F$ is inner if and only if
  its fractional transform
 ${\Psi_A[F]}$ is
inner, and that the  model boundary  function $\widetilde F$ is in
$\cA_n\bar\otimes_{min} B(\cE,\cG)$ if and only if
$\widetilde{\Psi_A[F]}$  is in $\cA_n \bar\otimes_{min} B(\cE,\cG)$.

We   mention that the noncommutative Schur class $\cS_{\text{\bf
ball}}(B(\cE,\cG))$ was introduced in \cite{Po-von} in connection
with a noncommutative von Neumann inequality for row contractions.
This class  was extended to  more general  settings  by
Ball-Groenewald-Malakorn (see \cite{BGM1}, \cite{BGM2}), and by
Muhly-Solel (see \cite{MuSo1},  \cite{MuSo2}, \cite{MuSo3}). The
Muhly-Solel paper \cite{MuSo3} gives an intrinsic characterization
for the Schur class $\cS_{\text{\bf ball}}(B(\cE,\cG))$ in terms of
completely positive kernels, and  presents a description of the
automorphism group of their Hardy algebra $H^\infty(E)$, which has
some overlap with \cite{Po-automorphism} and Theorem 1.3 of the
present paper.

Using fractional transforms and  a noncommutative version of
Schwarz's lemma \cite{Po-holomorphic}, we prove a maximum principle
for free holomorphic functions with operator-valued coefficients
(see \cite{Po-automorphism} for  the scalar case). On the other
hand, using  fractional transforms, the noncommutative Cayley
transforms of \cite{Po-free-hol-interp}, and
\cite{Po-pluriharmonic}, we obtain  results concerning the geometric
structure of bounded free holomorphic functions. More precisely, we
prove that a map $F:[B(\cH)^n]_{1}\to
  B( \cH)\bar\otimes_{min} B(\cE)$  is  a bounded
   free
holomorphic function  such that  $\|F\|_\infty\leq 1$ and
$\|F(0)\|<1$, if and only if  there exist a strict contraction
$A_0\in B(\cE)$, an $n$-tuple of isometries $(V_1,\ldots, V_n)$ on a
Hilbert space $\cK$,
 with orthogonal ranges,
  and  an isometry  $W:\cE\to \cK$, such that
$$F=(\Psi_{I\otimes A_0}\circ \text{\boldmath{$\cC$}})(G),
$$ where $\text{\boldmath{$\cC$}}$ is the noncommutative Cayley transform and   $G$ is defined by
$$
G(X_1,\ldots, X_n):=  ( I\otimes W^*) \left[ 2(I- X_1\otimes
V_1^*-\cdots - X_n\otimes V_n^* )^{-1}- I\right]( I\otimes W)
$$
for any $(X_1,\ldots, X_n)\in [B(\cH)^n]_1$. In particular, in the
scalar case, we obtain a characterization and parametrization of all
bounded free holomorphic functions on the unit ball $[B(\cH)^n]_1$.
We mention that, for the noncommutative polydisc, a representation
theorem of the same flavor was obtained in \cite{K} and \cite{AK}.

In Section 2, we provide a Vitali type convergence theorem \cite{Hi}
for uniformly bounded sequences of free holomorphic functions  with
operator-valued coefficients. As a consequence, we show that two
 free holomorphic functions $F$, $G$    coincide  if and only if  there exists a sequence
$\{A^{(k)}\}_{k=1}^\infty\subset [B(\cH)^n]_1$   of bounded-bellow
operators   such that $\lim_{k\to\infty}\|A^{(k)}\|=0$ and
$F(A^{(k)})=G(A^{(k)})$ for any $k=1,2,\ldots$.

In Section 3, we introduce the class of  free holomorphic functions
with the {\it radial infimum property}. A function $F$ is in this
class if
$$
\liminf_{r\to 1} \inf_{\|x\|=1}\|F(rS_1,\ldots,
rS_n)x\|=\|F\|_\infty,
$$
where $S_1,\ldots, S_n$ are the left creation operators on the full
Fock space $F^2(H_n)$ with $n$ generators. We obtain several
characterizations for  this class of functions and consider several
examples. We show that if $F$ is inner and its boundary function
$\widetilde F$ is in the noncommutative disc algebra $\cA_n$ then
$F$ has the radial infimum property. In particular, any free
holomorphic automorphism of $[B(\cH)^n]_1$  has the property. We
study the radial infimum property in connection with products,
direct sums, and compositions of free holomorphic functions. We also
show that the class of functions with the radial infimum property is
invariant under the fractional transforms of Section 1. These
results  are important in the following  sections.

It is well-known that if $f\in H^\infty(\DD)$, a bounded analytic
function on the open unit disc $\DD:=\{z\in \CC:\ |z|<1\}$,  is such
that $\|f\|_\infty\leq 1$ and
$$f(z)=\theta(z)
g(z), \qquad z\in \DD, $$
 where $\theta$ is an inner function in the
disc algebra $A(\DD)$ and $g$ is analytic in $\DD$, then
$\|g\|_\infty\leq 1$. If, in addition,  $f\in A(\DD)$, then $g\in
A(\DD)$.  Moreover, if $f\in A(\DD)$ is inner, then so is $g$. These
facts are fundamental for the theory of bounded analytic functions
(see \cite{Cara2}, \cite{Ga}).

 In Section 4, we obtain analogues of these results in the context
 of free holomorphic functions. Let $F, \Theta$, and $ G$ be free holomorphic functions
on $[B(\cH)^n]_1$   such that
$$F(X)=\Theta (X) G(X),\qquad X\in [B(\cH)^n]_1.
$$
Assume that   $F$ is bounded with $\|F\|_\infty\leq 1$ and $\Theta$
  has the radial infimum property with $\|\Theta\|_\infty=1$. Then we prove that
   $\|G\|_\infty\leq 1$ and
\begin{equation*}
 F(X)F(X)^*\leq \Theta(X)\Theta(X)^*, \qquad X\in [B(\cH)^n]_1.
\end{equation*}
 Moreover, we show that if  the boundary functions $\widetilde F$ and $\widetilde \Theta$
are in the noncommutative disc algebra, then so is $\widetilde G$.
When we add the condition that $F$ is inner, then  we deduce that
$G$ is  also inner.
In particular,   if $F $ is a bounded
   free
holomorphic function with \ $\|F\|_\infty\leq 1$ and  representation
$$F(X)=\sum_{k=m}^\infty \sum_{|\alpha|=k}
X_\alpha\otimes A_{(\alpha)},\qquad X=(X_1,\ldots, X_n)\in
[B(\cH)^n]_1,
$$
for some $m=1,2,\ldots,$ and $A_{(\alpha)}\in B(\cE, \cG)$,
 then
$$ F(X) F(X)^*\leq \sum_{|\beta|=m} X_\beta X_\beta^*\otimes I_\cG,\qquad X\in[B(\cH)^n]_1.
$$
 Consequently, we recover the corresponding   version
 of Schwarz's lemma  from
\cite{Po-holomorphic} and, when $m=1$, the one  from \cite{HKMS}.

The classical Schwarz's lemma  (see
 \cite{Co}, \cite{Ru1})
  states that if $f:\DD\to \CC$ is  a
bounded analytic function with $f(0)=0$ and $|f(z)|\leq 1$  for
$z\in \DD$, then $|f'(0)|\leq 1$ and $|f(z)|\leq |z|$ for $z\in
\DD$. Moreover, if $|f'(0)|= 1$ or if $|f(z)|=|z|$ for some $z\neq
0$, then there is a constant $c$ with $|c|=1$ such that $f(w)=cw$
for any $w\in \DD$. A faithful generalization of this result is
obtained (see Theorem \ref{Schwarz2}) when $f, \theta$, and $ g$ are
free holomorphic functions on $[B(\cH)^n]_1$ with scalar
coefficients such that:
\begin{enumerate}
\item[(i)] $f(X)=\theta (X) g(X),\quad X\in [B(\cH)^n]_1$;
 \item[(ii)] $f$ is
bounded with $\|f\|_\infty\leq 1$;
 \item[(iii)] $\theta$  has  the
radial infimum property and $\|\theta\|_\infty=1$.
\end{enumerate}
In the particular case when $n=1$ and
 $\theta(z)=z$, we recover the  Schwarz's lemma. We remark that
Schwarz's lemma has been extended to various   settings  by several
authors (e.g. \cite{J},  \cite{Si},  \cite{Pota}, \cite{Ph},
\cite{Ha}, \cite{KF1}, \cite{MuSo2}, \cite{Po-unitary},
\cite{Po-automorphism}).

In Section 4, we also obtain   noncommutative extensions of Harnack'
double inequality (see Theorem \ref{Har2}) for  a class of free
holomorphic functions $F=I+\Theta \Gamma$ with positive real parts.
 In  the particular case when  $\Theta(X)=X$, we deduce  that if $F$
 is a free holomorphic function
on $[B(\cH)^n]_1$ with coefficients in $B(\cE)$ such that $F(0)=I$
and $\Re F \geq 0$, then
$$
\frac{1-\|X\|}{1+\|X\|}\leq \|F(X)\|\leq
\frac{1+\|X\|}{1-\|X\|},\qquad X\in [B(\cH)^n]_1.
$$

The Borel-Carath\' eodory theorem \cite{Tit} establishes an upper
bound for the modulus of a function on the circle $|z|=r$ from
bounds for its real (or imaginary) parts on larger circles $|z|=R$.
More precisely,  if $f$ is an analytic function for $|z|\leq R$ and
$0<r<R$, then
$$
\sup_{|z|=r}|f(z)|\leq \frac{2r}{R-r}\sup_{|z|=R} \Re f(z)+
\frac{R+r}{R-r}|f(0)|.
$$
In Section 5, we obtain an analogue of this result for free
holomorphic functions (see Theorem \ref{borel-cara}). We also obtain
a Borel-Carath\'eodory type result for free holomorhic functions
which admit  factorizations  $F=\Theta \Gamma$, where  $\Theta$ is
an inner function with the radial infimum property and
$\|\Theta(0)\|<1$. We show that if $\Re F\leq I$
     then
    $$
    \|F(X)\|\leq \frac{2\|\Theta(X)\|}{1- \|\Theta(X)\|}, \qquad
     \ X\in [B(\cH)^n]_1.
    $$

Let $f:\DD\to \overline{\DD}$ be a nonconstant analytic function and
let $z_0\in \DD$ and $w_0=f(z_0)$. Pick's theorem \cite{Pic} (see
also \cite{Cara2}) asserts that
$$
\frac{w_0-f(z)}{1-\bar w_0 f(z)}=\frac{z_0-z}{1-\bar z_0 z}
g(z),\qquad z\in\DD,
$$
for some analytic function $g:\DD\to \DD$. In Section  6,   we
provide a generalization  of Pick's theorem, for bounded free
holomorphic functions. We show that
 if  $F:[B(\cH)^n]_1\to [B(\cH)^m]_1^-$ is a free holomorphic
function  with $\|F(0)\|<1$ and   $a\in \BB_n$, then there exists a
free holomorphic function $\Gamma$ with $\|\Gamma\|_\infty \leq 1$
such that
$$
\Phi_{F(a)}(F(X))=\Phi_a(X) (\Gamma \circ\Phi_a) (X),\qquad X\in
[B(\cH)^n]_1,
$$
where $\Phi_a$ and $\Phi_{F(a)}$   are the corresponding free
holomorphic automorphisms of  the noncommutative balls
$[B(\cH)^n]_1$ and $[B(\cH)^m]_1$, respectively. Consequently,
$$
 \Phi_{F(a)}(F(X))\Phi_{F(a)}(F(X))^*\leq  \Phi_a(X)\Phi_a(X)^*,\qquad X\in
[B(\cH)^n]_1.
$$
 We mention that the group $Aut([B(\cH)^n]_1)$ of all free holomorphic automorphisms
  of $[B(\cH)^n]_1$ was determined in \cite{Po-automorphism}, using
  the theory of characteristic functions for row contractions \cite{Po-charact}.
  We also
remark that the group of all free holomorphic automorphisms of
$[B(\cH)^n]_1$, can be identified with  a subgroup of the group of
automorphisms of $[B(\cH^n,\cH)]_1$ considered by R.S.~Phillips
\cite{Ph} (see also \cite{Y}).

We recall that  Julia's lemma  \cite{J} (see also \cite{Cara1})
says that if $f:\DD\to \DD$ is an
 analytic function and there is a sequence $\{z_k\}\subset \DD$
 with $z_k\to 1$,  $f(z_k)\to 1$, and such that
 $\frac{1-|f(z_k)|}{1-|z_k|}$ is bounded, then $f$ maps each disc in
 $\DD$ tangent to $\partial \DD$ at   $1$  into a disc of the same
 kind. Julia's lemma has been extended to analytic functions of a
 single operator variable by Fan \cite{KF2} and   to the setting of
 function algebras  by Glicksberg \cite{Gl}.

Using the above-mentioned noncommutative analogue of Pick's theorem
and basic facts concerning  the involutive  free holomorphic
automorphisms of $[B(\cH)^n]_1$, we obtain
 a free analogue of Julia's lemma (see Theorem \ref{Julia1}).
In particular, we prove the following result.

 Let
$F:[B(\cH)^n]_1\to [B(\cH)^m]_1 $ be a free holomorphic function.
Let $z_k\in \BB_n$ be such that $\lim_{k\to \infty} z_k
=(1,0,\ldots, 0)\in
\partial\BB_n$, $\lim_{k\to \infty} F(z_k) =(1,0,\ldots, 0)\in
\partial\BB_m$, and
$$
\lim_{k\to \infty} \frac{1-\|F(z_k)\|^2}{1-\|z_k\|^2}= L <\infty. $$
 If
$F:=(F_1,\ldots, F_m)$, then $L>0$ and
$$
(I-F_1(X)^*)(I-F(X)F(X)^*)^{-1} (I-F_1(X))\leq L
(I-X_1^*)(I-XX^*)^{-1}(I-X_1)
$$
for any $X=(X_1,\ldots, X_n)\in [B(\cH)^n]_1$.
   Moreover, if $0<c<1$, then
$$ F({\bf E}_c) \subset {\bf E}_\gamma,
\quad \text{where }\ \gamma :=\frac{Lc}{1+Lc-c}
$$
and ${\bf E}_c$  and ${\bf E}_\gamma$ are certain noncommutative
ellipsoids.  A similar result holds  if we  replace   the ellipsoids
with some  noncommutative Korany type regions (\cite{Ru2})  in the
unit ball $[B(\cH)^n]_1$ (see Corollary \ref{Korany}).

In Section 7, we use fractional transforms  and a version of the
noncommutative Schwarz's lemma to obtain Pick-Julia theorems for
free holomorphic functions  $F$ with operator-valued coefficients
such that  $\|F\|_\infty\leq 1$ (resp. $\Re F \geq 0$) (see Theorem
\ref{Pi-Ju} and Theorem \ref{Pi-Ju2}). As  a consequence, we obtain
 a Julia type lemma for free holomorphic
 functions with positive real parts (see Theorem \ref{Julia2}).
  We also provide commutative versions
  of these results for
operator-valued multipliers of the Drury-Arveson space (see
Corollary \ref{multiplier}). When $n=1$, we recover (with different
proofs) the   corresponding results obtained by Potapov \cite{Pota}
and Ando-Fan \cite{AF}.

In Section 8,  we provide a    noncommutative extension of a
classical result due to Lindel\" of (see \cite{Ga}, \cite{Kr}). We
prove that if $F:[B(\cH)^n]_1\to [B(\cH)^m]_1^-$ is   a free
holomorphic function, then
\begin{equation*}
 \|F(X)\|\leq \frac{\|X\|+\|F(0)\|}{1+\|X\|\|F(0)\|},\qquad X\in
[B(\cH)^n]_1.
\end{equation*}
If, in addition, the boundary function of $F$ has its entries in the
noncommutative disc algebra $\cA_n$, then the inequality above holds
for any $ X\in [B(\cH)^n]_1^-$. We remark that if $\|F(0)\|<1$, then
the inequality  above is sharper than the noncommutative von Neumann
inequality  (see \cite{Po-von}, \cite{Po-funct}).

In Section 9, we introduce a pseudohyperbolic metric  ${\bf d}$  on
$[B(\cH)^n]_1$ which  is  invariant under the action of  the free
holomorphic automorphism group of $[B(\cH)^n]_1$ and turns out to be
a noncommutative extension of the pseudohyperbolic distance  (see
\cite{Zhu}) on $\BB_n$,  the open unit ball of $\CC^n$, i.e.,
$$
d_n(z,w):= \|\psi_z(w)\|_2,\qquad z,w\in \BB_n,
$$
where $\psi_z$ is the involutive automorphism of $\BB_n$ that
interchanges $0$ and $z$.  We show that
 $$
 {\bf d}(X,Y)=\tanh
\delta(X,Y),\qquad X,Y\in [B(\cH)^n]_1,
$$
where  $ \delta$ is the hyperbolic ({\it Poincar\'e-Bergman}
\cite{Be} type) metric on $[B(\cH)^n]_1$ introduced and studied  in
\cite{Po-hyperbolic}. As a consequence, we obtain a Schwarz-Pick
lemma for free holomorphic functions on $[B(\cH)^n]_1$ with
operator-valued coefficients, with respect to the pseudohyperbolic
metric. More precisely, if $F=(F_1,\ldots, F_m)$ and $F_j$ are free
holomorphic functions with operator-valued coefficients such that
$\|F\|_\infty\leq 1$, then
$$
{\bf d}(F(X), F(Y))\leq {\bf d}(X,Y), \qquad X,Y\in [B(\cH)^n]_1.
$$

It is well-known (see \cite{Po-poisson}, \cite{Po-holomorphic},
\cite{Po-pluriharmonic}) that if $F$ is a contractive
($\|F\|_\infty\leq 1$) free holomorphic function  with coefficients
in $B(\cE)$, then  the evaluation map $\BB_n\ni z\mapsto F(z)\in
B(\cE )$ is a contractive operator-valued multiplier of the
Drury-Arveson space (\cite{Dr}, \cite{Arv}). Moreover, any such a
contractive multiplier has  this kind of representation. Due to this
reason, several results of the present paper have commutative
versions for operator-valued multipliers of the Drury-Arveson space.

It would be interesting to see if the results of this paper can be
extended to more general infinite-dimensional bounded domains such
as the $JB^*$-algebras of Harris \cite{Ha1},   or the noncommutative
domains from \cite{Po-domains} and \cite{HKMS}.  Since our results
are based on the power series representation of free holomorphic
functions we are inclined to believe in a positive answer for the
domains considered in  \cite{Po-domains} and \cite{HKMS}.

\bigskip

\section{ Free holomorphic functions: fractional transforms, maximum principle,  and geometric  structure
   }

In this section,  we present  results concerning the composition
 and  fractional transforms  of free
holomorphic functions, and the behavior of their model boundary
functions. These results are used to prove a maximum principle and a
Naimark type representation theorem for free holomorphic functions
with operator-valued coefficients.

Let $H_n$ be an $n$-dimensional complex  Hilbert space with
orthonormal
      basis
      $e_1$, $e_2$, $\dots,e_n$, where $n=1,2,\dots$, or $n=\infty$.
       We consider the full Fock space  of $H_n$ defined by
      $$F^2(H_n):=\CC1\oplus \bigoplus_{k\geq 1} H_n^{\otimes k},$$
      where  $H_n^{\otimes k}$ is the (Hilbert)
      tensor product of $k$ copies of $H_n$.
      Define the left  (resp.~right) creation
      operators  $S_i$ (resp.~$R_i$), $i=1,\ldots,n$, acting on $F^2(H_n)$  by
      setting
      $$
       S_i\varphi:=e_i\otimes\varphi, \quad  \varphi\in F^2(H_n),
      $$
       (resp.~$
       R_i\varphi:=\varphi\otimes e_i, \quad  \varphi\in F^2(H_n)
      $).
The noncommutative disc algebra $\cA_n$ (resp.~$\cR_n$) is the norm
closed algebra generated by the left (resp.~right) creation
operators and the identity. The   noncommutative analytic Toeplitz
algebra $F_n^\infty$ (resp.~$\cR_n^\infty$)
 is the  weakly
closed version of $\cA_n$ (resp.~$\cR_n$). These algebras were
introduced in \cite{Po-von} in connection with a noncommutative von
Neumann  type inequality \cite{von},   and  have been intensively
studied in recent years (see \cite{Po-funct}, \cite{Po-analytic},
\cite{Po-poisson}, \cite{Po-curvature}, \cite{D}, \cite{Po-unitary},
\cite{MuSo1}, and the references therein).

 We denote $e_\alpha:=
e_{i_1}\otimes\cdots \otimes  e_{i_k}$  if $\alpha=g_{i_1}\cdots
g_{i_k}\in \FF_n^+$ and $e_{g_0}:=1$. Note that
$\{e_\alpha\}_{\alpha\in \FF_n^+}$ is an orthonormal basis for
$F^2(H_n)$. Let $C^*(S_1,\ldots, S_n)$ be the Cuntz-Toeplitz
$C^*$-algebra generated by the left creation operators (see
\cite{Cu}). The noncommutative Poisson transform at
 $T:=(T_1,\ldots, T_n)\in [B(\cH)^n]_1^-$ is the unital completely contractive  linear map
 $P_T:C^*(S_1,\ldots, S_n)\to B(\cH)$ defined by
 \begin{equation*}
 P_T[f]:=\lim_{r\to 1} K_{T,r}^* (I_\cH \otimes f)K_{T,r}, \qquad f\in C^*(S_1,\ldots,
 S_n),
\end{equation*}
 where the limit exists in the norm topology of $B(\cH)$.
Here, the noncommutative Poisson  kernel
$$
K_{T,r} :\cH\to  \overline{\Delta_{T,r}\cH} \otimes  F^2(H_n),
\qquad 0< r\leq 1,
$$
is defined by
\begin{equation*}
K_{T,r}h:= \sum_{k=0}^\infty \sum_{|\alpha|=k} r^{|\alpha|}
\Delta_{T,r} T_\alpha^*h\otimes  e_\alpha,\qquad h\in \cH,
\end{equation*}
where   $\Delta_{T,r}:=(I_\cH-r^2T_1T_1^*-\cdots -r^2
T_nT_n^*)^{1/2}$ and $\Delta_T:=\Delta_{T,1}$. We recall that
 $$
 P_T[S_\alpha S_\beta^*]=T_\alpha T_\beta^*, \qquad \alpha,\beta\in \FF_n^+.
 $$
 When $T:=(T_1,\ldots, T_n)$  is a pure row contraction, i.e.,
 $ \text{\rm SOT-}\lim\limits_{k\to\infty} \sum_{|\alpha|=k}
T_\alpha T_\alpha^*=0$, then
   we have $$P_T[f]=K_T^*(I_{\cD_{T}}\otimes f)K_T, \qquad f\in C^*(S_1,\ldots,
 S_n)\ \text{ or } \ f\in F_n^\infty,
   $$
   where $\cD_T:=\overline{\Delta_T \cH}$.
We refer to \cite{Po-poisson}, \cite{Po-curvature},  and
\cite{Po-unitary} for more on noncommutative Poisson transforms on
$C^*$-algebras generated by isometries.

Let $\cE,\cG$ be Hilbert spaces and let $B(\cE,\cG)$ be the set of
all bounded linear operators  from $\cE$ to $\cG$.
  A map $F:[B(\cH)^n]_{\gamma}\to
  B( \cH)\bar\otimes_{min} B(\cE, \cG)$ is a
  {\it free
holomorphic function} on  $[B(\cH)^n]_{\gamma}$  with coefficients
in $B(\cE, \cG)$ if there exist $A_{(\alpha)}\in B(\cE, \cG)$,
$\alpha\in \FF_n^+$, such that
$$
F(X_1,\ldots, X_n)=\sum\limits_{k=0}^\infty \sum\limits_{|\alpha|=k}
X_\alpha\otimes  A_{(\alpha)},
$$
where the series converges in the operator  norm topology  for any
$(X_1,\ldots, X_n)\in [B(\cH)^n]_{\gamma}$. According to
\cite{Po-holomorphic}, a power series $F:=\sum\limits_{\alpha\in
\FF_n^+}
 Z_\alpha\otimes A_{(\alpha)}$ represents a free holomorphic function
on the open  operatorial $n$-ball of radius $\gamma$, with
coefficients in $B(\cE, \cG)$,      if and only if  \ $
\limsup\limits_{k\to\infty} \left\|\sum_{|\alpha|=k} A_{(\alpha)}^*
A_{(\alpha)}\right\|^{\frac{1} {2k}}\leq \frac{1}{\gamma}.$ This is
also equivalent to the fact that
 the series
$$
\sum\limits_{k=0}^\infty  \sum\limits_{|\alpha|=k} r^{|\alpha|}
 S_\alpha\otimes A_{(\alpha)}
$$
is convergent in the operator norm topology for any $r\in
[0,\gamma)$, where $S_1,\ldots, S_n$ are the left creation operators
on the Fock space $F^2(H_n)$.
We denote by $H_{\text{\bf ball}}(B(\cE,\cG))$ the set of all free
holomorphic functions on the noncommutative ball $[B(\cH)^n]_1$ and
coefficients in $B(\cE,\cG)$. Let $H_{\text{\bf ball}}^\infty
(B(\cE,\cG))$ denote the set of all elements $F$ in  $H_{\text{\bf
ball}}(B(\cE,\cG))$  such that
$$
\|F\|_\infty:=\sup  \|F(X_1,\ldots, X_n)\|<\infty,
$$
where the supremum is taken over all $n$-tuples  of operators
$(X_1,\ldots, X_n)\in [B(\cH)^n]_1$, where $\cH$ is an infinite
dimensional   Hilbert space.

According to \cite{Po-holomorphic} and \cite{Po-pluriharmonic}, the
noncommutative Hardy space
  $H_{\text{\bf ball}}^\infty (B(\cE,\cG))$   can be identified to the operator
  space
$  F_n^\infty\bar\otimes B(\cE,\cG)$ (the weakly closed operator
space generated by the spatial tensor product), where $F_n^\infty$
is the noncommutative analytic Toeplitz algebra. More precisely, a
bounded free holomorphic function $F$ is uniquely determined by its
{\it (model) boundary function} $\widetilde F(S_1,\ldots, S_n)\in
F_n^\infty\bar \otimes B(\cE, \cG)$ defined by
$$\widetilde F=\widetilde F(S_1,\ldots, S_n):=\text{\rm SOT-}\lim_{r\to 1}
F(rS_1,\ldots, rS_n). $$ Moreover, $F$ is  the noncommutative
Poisson transform of $\widetilde F(S_1,\ldots, S_n)$ at
$X:=(X_1,\ldots, X_n)\in [B(\cH)^n]_1$, i.e.,
$$
F(X_1,\ldots, X_n)=(P_X\otimes I)[\widetilde F(S_1,\ldots, S_n)].
$$
Similar results hold for bounded free holomorphic functions on the
noncommutative ball  $[B(\cH)^n]_\gamma$, $\gamma>0$.

We recall from \cite{Po-automorphism} some   facts concerning the
composition of free holomorphic functions with operator-valued
coefficients. Let $\Phi:[B(\cH)^n]_{\gamma_1}\to [B(\cH)
\bar\otimes_{min} B(\cY)]^m$ be a free holomorphic function with
$\Phi(X):=(\Phi_1(X),\ldots, \Phi_m(X))$, where
$\Phi_j:[B(\cH)^n]_{\gamma_1}\to B(\cH)\bar \otimes_{min} B(\cY)$,
$j=1,\ldots, m$,  are free holomorphic functions with standard
representations
$$
\Phi_j(X)=\sum_{k=0}^\infty \sum_{\alpha\in \FF_n^+,|\alpha|=k}
X_\alpha\otimes B_{(\alpha)}^{(j)}, \qquad X:=(X_1,\ldots, X_n)\in
[B(\cH)^n]_{\gamma_2},
$$
for some $B_{(\alpha)}^{(j)}\in B(\cY)$, where  $\alpha\in \FF_n^+$,
$j=1,\ldots,m$. Assume that $$ \|\Phi(X)\|<\gamma_2 \quad \text{ for
any } \  X\in
 [B(\cH)^n]_{\gamma_1}.
$$
This is equivalent to
$$ \|\Phi(rS_1,\ldots, rS_n)\|<\gamma_2
\quad \text{ for any } \ r\in[0,\gamma_1).
$$
Let $F:[B(\cK)^m]_{\gamma_2}\to B(\cK)\bar \otimes_{min} B(\cE,\cG)$
be a free holomorphic function with standard representation
$$
F(Y_1,\ldots, Y_m):= \sum_{k=0}^\infty \sum_{\alpha\in
\FF_m^+,|\alpha|=k} Y_\alpha\otimes A_{(\alpha)}, \qquad
(Y_1,\ldots, Y_m)\in [B(\cK)^m]_{\gamma_2},
$$
for some  operators $A_{(\alpha)}\in B(\cE,\cG)$, $\alpha\in
\FF_m^+$. Then it makes sense to define the map $F\circ
\Phi:[B(\cH)^n]_{\gamma_1}\to B(\cH)\bar \otimes_{min}
B(\cY)\bar\otimes_{min} B(\cE,\cG)$ by setting
$$
(F\circ \Phi)(X_1,\ldots, X_n):=\sum_{k=0}^\infty \sum_{\alpha\in
\FF_m^+,|\alpha|=k} \Phi_\alpha(X_1,\ldots, X_n)\otimes
A_{(\alpha)}, \qquad (X_1,\ldots, X_n)\in [B(\cH)^n]_{\gamma_1},
$$
where  the convergence is in the operator norm topology. We proved
in \cite{Po-automorphism} that  $F\circ \Phi$ is a free holomorphic
function on $[B(\cH)^n]_1$  with standard representation
$$
(F\circ \Phi)(X_1,\ldots, X_n)=\sum_{k=0}^\infty \sum_{\sigma\in
\FF_n^+,|\sigma|=k} X_\sigma\otimes C_{(\sigma)},
$$
where
$$
\left< C_{(\sigma)}x,y\right>= \frac{1}{r^{|\sigma|}}\left<(
S_\sigma^*\otimes I_{\cY\otimes \cG})  (F\circ \Phi)(rS_1,\ldots,
rS_n)(1\otimes x), 1\otimes y\right>
$$
for any $\sigma\in \FF_n^+$, $x\in \cY\otimes \cE$, and $y\in
\cY\otimes \cG$. Actually, this is a slight extension of the
corresponding result  from  \cite{Po-automorphism}. However,  the
proof is basically the same.

For simplicity, throughout this paper, $[X_1,\ldots, X_n]$ denotes
either the $n$-tuple $(X_1,\ldots, X_n)\in B(\cH)^n$ or the operator
row matrix $[ X_1\, \cdots \, X_n]$ acting from $\cH^{(n)}$, the
direct sum  of $n$ copies of  a Hilbert space $\cH$, to $\cH$.

Now, we present new results concerning the composition of bounded
free holomorphic functions with operator-valued coefficients.
\begin{theorem}
\label{more-prop} Let  $F:[B(\cK)^m]_{\gamma_2}\to B(\cK)\bar
\otimes_{min} B(\cE,\cG)$  and $\Phi:[B(\cH)^n]_{\gamma_1}\to
[B(\cH)\bar \otimes_{min} B(\cY)]^m$  be   bounded free holomorphic
functions  such that
$$ \|\Phi(X)\|<\gamma_2 \quad \text{ for
any } \  X\in
 [B(\cH)^n]_{\gamma_1}.
$$
Then the boundary function of     the   bounded free holomorphic
function $F\circ \Phi$ satisfies the equation
$$
\widetilde {F\circ\Phi}=\text{\rm SOT-}\lim_{r\to 1} F(r\widetilde
\Phi_1,\ldots, r\widetilde \Phi_m).
$$
Moreover, if $\widetilde F\in \cA_n\bar \otimes_{min} B(\cE,\cG)$
and $\widetilde \Phi:=[\widetilde \Phi_1,\ldots, \widetilde \Phi_m]$
is
  such that   $\widetilde \Phi_j \in \cA_n\bar \otimes_{min} B(\cY)$,
$j=1,\ldots, m$,  and   $\|\widetilde \Phi\|<\gamma_2$,  then
$\widetilde {F\circ\Phi}\in \cA_n\bar \otimes_{min} B(\cY \otimes
\cE,\cY \otimes \cG)$.
\end{theorem}

\begin{proof}  Using the fact that a function $X\mapsto
G(X)$ is free holomorphic on $[B(\cK)^m]_{\gamma}$, $\gamma>0$, if
and only if the mapping $Y\mapsto G(\gamma Y)$ is free holomorphic
on $[B(\cK)^m]_{1}$, we can assume, without loss of generality, that
$\gamma_1=\gamma_2=1$.

Due to \cite{Po-automorphism} (see the considerations preceding this
theorem), $F\circ \Phi$
  is a  bounded free holomorphic function.
Let $F$ have the representation
$$
F(Y_1,\ldots, Y_m):= \sum_{k=0}^\infty \sum_{\alpha\in
\FF_m^+,|\alpha|=k} Y_\alpha\otimes A_{(\alpha)}, \qquad
(Y_1,\ldots, Y_m)\in [B(\cK)^m]_{1}.
$$
 Since $F$ is
bounded on $[B(\cK)^n]_1$, we have $$\left(\sum_{\beta\in
\FF_m^+}\|A_{(\beta)}h\|^2\right)^{1/2}\leq \|F\|_\infty \|h\|,
\qquad h\in \cE.
$$
Given $\epsilon>0$ and $h\in \cE$, we choose $q\in \NN$ such that
\begin{equation}
\label{<eps} \sum_{\beta\in \FF_m^+, |\beta|\geq
q}\|A_{(\beta)}h\|^2< \epsilon^2.
 \end{equation}
 For any $x\in
F^2(H_n)$, we have
\begin{equation*}
\begin{split}
&\left\|\sum_{k=q}^\infty \sum_{\beta\in \FF_m^+, |\beta|=k
}(\Phi_\beta(rS_1,\ldots, rS_n)\otimes A_{(\beta)})(x\otimes
h)\right\|\\
&\qquad \leq \sum_{k=q}^\infty\left\| [\Phi_\beta(rS_1,\ldots,
rS_n)\otimes
I:\ |\beta|=k] \left[\begin{matrix} I\otimes A_{(\beta)}\\
:\\
|\beta|=k\end{matrix}\right] (x\otimes h)\right\|\\
&\qquad \leq \sum_{k=q}^\infty\left\|  \left[\begin{matrix} I\otimes A_{(\beta)}\\
:\\
|\beta|=k\end{matrix}\right] (x\otimes h)\right\|\\
&\leq \|x\|\sum_{k=q}^\infty\left(\sum_{\beta\in \FF_m^+, |\beta|=k
}\|A_{(\beta)}h\|^2\right)^{1/2}\leq \epsilon \|x\|
\end{split}
\end{equation*}
for any $r\in (0,1)$. Here we used the fact that
$[\Phi_1(rS_1,\ldots, rS_n),\ldots, \Phi_n(rS_1,\ldots, rS_n)]$ is a
contraction and, therefore, the operator row  matrix
$[\Phi_\beta(rS_1,\ldots, rS_n)\otimes I:\ |\beta|=k]$ is also  a
contraction.

Now denote $F_r(Y_1,\ldots, Y_m):= F(rY_1,\ldots, rY_m)$, $0<r<1$,
and note that $F_r$ is  a bounded free holomorphic function on
$[B(\cK)^n]_{1/r}$. Since the boundary function
$\widetilde\Phi:=[\widetilde\Phi_1,\ldots, \widetilde\Phi_m]$ is a
row contraction, we have
$$
F_r(\widetilde\Phi_1,\ldots, \widetilde\Phi_m)= \sum_{k=0}^\infty
\sum_{\alpha\in \FF_m^+,|\alpha|=k}
r^{|\alpha|}\widetilde\Phi_\alpha\otimes A_{(\alpha)},
$$
where the convergence is in the operator norm topology.

Using  relation \eqref{<eps} and that
$[r^{|\beta|}\widetilde\Phi_\beta\otimes I:\ |\beta|=k]$ is a row
contraction, one can show, as above,  that
\begin{equation*}
 \left\|\sum_{k=q}^\infty \sum_{\beta\in \FF_m^+, |\beta|=k
}(r^{|\beta|}\widetilde\Phi_\beta\otimes A_{(\beta)})(x\otimes
h)\right\|<\epsilon \|x\|
\end{equation*}
for any $r\in (0,1)$. On the other hand, we have
 $$
\lim_{r\to1} \sum_{\beta\in \FF_m^+, |\beta|<k
}[(\Phi_\beta(rS_1,\ldots,
rS_n)-r^{|\beta|}\widetilde\Phi_\beta)\otimes A_{(\beta)}](x\otimes
h)=0.
$$
Now, combining this  equality with the inequalities above, one can
easily deduce that
\begin{equation}
\label{lim<} \lim_{r\to1}(F\circ\Phi)(rS_1,\ldots, rS_n)(x\otimes
h)=\lim_{r\to 1} F(r\widetilde \Phi_1,\ldots, r\widetilde
\Phi_m)(x\otimes h)
\end{equation}
for any $x\in F^2(H_n)$ and $h\in \cE$. Since
$$
\|F(r\widetilde \Phi_1,\ldots, r\widetilde \Phi_m)\|\leq
\|F\|_\infty\quad \text{ and }\quad \|(F\circ\Phi)(rS_1,\ldots,
rS_n)\|\leq \|F\|_\infty,
$$
relation \eqref{lim<} implies
$$
\widetilde {F\circ\Phi}=\text{\rm SOT-}\lim_{r\to 1} F(r\widetilde
\Phi_1,\ldots, r\widetilde \Phi_m).
$$

 To prove the second part of
the theorem, assume that $\widetilde F\in \cA_n \bar\otimes_{min}
B(\cE,\cG)$,
 $\widetilde \Phi:=[\widetilde\Phi_1,\ldots, \widetilde\Phi_m]$ is
in $M_{1\times m}(\cA_n\bar \otimes_{min} B(\cY))$,  and
$\|\widetilde \Phi\|<1$.  Since $F$ is a free holomorphic function
on $[B(\cK)^n]_1$,
$$G:=\sum_{k=0}^\infty \sum_{\alpha\in \FF_m^+,|\alpha|=k} \widetilde
\Phi_\alpha\otimes A_{(\alpha)}$$
 is convergent in the  operator
norm topology.  On the other hand,  $\widetilde\Phi_\alpha\in
\cA_n\bar \otimes_{min} B(\cY)$. Consequently, $G$ is in
$\cA_n\bar \otimes_{min} B(\cY \otimes \cE, \cY\otimes \cG)$. Now,
for any $\epsilon>0$, there exists $p\in\NN$ such that
 $$
 \left\| \sum_{k=p}^\infty \sum_{\alpha\in
\FF_m^+,|\alpha|=k} \widetilde \Phi_\alpha\otimes
A_{(\alpha)}\right\|<\epsilon.
$$
 Due to the  the noncommutative von
Neumann inequality (see \cite{Po-von}), we have
$$
\left\| \sum_{k=p}^\infty\sum_{\alpha\in \FF_m^+,|\alpha|=k}
\Phi_\alpha(rS_1,\ldots, rS_n)\otimes A_{(\alpha)}\right\|\leq
\left\| \sum_{k=p}^\infty\sum_{\alpha\in \FF_m^+,|\alpha|=k}
\widetilde \Phi_\alpha\otimes A_{(\alpha)}\right\|
$$
for any $k\in \NN$. Consequently, we have
\begin{equation}\label{G}
\begin{split}
\|(F\circ \Phi)(rS_1,\ldots, rS_n)-G\|&\leq \sum_{k=0}^p \left\|
\sum_{\alpha\in \FF_m^+,|\alpha|=k} [\Phi_\alpha(rS_1,\ldots,
rS_n)-\widetilde \Phi_\alpha]\otimes A_{(\alpha)}\right\| +2\epsilon
\end{split}
\end{equation}
On the other hand, since $\widetilde\Phi_i\in \cA_n\bar
\otimes_{min} B(\cY)$, $i=1,\ldots, n$, we have
$$ \lim_{r\to 1}\Phi_\alpha(rS_1,\ldots,
rS_n)=\widetilde\Phi_\alpha, \qquad \alpha\in \FF_m^+, $$
 in the
operator norm topology.  Now,  using relation \eqref{G}, we deduce
that
$$\widetilde{F\circ \Phi}:=\lim_{r\to 1} (F\circ \Phi)(rS_1,\ldots,
rS_n)=G,
$$
where the limit is in the operator norm topology. Therefore
$\widetilde{F\circ \Phi}$ is in $\cA_n\bar \otimes_{min} B(\cY
\otimes \cE, \cY\otimes \cG)$. This completes the proof.
\end{proof}

Using  Theorem \ref{more-prop}, and Theorem 4.1 from \cite{PPoS}, we
can  prove  the following result for  bounded free holomorphic
functions with operator-valued coefficients. We recall that a
bounded free holomorphic function  is called inner (resp.~outer) if
its model boundary function  is an isometry (resp.~has dense range).

\begin{theorem}\label{inner} Let  $F:[B(\cK)^m]_{1}\to B(\cK)\bar
\otimes_{min} B(\cE,\cG)$  and $\Phi:[B(\cH)^n]_{1}\to [B(\cH)]^m$
be bounded free holomorphic functions. Assume that
$\Phi=(\Phi_1,\ldots, \Phi_m)$ is inner and $\widetilde \Phi_1$ is
non-unitary if $m=1$.
 Then
 the following statements hold:
\begin{enumerate}
\item[(a)]
$\|F\circ\Phi\|_\infty=\|F\|_\infty$;

\item[(b)] if $F$  is  inner, then $F\circ \Phi$ is
inner;
\item[(c)] if $F$ is outer, then $F\circ \Phi$
is outer.
\end{enumerate}
\end{theorem}

\begin{proof}
Let $\Phi_j: [B(\cH)^n]_{1}\to B(\cH)$, $j=1,\ldots,m$,  be free
holomorphic functions with scalar coefficients and assume that
$\Phi=[\Phi_1,\ldots, \Phi_n]$ is inner, i.e.,
$\widetilde\Phi:=[\widetilde\Phi_1,\ldots, \widetilde\Phi_n]$ is an
isometry. According to  Theorem 4.1 from \cite{PPoS},
$\widetilde\Phi$ is a pure isometry, i.e.,
$$
\text{\rm WOT-}\lim_{k\to\infty}\sum_{\omega\in \FF_m^+, |\omega|=k}
\widetilde\Phi_\omega \widetilde\Phi_\omega^*=0.
$$
Due to the noncommutative Wold-type decomposition for sequences of
isometries with orthogonal ranges \cite{Po-isometric},
$\widetilde\Phi$ is unitarily equivalent to  $[I_\cL\otimes
S_1',\ldots, I_\cL\otimes S_m']$, where $S_1',\ldots, S_m'$ are the
left creation operators on the full Fock space $F^2(H_m)$, and $\cL$
is a separable  Hilbert space. Consequently, there is a unitary
operator $U:F^2(H_n)\to \cL\otimes F^2(H_m)$ such that
\begin{equation*}
 U\widetilde \Phi_j=(I_\cL\otimes S_j')U,\qquad j=1,\ldots,m.
\end{equation*}
Hence,  if $F$ has the representation
$$
F(Y_1,\ldots, Y_m):= \sum_{k=0}^\infty \sum_{\alpha\in
\FF_m^+,|\alpha|=k} Y_\alpha\otimes A_{(\alpha)}, \qquad
(Y_1,\ldots, Y_m)\in [B(\cK)^m]_{1},
$$
we deduce that
\begin{equation*}
\begin{split}
(U\otimes I)F(r\widetilde \Phi_1,\ldots, r\widetilde \Phi_m) &=
\sum_{k=0}^\infty \sum_{\alpha\in \FF_m^+,|\alpha|=k}
r^{|\alpha|}(U\otimes I)\widetilde\Phi_\alpha\otimes A_{(\alpha)}\\
 &=
 \left(\sum_{k=0}^\infty \sum_{\alpha\in \FF_m^+,|\alpha|=k}
r^{|\alpha|}(I_\cL\otimes S_\alpha')\otimes
A_{(\alpha)}\right)(U\otimes
I)\\
&=[I_\cL\otimes F(rS_1',\ldots, rS_m')](U\otimes I).
\end{split}
\end{equation*}
Now, using Theorem \ref{more-prop}, we obtain
\begin{equation}\label{FPhi}
\begin{split}
(U\otimes I)(\widetilde {F\circ\Phi})&=(U\otimes I)\left(\text{\rm
SOT-}\lim_{r\to 1} F(r\widetilde
\Phi_1,\ldots, r\widetilde \Phi_m)\right)\\
&=\text{\rm SOT-}\lim_{r\to 1} [I_\cL\otimes F(rS_1',\ldots,
rS_m')](U\otimes I).
\end{split}
\end{equation}
Since  the Hilbert space $\cL$ is separable,  $\text{\rm
SOT-}\lim_{r\to 1} F(rS_1',\ldots, rS_m')=\widetilde F$,  and
$\|F(rS_1',\ldots, rS_m')\|\leq \|F\|_\infty$, we also have
$$\text{\rm SOT-}\lim_{r\to 1}I_\cL\otimes  F(rS_1',\ldots,
rS_m')=I_\cL\otimes \widetilde F. $$
 Combining the result with  relation \eqref{FPhi}, we conclude that
 \begin{equation}
 \label{FF}
(U\otimes I)\widetilde {F\circ\Phi}=(I_\cL\otimes \widetilde
F)(U\otimes I).
 \end{equation}
Hence,  we deduce that
$$
\|F\circ \Phi\|_\infty=\|\widetilde {F\circ
\Phi}\|_\infty=\|\widetilde F\|=\|F\|_\infty.
$$

Now, if we assume that $F$ is inner, i.e., $\widetilde F^*
\widetilde F=I$, then relation \eqref{FF} implies $(\widetilde
{F\circ \Phi})^* \widetilde {F\circ \Phi}=I$. Therefore,  $F\circ
\Phi$ is inner. Finally, assume that  $\Phi$ is inner and $F$ is
outer, i.e., $\widetilde F$ has dense range. Using again relation
\eqref{FF}, we deduce that $\widetilde {F\circ \Phi}$ has dense
range and, therefore, $F\circ \Phi$ is outer. The proof is complete.
\end{proof}

We  recall   a few well-known facts (see \cite{Si}, \cite{Ph},
\cite{Y}) about fractional maps on the unit ball
$$[B(\cX, \cY)]_1^-:=\{W\in B(\cX, \cY):\ \|W\|\leq 1\},
$$
 where $\cX$  and $\cY$ are Hilbert spaces. We denote by $[B(\cX,
 \cY)]_1$ the open ball of strict contractions.
Let $A,B\in [B(\cX, \cY)]_1^-$ be  such that $\|A\|<1$ and define
$\Psi_A(B)\in B(\cX, \cY)$ by setting
\begin{equation}
\label{frac} \Psi_A(B):=A-D_{A^*}(I-BA^*)^{-1} BD_A,
\end{equation}
where $D_A:=(I-A^*A)^{1/2}$ and $D_{A^*}:=(I-AA^*)^{1/2}$.  One can
show that, for any contractions  $A,B, C\in B(\cX, \cY)$ with
$\|A\|<1$,

\begin{equation}
\label{formulas}
\begin{split}
I-\Psi_A(B)\Psi_A(C)^*&=D_{A^*}(I-BA^*)^{-1}(I-BC^*)(I-A
C^*)^{-1} D_{A^*},\\
I-\Psi_A(B)^*\Psi_A(C)&=D_{A}(I-B^*A)^{-1}(I-B^*C)(I-A^* C)^{-1}
D_{A}.
\end{split}
\end{equation}
Hence, we deduce that $\|\Psi_A(B)\|\leq 1$ and $\|\Psi_A(B)\|<1$
when $\|B\|<1$.
 Straightforward
calculations reveal that \begin{equation} \label{3eq}
\Psi_A(0)=A,\quad \Psi_A(A)=0, \quad  \text{ and } \quad
\Psi_A(\Psi_A(B))=B \quad  \text {for any }\ B\in [B(\cX, \cY)]_1^-.
\end{equation}
Consequently, the fractional map $\Psi_A:[B(\cX, \cY)]_1^- \to
[B(\cX, \cY)]_1^-$ is a homeomorphism and, moreover,
$\Psi_A([B(\cX,\cY)]_1)=[B(\cX, \cY)]_1$.

Consider the noncommutative Schur class
$$
\cS_{\text{\bf ball}}(B(\cE,\cG)):=\left\{G\in H_{\text{\bf
ball}}^\infty (B(\cE,\cG)):\ \|G\|_\infty\leq 1\right\},
$$
which can be identified to the unit ball of  the operator
  space
$F_n^\infty\bar\otimes B(\cE,\cG)$. We also use the notation
$$\cS^0_{\text{\bf ball}}(B(\cE,\cG)):=\{G\in \cS_{\text{\bf
ball}}(B(\cE,\cG)):\ G(0)=0\}.
$$

Fractional transforms of free holomorphic functions were considered
in \cite{Po-hyperbolic2} (see the proof of Theorem 6.1). In what
follows we expand on those ideas and provide new properties.

\begin{theorem} \label{frac-trans}
 Let    $F:[B(\cH)^n]_{1}\to
  B( \cH)\bar\otimes_{min} B(\cE, \cG)$  be a bounded
   free
holomorphic function with $\|F\|_\infty\leq 1$ and let $\widetilde
F$ be its model  boundary function. For each operator
$A=I_\cH\otimes A_0$ with $A_0\in B(\cE, \cG)$ and $\|A_0\|<1$, we
define  the  map
 $$\Psi_A[F] :[B(\cH)^n]_{1}\to
  B( \cH)\bar\otimes_{min} B(\cE, \cG)$$ by setting
  $$
  \Psi_A[F]:=A-D_{A^*}(I-FA^*)^{-1} F D_{A}.
  $$

Then  the following statements hold:

\begin{enumerate}
\item[(i)]
$\Psi_A[F]$ is a bounded free holomorphic function with
$\|\Psi_A[F]\|_\infty\leq 1$ and its boundary function has the
following  properties: $ \widetilde{\Psi_A[F]}=\Psi_A(\widetilde
F)$,
\begin{equation}\label{formulas2}
\begin{split}
I-\widetilde{\Psi_A[F]}
\widetilde{\Psi_A[F]}^*&=D_{A^*}(I-\widetilde
FA^*)^{-1}(I-\widetilde F \widetilde F^*)(I-A
\widetilde F^*)^{-1} D_{A^*},\\
I- \widetilde{\Psi_A[F]}^* \widetilde{\Psi_A[F]}&=D_{A}(I-
\widetilde F^*A)^{-1}(I- \widetilde F^*\widetilde F)(I-A^*
\widetilde F)^{-1} D_{A};
\end{split}
\end{equation}
\item[(ii)] for any $X\in [B(\cH)^n]_1$,
$$\Psi_A[F](X)=(P_X\otimes I)\{\Psi_A(\widetilde F)\}=
A-D_{A^*}[I-F(X)A^*]^{-1} F(X) D_{A}=\Psi_A[F(X)],
$$
where $P_X$ is the noncommutative Poisson at $X$;

\item[(iii)] $\Psi_A[0]=A$, $\Psi_A[A]=0$, and  $\Psi_A [\Psi_A[F]]=F$;
\item[(iv)] $\Psi_A:\cS_{\text{\bf ball}}(B(\cE,\cG))\to \cS_{\text{\bf
ball}}(B(\cE,\cG))$ is a homeomorphism;
\item[(v)] $\widetilde F$ is inner if and only if  $\widetilde{\Psi_A[F]}$
is inner;
\item[(vi)]  $\widetilde F$ is in $\cA_n\bar \otimes_{min} B(\cE,\cG)$ if
and only if $\widetilde{\Psi_A[F]}$  is in $\cA_n\bar \otimes_{min}
B(\cE,\cG)$.
\end{enumerate}
\end{theorem}

\begin{proof} To prove (i), note that $FA^*$ is a bounded free
holomorphic function on $[B(\cH)^n]_1$ and $\|FA^*\|_\infty\leq
\|A\|<1$. Since the map $Y\mapsto (I-Y)^{-1}$ is a free holomorphic
on $[B(\cK)]_1$, Theorem \ref{more-prop}  implies that
$(I-FA^*)^{-1}$ and, consequently, $\Psi_A[F]$  are bounded free
holomorphic functions on $[B(\cH)^n]_1$. On the other hand, since
$$\|F\|_\infty=\sup_{r\in [0,1)}\|F(rS_1,\ldots, rS_n)\|\leq 1$$
and using the properties of the fractional transform $\Psi_A$, we
deduce that
$$
\|\Psi_A[F](rS_1,\ldots, rS_n)\|=\|\Psi_A(F(rS_1,\ldots,
rS_n))\|\leq 1
$$
for any $r\in [0,1)$. Hence $\|\Psi_A[F]\|_\infty\leq 1$. Since $F$
is a bounded free holomorphic function, we know  (see
\cite{Po-holomorphic}, \cite{Po-pluriharmonic}) that the boundary
function
$$\widetilde F:=\text{\rm SOT-}\lim_{r\to 1}
F(rS_1,\ldots, rS_n) $$ exists. Taking into account that $\|A\|<1$
and $\|F(rS_1,\ldots, rS_n)\|\leq \|\widetilde F\|\leq 1$, one can
easily see that
$$
\text{\rm SOT-}\lim_{r\to 1}(I- F(rS_1,\ldots,
rS_n)A^*)^{-1}=(I-\widetilde F A^*)^{-1}
$$
and, moreover,
$$
\widetilde {\Psi_A[F]}=\text{\rm SOT-}\lim_{r\to 1}\Psi_A(
F(rS_1,\ldots, rS_n))=\Psi_A(\widetilde F).
$$
Now, notice that relation \eqref{formulas2} follows from
\eqref{formulas}. This proves part (i).

Using the Poisson representation for bounded free holomorphic
functions and the continuity of the Poisson transform in the
operator norm topology, we obtain
\begin{equation*}
\begin{split}
\Psi_A[F](X)&=(P_X\otimes I)\{\widetilde {\Psi_A[F]}\}=(P_X\otimes
I)\{ {\Psi_A[\widetilde F]}\}\\
&=A-D_{A^*}[I-F(X)A^*]^{-1} F(X) D_{A}=\Psi_A[F(X)]
\end{split}
\end{equation*}
for any $X\in [B(\cH)^n]_1$, which proves part (ii). Hence and using
relation  \eqref{3eq}, one can deduce  (iii).

Now let us prove item (iv). Let $F, F_m\in \cS_{\text{\bf
ball}}(B(\cE,\cG))$ and assume that $\|F_m- F\|_\infty \to 0$ as
$m\to \infty$, which  is equivalent to $\|\widetilde F_m- \widetilde
F \|\to 0$. Using the fact that
\begin{equation*}
\begin{split}
\|(I- \widetilde F_m^*A)^{-1}-(I- \widetilde F^*A)^{-1}\| &\leq
\|(I- \widetilde F_m^*A)^{-1} (\widetilde F_m A^*-\widetilde F
A^*)(I- \widetilde F_m^*A)^{-1}\| \\
&\leq \frac{\|A\|}{(1-\|A\|)^2} \|\|\widetilde F_m -\widetilde F\|,
\end{split}
\end{equation*}
we deduce that $\Psi_A(\widetilde F_m)\to \Psi_A(\widetilde F)$, as
$m\to\infty$. Due to (i), we have
 $$
 \lim_{m\to\infty}\|\Psi_A[F_m]-\Psi_A[F]\|_\infty=
 \lim_{m\to \infty}\|\Psi_A(\widetilde F_m)- \Psi_A(\widetilde
 F)\|=0.
 $$
Moreover, since $\Psi_A[ \Psi_A[F]]= F$, we deduce that
$\Psi_A^{-1}$ is continuous, as well, in the uniform norm
$\|\cdot\|_\infty$.

Note that item (v) follows from relation \eqref{formulas2}. To prove
(vi), we assume that $\widetilde F$ is in $\cA_n\bar \otimes_{min}
B(\cE,\cG)$. Then, due to \cite{Po-holomorphic}, $\widetilde
F=\lim_{r\to1} F(rS_1,\ldots, rS_n)$ in the operator norm topology.
Since $ \|F(rS_1,\ldots, rS_n)A^*\| \leq \|F\|_\infty \|A\|<1$, we
deduce that
$$
\lim_{r\to 1}(I- F(rS_1,\ldots, rS_n)A^*)^{-1}=(I-\widetilde F
A^*)^{-1}
$$
and, due to (iv),
$$
\widetilde {\Psi_A[F]}=\lim_{r\to 1}\Psi_A[ F](rS_1,\ldots,
rS_n)=\lim_{r\to 1}\Psi_A( F(rS_1,\ldots, rS_n))=\Psi_A(\widetilde
F)
$$
where the limits are in the operator norm topology. Using again
\cite{Po-holomorphic}, we conclude  that $\widetilde{\Psi_A[F]}$  is
in $\cA_n \bar \otimes_{min} B(\cE,\cG)$. The converse  follows
using item (iv) and  the fact that  $\Psi_A[ \Psi_A[F]]= F$. Indeed,
if $\widetilde{\Psi_A[F]}$  is in $\cA_n\bar \otimes_{min}
B(\cE,\cG)$, then
\begin{equation*}
\begin{split}
\Psi_A\{\widetilde {\Psi_A[F]}\}&=\Psi_A\{\lim_{r\to 1}\Psi_A[
F](rS_1,\ldots, rS_n)\}\\
&=\lim_{r\to 1}\Psi_A\{\Psi_A( F(rS_1,\ldots, rS_n))\}\\
&=\lim_{r\to 1} F(rS_1,\ldots, rS_n)=\widetilde F,
\end{split}
\end{equation*}
where the limits are in the operator norm topology. Consequently,
$\widetilde F$ is in $\cA_n \bar\otimes_{min} B(\cE,\cG)$. The proof
is complete.
\end{proof}

Note that under the conditions of Theorem \ref{frac-trans}, we have
\begin{equation}\label{two-ident}
\begin{split}
I-\Psi_A[F](X) \Psi_A[F](X)^*
&=D_{A^*}[I-F(X)A^*]^{-1}[I-F(X)F(X)^*][I-A
F(X)^*]^{-1} D_{A^*},\\
I- \Psi_A[F](X)^* \Psi_A[F](X)
&=D_{A}[I-F(X)^*A]^{-1}[I-F(X)^*F(X)][I-A^* F(X)]^{-1} D_{A}
\end{split}
\end{equation}
for any $X\in [B(\cH)^n]_1$. Moreover, we have
\begin{enumerate}
\item[(i)] $\|F(X)\|<1$ if and only if $\|\Psi_A[F](X)\|<1$;
\item[(ii)] if $X\in [B(\cH)^n]_1$, then $F(X)$ is an isometry (resp. co-isometry) if and only if
$\Psi_A[F](X)$ has the same property.
\end{enumerate}

We recall (see \cite{Po-holomorphic}, \cite{Po-pluriharmonic}) that
if $F:[B(\cH)^n]_1 \to B(\cH)\bar\otimes_{min} B(\cE,\cG)$ is a free
holomorphic function  with coefficients in $B(\cE,\cG)$ and $F_r(X
):=F(rX)$ for any $X:=(X_1,\ldots, X_n)\in[B(\cH)^n]_{1/r}$, $r\in
(0,1)$, then $F_r$ is free holomorphic on $[B(\cH)^n]_{1/r}$ and
$$
\|F_r\|_\infty=\sup_{\|X\|\leq
r}\|F(X)\|=\sup_{\|X\|=r}\|F(X)\|=\|F(rS_1,\ldots, rS_n)\|,
$$
where $S_1,\ldots, S_n$ are the left creation operators.
 Moreover, the map $[0,1)\ni r\mapsto \|F_r\|_\infty$ is
increasing.

 This result can be improved  for free holomorphic functions with
 scalar coefficients.
We recall that, in \cite{Po-automorphism}, we  proved  a maximum
principle for free holomorphic functions on the noncommutative ball
$[B(\cH)^n]_1$, with scalar coefficients. More precisely, we showed
that if $f:[B(\cH)^n]_1 \to B(\cH)$ is  a   free holomorphic
function and there exists $X_0\in [B(\cH)^n]_1$ such that
$\|f(X_0)\|\geq \|f(X)\|$ for any $X\in [B(\cH)^n]_1$, the $f$ must
be a constant. As a consequence of this principle and the
noncommutative von Neumann inequality, one can easily obtain the
following.

\begin{proposition}\label{strict}
Let $f:[B(\cH)^n]_1 \to B(\cH)$ be   a non-constant   free
holomorphic function with $\|f\|_\infty\leq 1$.  Then the following
statements hold:
\begin{enumerate}
\item[(i)] $\|f(X_1,\ldots, X_n)\|<1$ for any $(X_1,\ldots,
X_n)\in[B(\cH)^n]_1$;
\item[(ii)] the map $[0,1)\ni r\mapsto \|f_r\|_\infty$ is strictly
increasing.
\end{enumerate}
\end{proposition}
\begin{proof} The first part follows immediately from the maximum
principle for free holomorphic functions on the noncommutative ball
$[B(\cH)^n]_1$. To prove the second part, let $0\leq r_1<r_2<1$. We
recall that,  if $r\in[0,1)$, then the boundary function $\widetilde
f_r$ is in the noncommutative disc algebra $\cA_n$ and
$\|f_r\|_\infty=\|\widetilde f_r\|=\|f_r(S_1,\ldots, rS_n)\|$.
Applying part (i) to $f_{r_2}$ and $(X_1,\ldots,
X_n):=(\frac{r_1}{r_2}S_1,\ldots, \frac{r_1}{r_2}S_n)$, we obtain
$$
\|f_{r_1}\|_\infty=\|f_{r_1}(S_1,\ldots,
S_n)\|=\left\|f_{r_2}\left(\frac{r_1}{r_2}S_1,\ldots,
\frac{r_1}{r_2}S_n\right)\right\|<\|f_{r_2}(S_1,\ldots,
S_n)\|=\|f_{r_2}\|_\infty,
$$
which completes the proof.
\end{proof}

Now, using fractional transforms,  and the noncommutative version of
Schwarz's lemma \cite{Po-holomorphic},  we extend the maximum
principle to free holomorphic functions with operator-valued
coefficients.

\begin{theorem}\label{maxmod}
 Let    $F:[B(\cH)^n]_{1}\to
  B( \cH)\bar\otimes_{min} B(\cE, \cG)$  be a bounded
   free
holomorphic function with  $\|F(0)\|<\|F\|_\infty$.  Then   there is
no $X_0\in [B(\cH)^n]_1$ such that
$$\|F(X_0)\|= \|F\|_\infty.
$$
\end{theorem}

\begin{proof} Without loss of generality, we can assume that
$\|F\|_\infty=1$. Set $A:=F(0)$ and let $G:=\Psi_A[F]$. Due to
Theorem \ref{frac-trans}, $G$ is a bounded free holomorphic fonction
with $\|G\|_\infty\leq 1$ and $G(0)=\Psi_A(A)=0$. Applying the
noncommutative Schwarz lemma (see \cite{Po-holomorphic}), we obtain
$$
\|G(X)\|=\|\Psi_A(F(X))\|\leq \|X\|<1, \qquad X\in [B(\cH)^n]_1.
$$
Using again Theorem \ref{frac-trans},   we  have  $(\Psi_A\circ
\Psi_A)[F]=F$ and, therefore,
$$\|F(X)\|<1=\|F\|_\infty,\qquad X\in [B(\cH)^n]_1.
$$
The proof is complete.
\end{proof}

We need to make a few remarks, which are familiar in the case $n=1$
(see \cite{SzF-book}). First, we recall (see \cite{Po-automorphism})
that in the scalar case, $\cE=\cG=\CC$, if $f:[B(\cH)^n]_1\to
B(\cH)$ is a free holomorphic function  and $\|f\|_\infty=|f(0)|$,
then $f$ must be a constant.
 On the other hand, if $F$ is not a scalar free holomorphic function and
$\|F\|_\infty=\|F(0)\|$,  then  Theorem \ref{maxmod} fails. Indeed,
take $\cE=\cG=\CC^2$, and
$$F(X)=
 \left[\begin{matrix} I& 0\\
0& g(X)\end{matrix}\right],$$ where $g$ is a scalar free holomorphic
function  with $\|g(X)\|<1$ for $X=(X_1,\ldots, X_n)\in
[B(\cH)^n]_1$ (for example, $g(X)=X_\alpha$,  $\alpha\in \FF_n^+$
with $|\alpha|\geq 1$). Note that $\|F(X)\|=1=\|F(0)\|$ for any
$X\in [B(\cH)^n]_1$.

We also mention that,  if  $\|F\|_\infty= 1$ and  $F(0)$ is an
isometry, then $F$ must be a constant. Indeed, if $F$ has the
representation $f(X)=\sum_{k=0}^\infty \sum_{|\alpha|=k} X_\alpha
\otimes A_{(\alpha)}$, then, due to \cite{Po-Bohr}, we have
$$ \sum_{|\alpha|=k}A_{(\alpha)}^*A_{(\alpha)}\leq I-F(0)^* F(0)
\quad \text{ for }  k=1,2,\ldots. $$
 Hence, we deduce our assertion.

Using Theorem \ref{maxmod}, one can prove the following result.
Since the proof is similar to that of Proposition \ref{strict}, we
shall omit it.
\begin{corollary}
\label{not-strict}
 Let    $F:[B(\cH)^n]_{1}\to
  B( \cH)\bar\otimes_{min} B(\cE, \cG)$  be a bounded
   free
holomorphic function with  $\|F\|_\infty\leq 1$ and $\|F(0)\|< 1$.
Then
$$
\|F(X)\|<\|F\|_\infty\quad \text{ for any } \ X\in [B(\cH)^n]_{1}.
$$
If $F(0)=0$, then the map $[0,1)\ni r\mapsto \|F_r\|_\infty$ is
strictly increasing.
\end{corollary}

We remark that, in general,  under the conditions of Corollary
\ref{not-strict}, but without the condition $F(0)=0$, the map
$[0,1)\ni r\mapsto \|F_r\|_\infty$ is not necessarily  strictly
increasing. Indeed, take
$$
F(X_1,\ldots,X_n)=\left[\begin{matrix} \frac{1}{3}I& 0\\
0& \frac{1}{2} X_1\end{matrix}\right]$$ and note that
$\|F\|_\infty=\frac{1}{2}$ and
$$
\|F(rS_1,\ldots, rS_n)\|= \begin{cases} \frac{1}{3} &\quad r\in
[0,\frac{2}{3}]\\
\frac{r}{2} &\quad  r\in (\frac{2}{3},1].
\end{cases}
$$

 Denote by $H^+_{\bf ball}(B(\cE))$ the
set of all free holomorphic functions $f$   on the noncommutative
ball $[B(\cH)^n]_1$ with coefficients in $B(\cE)$, where $\cE$ is a
separable Hilbert space, such that $\Re f\geq 0$, where
$$(\Re F)(X):=\frac {F(X)^*+ F(X)}{2}, \qquad X\in [B(\cH)^n]_1.$$
Let $[H_{\bf ball}^\infty(B(\cE))]^{\text{\rm inv}}$ denote the set
of all bounded free holomorphic functions  on $[B(\cH)^n]_1$ with
representation $F(X_1,\ldots, X_n)=\sum_{\alpha\in \FF_n^+}
A_{(\alpha)}\otimes X_\alpha$   such that $I_\cE- A_{(0)}$ is an
invertible operator in $B(\cE)$.
According to \cite{Po-free-hol-interp}, the noncommutative Cayley
transform   defined by
$$\text{\boldmath{$\cC$}} [F] := [F -1][1+ F ]^{-1}
$$
 is a
bijection between $H^+_{\bf ball}(B(\cE))$ and the unit ball of
$[H_{\bf ball}^\infty(B(\cE))]^{\text{\rm inv}}$. In this case, we
have
$$\text{\boldmath{$\cC$}}^{-1}[G] =[I+G ][I-G ]^{-1}.
$$
Consider   also the   set
\begin{equation*}
{\bf H}_1^+(B(\cE)):=\left\{ f\in H^+_{\bf ball}(B(\cE)) :\ f(0)=I
\right\}.
\end{equation*}
 Now, we recall  that the restriction to ${\bf
H}_1^+(B(\cE))$ of the  {\it noncommutative Cayley transform}   is a
bijection $\text{\boldmath{$\cC$}}:  {\bf H}_1^+(B(\cE))\to
\cS^0_{\text{\bf ball}}(B(\cE))$, where the noncommutative Schur
class  $\cS^0_{\text{\bf ball}}(B(\cE))$ was introduced before
Theorem \ref{frac-trans}.

Using fractional transforms, we  can prove  the following theorem
concerning the structure of bounded free holomorphic functions.

\begin{theorem} \label{structure} A map  $F:[B(\cH)^n]_{1}\to
  B( \cH)\bar\otimes_{min} B(\cE)$  is  a bounded
   free
holomorphic function  such that  $\|F\|_\infty\leq 1$ and
$\|F(0)\|<1$, if and only if  there exist a strict contraction
$A_0\in B(\cE)$, an $n$-tuple of isometries $(V_1,\ldots, V_n)$ on a
Hilbert space $\cK$,
 with orthogonal ranges,
  and  an isometry  $W:\cE\to \cK$, such that
$$F=(\Psi_{I\otimes A_0}\circ \text{\boldmath{$\cC$}})[G],
$$ where $\text{\boldmath{$\cC$}}$ is the noncommutative Cayley transform and   $G$ is defined by
$$
G(X_1,\ldots, X_n)=  ( I\otimes W^*) \left[ 2(I- X_1\otimes
V_1^*-\cdots - X_n\otimes V_n^* )^{-1}- I\right]( I\otimes W)
$$
for any $X:=(X_1,\ldots, X_n)\in [B(\cH)^n]_1$. In this case,
$F(0)=I\otimes A_0$.
\end{theorem}
\begin{proof}

Let $F:[B(\cH)^n]_{1}\to
  B( \cH)\bar\otimes_{min} B(\cE)$  be a bounded
   free
holomorphic function with $\|F\|_\infty\leq 1$ and $\|F(0)\|<1$.
Then $F\in \cS_{\text{\bf ball}}(B(\cE))$ and, due to Theorem
\ref{frac-trans}, $\Psi_{F(0)}[F]\in \cS_{\text{\bf
ball}}^0(B(\cE))$. Since the noncommutative Cayley transform
$\text{\boldmath{$\cC$}}:  {\bf H}_1^+(B(\cE))\to \cS^0_{\text{\bf
ball}}(B(\cE))$  is a bijection, we deduce that
$\text{\boldmath{$\cC$}}^{-1}\left(\Psi_{F(0)}[F]\right)\in {\bf
H}_1^+(B(\cE))$.

According to \cite{Po-pluriharmonic}, a free holomorphic function
$G$  is in  $ {\bf H}_1^+(B(\cE))$, i.e., $G(0)=I$ and  $\Re G\geq
0$, if and only if there exists an $n$-tuple of isometries
$(V_1,\ldots, V_n)$ on a Hilbert space $\cK$,
 with orthogonal ranges,
  and  an isometry  $W:\cE\to \cK$
such that
$$
G(X_1,\ldots, X_n)=  ( I\otimes W^*) \left[ 2(I- X_1\otimes
V_1^*-\cdots - X_n\otimes V_n^* )^{-1}- I\right]( I\otimes W).
$$
This completes the proof.
\end{proof}

We remark that, in the scalar case, i.e., $\cE=\CC$,  due to the
maximum principle  for free holomorphic functions, any nonconstant
   free
holomorphic function  $f$ such that  $\|f\|_\infty\leq 1$,  has the
property that $|f(0)|<1$. Therefore, we can apply  Theorem
\ref{structure} and obtain  a characterization and  a
parametrization of all   bounded free holomorphic functions on
$[B(\cH)^n]_1$.

 \bigskip

\section{  Vitali convergence and identity theorem for free holomorphic   functions
   }

In this section, we provide a Vitali type convergence theorem   for
uniformly bounded sequences of free holomorphic functions  with
operator-valued coefficients.

\begin{theorem}\label{Vitali}
Let $\{F_m\}_{m=1}^\infty$ be a  uniformly bounded sequence of free
holomorphic functions on $[B(\cH)^n]_1$ with coefficients in
$B(\cE)$.  Let $\{A^{(k)}\}_{k=1}^\infty\subset [B(\cH)^n]_1$ be a
sequence of operators with  the following properties:
\begin{enumerate}
\item[(i)]  $A^{(k)}$ is bounded below, for each $k=1,2,\ldots;$
\item[(ii)] $\lim_{k\to\infty}\|A^{(k)}\|=0;$
\item[(iii)] $\lim_{m\to \infty} F_m(A^{(k)})$ exists in the
operator norm topology, for each $k=1,2,\ldots$.
\end{enumerate}
Then there exists a free holomorphic function $F$ on $[B(\cH)^n]_1$
with coefficients in $B(\cE)$ such that $F_m$ converges to $F$
uniformly on any closed ball $[B(\cH)^n]_r^-$, $r\in [0,1)$.
\end{theorem}

\begin{proof} For each $m=1,2,\ldots, $ let $F_m$ have the
representation
$$
F_m(X_1,\ldots, X_n)=\sum\limits_{k=0}^\infty
\sum\limits_{|\alpha|=k} X_\alpha\otimes  C^{(m)}_{(\alpha)},
$$
where the series converges in the operator  norm topology  for any
$X=(X_1,\ldots, X_n)\in [B(\cH)^n]_1$. Let $M>0$ be such that
$\|F_m(X)\|\leq M$ for any $X\in [B(\cH)^n]_1$ and $m=1,2,\ldots$.
Due to the Cauchy type estimation of Theorem 2.1 from
\cite{Po-holomorphic}, we have
\begin{equation}
\label{Cauchy} \left\| \sum_{|\alpha|=j}(C^{(m)}_{(\alpha)})^*
C^{(m)}_{(\alpha)}\right\|^{1/2}\leq M,\quad \text{ for any }
m=1,2,\ldots,  \text{ and } j=0,1,\ldots.
\end{equation}
Since $\|F_m(X)-F_m(0)\|\leq 2 M$ for $X\in [B(\cH)^n]_1$, the
Schwarz type result for free holomorphic functions
\cite{Po-holomorphic} implies
$$
\|F_m(A^{(k)})-I_\cH\otimes A_{(0)}^{(m)}\|\leq 2M\|A^{(k)}\|
$$
for any $m,k=1,2,\ldots$. Hence, we deduce that
\begin{equation*}
\begin{split}
\|A_{(0)}^{(m)}-A_{(0)}^{(q)}\|&\leq \|I_\cH\otimes
A_{(0)}^{(m)}-F_m(A^{(k)})\|+ \|F_m(A^{(k)})-F_q(A^{(k)})\|+
\|F_q(A^{(k)})- I_\cH\otimes A_{(0)}^{(m)}\| \\
&\leq 4M \|A^{(k)}\|+\|F_m(A^{(k)})-F_q(A^{(k)})\|.
\end{split}
\end{equation*}
Since $\lim_{k\to\infty}\|A^{(k)}\|=0$ and $\lim_{m\to \infty}
F_m(A^{(k)})$ exists in the operator norm topology, for each
$k=1,2,\ldots$, we deduce that $C_{(0)}:=\lim_{m\to\infty}
C_{(0)}^{(m)}$ exists.

Let $F_{m,0}:=F_m$ and note that
$$
F_{m,0}=I\otimes C_{(0)}^{(m)}+ [X_1\otimes I_\cE,\ldots, X_n\otimes
I_\cE]F_{m,1}(X)
$$
where
$$
F_{m,1}(X)=\left[\begin{matrix} F_{m,1}^{(g_1)}(X)\\\vdots\\
F_{m,1}^{(g_n)}(X)
\end{matrix}\right]
$$
and
$$
F_{m,1}^{(g_i)}(X)= I_\cH\otimes C_{(g_i)}^{(m)}
+\sum_{j=1}\sum_{|\beta|=j}X_\beta\otimes
C_{(g_i\beta)}^{(m)},\qquad i=1,\ldots, n.
$$
By induction over $q=0,1,2,\ldots$, we can easily prove that
$$
\sum_{|\alpha|\geq q} X_\alpha \otimes C_{(\alpha)}=[X_\beta\otimes
I_\cE:\ |\beta|=q] F_{m,q}(X),
$$
where
$$
F_{m,q}(X):=
\left[\begin{matrix} F_{m,q}^{(\beta)}(X)\\:\\
|\beta|=q
\end{matrix}\right]
$$
and

$$
F_{m,q}^{(\beta)}(X):=I_\cH\otimes C_{(\beta)}^{(m)}+
\sum_{j=1}\sum_{|\gamma|=j}X_\gamma\otimes C_{(\beta\gamma)}^{(m)}
\quad \text{ for } |\beta|=q.
$$

Now, note that, for each $i=1,\ldots, n$, we have
$$
F_{m,1}^{(g_i)}(X)=I_\cH\otimes C_{(g_i)}^{(m)}+ [X_1\otimes
I_\cE,\ldots, X_n\otimes I_\cE] \left[\begin{matrix} F_{m,2}^{(g_i
g_1)}(X)\\\vdots\\ F_{m,2}^{(g_i g_n)}(X)|
\end{matrix}\right].
$$
Consequently, we have

\begin{equation}
\label{Fq} F_{m,1}(X)=
\left[\begin{matrix} I_\cH\otimes C_{(g_1)}^{(m)}\\\vdots\\
I_\cH\otimes C_{(g_n)}^{(m)} \end{matrix}\right] +  ([X_1\otimes
I_\cE,\ldots, X_n\otimes I_\cE]\otimes I_{\CC^{N_1}}) F_{m,2}(X),
\end{equation}
where $N_1:=n$. One  can easily  prove by induction over
$q=0,1,\ldots$ that

\begin{equation}
\label{F12} F_{m,q}(X)=
\left[\begin{matrix} I_\cH\otimes C_{(\beta)}^{(m)}\\:\\
|\beta|=q \end{matrix}\right] +  ([X_1\otimes I_\cE,\ldots,
X_n\otimes I_\cE]\otimes I_{\CC^{N_q}}) F_{m,q+1}(X)
\end{equation}
for any $X\in [B(\cH)^n]_1$ and $m=1,2,\ldots$, where
$N_q:=\text{\rm card}\, \{\beta\in \FF_n^+:\ |\beta|=q\}$.

In what follows we prove by induction over $p=0,1,\ldots$ the
following statements:
\begin{enumerate}
\item[(a)]
$\lim\limits_{m\to \infty} F_{m,p}(A^{(k)})$ exists in the operator
norm topology, for each $k=1,2,\ldots$;
\item[(b)] $\|F_{m,p}(X)\|\leq (p+1) M$ for any $X\in [B(\cH)^n]_1$
and $m=1,2,\ldots;$
\item[(c)]
$\left\|F_{m,p}(A^{(k)})-\left[\begin{matrix} I_\cH\otimes C_{(\beta)}^{(m)}\\:\\
|\beta|=p \end{matrix}\right]\right\|\leq (p+2)M\|A^{(k)}\|$ for any
$k,m=1,2,\ldots;$
\item[(d)]
$\left[\begin{matrix}  C_{(\beta)}\\:\\
|\beta|=p \end{matrix}\right]:=\lim_{m\to\infty}
\left[\begin{matrix}  C_{(\beta)}^{(m)}\\:\\
|\beta|=p \end{matrix}\right]$ exists in the operator norm topology.
\end{enumerate}
Assume that these relations hold for $p=q$. Using relation
\eqref{F12} when $X=A^{(k)}$ and  taking into account (a) and (d)
(when $p=q$), we deduce that the sequence $\{(A^{(k)}\otimes
I_{\cE\otimes \CC^{N_{q+1}}})F_{m,q+1}(A^{(k)})\}_{m=1}^\infty$ is
convergent in the operator norm topology and, consequently, a Cauchy
sequence. On the other hand, since $A^{(k)}$ is  bounded below,
there exists $C>0$ such that $\|A^{(k)} y\|\geq C \|y\|$ for any
$y\in \oplus_{i=1}^n\cH$. This implies that
\begin{equation*}
\begin{split}
&\left\|(A^{(k)}\otimes I_{\cE\otimes
\CC^{N_{q+1}}})F_{m,q+1}(A^{(k)})x -(A^{(k)}\otimes I_{\cE\otimes
\CC^{N_{q+1}}})F_{t,q+1}(A^{(k)})x\right\|\\
&\qquad \qquad \geq C
\left\|F_{m,q+1}(A^{(k)})x-F_{t,q+1}(A^{(k)})x\right\|
\end{split}
\end{equation*}
for any $x\in \cH\otimes \cE\otimes \CC^{N_{q+1}}$ and $m,t=1,2,
\ldots$. Hence, we deduce that $\{F_{m,q+1}(A^{(k)})\}_{m=1}^\infty$
is a Cauchy sequence and, therefore,
$\lim_{m\to\infty}F_{m,q+1}(A^{(k)})$ exists.

Now, due to  relation \eqref{Cauchy} and (b), we have
$$\left\| F_{m,q}(X)-\left[\begin{matrix} I_\cH\otimes C_{(\beta)}^{(m)}\\:\\
|\beta|=q \end{matrix}\right]\right\|\leq (q+2)M,\qquad X\in
[B(\cH)^n]_1.
$$
Using relation \eqref{Fq} and the noncommutatice Schwarz lemma, we
obtain
$$
\left\|(X\otimes I_{\cE\otimes \CC^{N_q}}) F_{m,q+1}(X)\right\|
=
\left\| F_{m,q}(X)-\left[\begin{matrix} I_\cH\otimes C_{(\beta)}^{(m)}\\:\\
|\beta|=q \end{matrix}\right]\right\|\leq (q+2)M\|X\|
$$
for any $X\in [B(\cH)^n]_1$, which implies
$$
\left\| F_{m,q+1}(X)\right\|\leq (q+2)M.
$$
Hence and using again \eqref{Cauchy}, we obtain
$$
\left\| F_{m,q+1}(X)-\left[\begin{matrix} I_\cH\otimes C_{(\beta)}^{(m)}\\:\\
|\beta|=q+1 \end{matrix}\right]\right\|\leq (q+3)M
$$
for any $X\in [B(\cH)^n]_1$. Once again, applying the Schawarz lemma
for free holomorphic functions, we deduce that
$$
\left\| F_{m,q+1}(A^{(k)})-\left[\begin{matrix} I_\cH\otimes C_{(\beta)}^{(m)}\\:\\
|\beta|=q+1 \end{matrix}\right]\right\|\leq (q+3)M\|A^{(k)}\|
$$
for any $k,m=1,2,\ldots$, which is condition (c), when $p=q+1$. Now,
note that
\begin{equation*}
\begin{split}
\left\|\left[\begin{matrix} C_{(\beta)}^{(s)}- C_{(\beta)}^{(m)}\\:\\
|\beta|=q+1 \end{matrix}\right] \right\| &\leq \left\|
\left[\begin{matrix} C_{(\beta)}^{(s)}\\:\\
|\beta|=q+1 \end{matrix}\right]-F_{s,q+1}(A^{(k)})\right\|+
\left\|F_{s,q+1}(A^{(k)})-F_{m,q+1}(A^{(k)})\right\|\\
&\qquad +\left\|F_{m,q+1}(A^{(k)})-
\left[\begin{matrix}  C_{(\beta)}^{(m)}\\:\\
|\beta|=q+1 \end{matrix}\right]\right\|\\
&\leq 2(q+3)M\|A^{(k)}\|+
\left\|F_{s,q+1}(A^{(k)})-F_{m,q+1}(A^{(k)})\right\|.
\end{split}
\end{equation*}
for any $k,s,m=1,2,\ldots$. Since
$\{F_{m,q+1}(A^{(k)})\}_{m=1}^\infty$ is a Cauchy sequence and
$\lim_{k\to \infty} \|A^{(k)}\|=0$, we deduce  condition (d) when
$p=q+1$, i.e., $$
\left[\begin{matrix}  C_{(\beta)}\\:\\
|\beta|=q+1 \end{matrix}\right]:=\lim_{m\to\infty}
\left[\begin{matrix}  C_{(\beta)}^{(m)}\\:\\
|\beta|=q+1 \end{matrix}\right]$$
 exists in the operator norm
topology. This concludes our proof by induction.

Now, due to \eqref{Cauchy}, we deduce that

\begin{equation*}
 \left\| \sum_{|\alpha|=j}C_{(\alpha)}^*
C_{(\alpha)}\right\|^{1/2}\leq M,\quad \text{ for any }
j=0,1,\ldots,
\end{equation*}
which implies $\limsup\limits_{k\to\infty} \left\|\sum_{|\alpha|=k}
C_{(\alpha)}^* C_{(\alpha)}\right\|^{\frac{1} {2k}}\leq1$.
Consequently, the mapping $F(X):=\sum_{j=0}^\infty
\sum_{|\alpha|=j}X_\alpha \otimes C_{(\alpha)}$ is a free
holomorphic function on $[B(\cH)^n]_1$ with coefficients in
$B(\cE)$. If $\|X\|\leq r<1$, then we have
\begin{equation*}
\begin{split}
\|F_m(X)-F(X)\|&\leq
\sum_{j=0}^{p-1}\left\|\sum_{|\alpha|=j}X_\alpha\otimes
(C_{(\alpha)}^{(m)}-C_{(\alpha)})\right\|+
\sum_{j=p}^{\infty}\left\|\sum_{|\alpha|=j}X_\alpha\otimes
(C_{(\alpha)}^{(m)}-C_{(\alpha)})\right\|\\
&\leq \sum_{j=0}^{p-1}\left\|
\left[\begin{matrix}  C_{(\alpha)}^{(m)}-C_{(\alpha)}\\:\\
|\alpha|=j \end{matrix}\right]\right\|+
\sum_{j=p}^{\infty}r^{j}\left\|
\left[\begin{matrix}  C_{(\alpha)}^{(m)}-C_{(\alpha)}\\:\\
|\alpha|=j \end{matrix}\right]\right\|\\
&\leq \sum_{j=0}^{p-1}\left\|
\left[\begin{matrix}  C_{(\alpha)}^{(m)}-C_{(\alpha)}\\:\\
|\alpha|=j \end{matrix}\right]\right\|+ 2M\frac{r^p}{1-r}.
\end{split}
\end{equation*}
Hence and due to relation (d), we deduce that $\|F_m(X)-F(X)\|\to
0$, as $m\to\infty$, uniformly for $X\in [B(\cH)^n]_r^-$, $r\in
[0,1)$. The proof is complete.
\end{proof}

\begin{corollary}\label{identity}
 Let $F$, $G$ be free holomorphic functions on
 $[B(\cH)^n]_1$ with coefficients in
$B(\cE)$.  If there exists a sequence
$\{A^{(k)}\}_{k=1}^\infty\subset [B(\cH)^n]_1$   of bounded bellow
operators   such that  $\lim_{k\to\infty}\|A^{(k)}\|=0$ and
$F(A^{(k)})=G(A^{(k)})$ for any $k=1,2,\ldots$, then $F=G$
\end{corollary}

\begin{proof} When $F$ and $G$ are bounded on $[B(\cH)^n]_1$, we can
apply Theorem \ref{Vitali} to the sequence $F,G,F,G,\ldots$, and
deduce that $F=G$.  Otherwise,  let $r\in (0,1)$  be such that
$\|A^{(k)}\|<r$ and consider $F_r(X):=F(rX)$ and $G_r(X):=G(rX)$ for
$X\in [B(\cH)^n]_1$. Since $F_r$ and $G_r$ are bounded free
holomorphic functions on $[B(\cH)^n]_1$ and
$$F_r(r^{-1}A^{(k)})=F(A^{(k)})=G(A^{(k)})=G_r(r^{-1} A^{(k)})
$$
we can apply the first part of the proof and deduce that $F_r=G_r$.
Consequently, $F(rS_1,\ldots, rS_n)=G(rS_1,\ldots, rS_n)$, where
$S_1,\ldots, S_n$ are the left creation operators. Hence $F=G$.
\end{proof}

\begin{remark} Theorem \ref{Vitali} fails if the  operators
$A^{(k)}$ are not bounded bellow.
\end{remark}
\begin{proof} Let $m=2,3,\ldots$, and
consider the sequence of strict row contractions
$$
A^{(k)}:=\left[ \frac{1}{k}P_{\cP_{m-1}}S_1|_{\cP_{m-1}},\ldots,
\frac{1}{k}P_{\cP_{m-1}}S_n|_{\cP_{m-1}}\right],\quad k=1,2,\ldots,
$$
where $S_1,\ldots, S_n$ are the left creation operators and
$\cP_{m-1}$ is the subspace  of $F^2(H_n)$ spanned by the vectors
$e_\alpha$, with $\alpha\in \FF_n^+$ and $|\alpha|\leq m-1$. Let $F$
and $G$ be any  free holomorphic functions on $[B(\cH)^n]_1$ such
that $F(X)-G(X)=X_\beta$ for some $\beta\in \FF_n^+$ with
$|\beta|=m$. Since $A_\alpha^{(k)}=0$ for $|\alpha|\geq m$, we have
$$
F(A^{(k)})=G(A^{(k)}),\quad k\geq 2,
$$
and $\lim_{k\to \infty}\|A^{(k)}\|=0$. However, $F\neq G$.
\end{proof}

We should mention that in the particular case when $n=1$, $\cE=\CC$,
and $\{A^{(k)}\}$ is a sequence of invertible strict contractions,
we recover the corresponding  results  obtained by Fan \cite{KF1}.

\bigskip

\section{Free holomorphic functions with the radial infimum
property}

We introduce the class of  free holomorphic functions with the
radial infimum property, obtain several characterizations,  and
consider several examples  We study the radial infimum property in
connection with products, direct sums, and compositions of free
holomorphic functions. We also show that the class of functions with
the radial infimum property is invariant under the fractional
transforms of Section 1. These results  are important in the
following  sections.

Let    $F:[B(\cH)^n]_{1}\to
  B( \cH)\bar\otimes_{min} B(\cE, \cG)$  be a bounded
   free
holomorphic function on  $[B(\cH)^n]_1$. Due to
\cite{Po-holomorphic} and \cite{Po-pluriharmonic}, the model
boundary function
\begin{equation*}
\widetilde F(S_1,\ldots, S_n):=\text{\rm SOT-}\lim_{r\to
1}F(rS_1,\ldots, rS_n)
\end{equation*}
exists, and
 $F$ has the {\it
radial supremum property}, i.e.,
$$
\lim_{r\to 1} \sup_{\|x\|=1}\|F(rS_1,\ldots, rS_n)x\|=\|F\|_\infty.
$$

We introduce now the class of free holomorphic functions with the
radial infimum property. We say that $F$ has the {\it radial infimum
property} if
$$
\liminf_{r\to 1} \inf_{\|x\|=1}\|F(rS_1,\ldots,
rS_n)x\|=\|F\|_\infty.
$$

\begin{proposition}
If   $F:[B(\cH)^n]_{1}\to
  B( \cH)\bar\otimes_{min} B(\cE, \cG)$  is   a  bounded
   free
holomorphic function  with  the  radial infimum property such that
$\|F\|_\infty=1$, then $ F$ is inner.
\end{proposition}

\begin{proof}
 We have to show that the boundary function $\widetilde F(S_1,\ldots, S_n)$ is an isometry. To
this end, denote
$$\mu(r):= \inf_{\|x\|=1}\|F(rS_1,\ldots, rS_n)x\|\quad \text{ for  }\ r\in
[0,1).
$$
Due to the noncommutative von Neumann inequality, we have
$$\mu(r)\leq \frac{\|F(rS_1,\ldots, rS_n)y\|}{\|y\|}\leq \|F(rS_1,\ldots,
rS_n)\|\leq \|F\|_\infty
$$
for any $y\in F^2(H_n)\otimes \cE$, $y\neq 0$ and $r\in [0,1)$.
Taking into account that $\liminf_{r\to 1} \mu(r)=\|F\|_\infty$, we
deduce that $\lim_{r\to 1}\|F(rS_1,\ldots, rS_n)y\|=\|y\|$. Since
$\widetilde F(S_1,\ldots, S_n):=\text{\rm SOT-}\lim_{r\to
1}F(rS_1,\ldots, rS_n)$,  it is clear that $\|\widetilde
F(S_1,\ldots, S_n)y\|=\|y\|$ for any $y\in F^2(H_n)\otimes \cE$,
which  Shows that $F$ is inner and completes the proof.
\end{proof}

Now, we present several characterizations for free holomorphic
functions with the radial infimum property.

\begin{theorem}\label{radial}
Let    $F:[B(\cH)^n]_{1}\to
  B( \cH)\bar\otimes_{min} B(\cE, \cG)$  be a bounded
   free
holomorphic function with $\|F\|_\infty=1$. Then the following
statements are equivalent.
\begin{enumerate}
\item[(i)]  $F$ has the radial infimum property.
\item[(ii)]   $ \lim\limits_{r\to 1} \inf\limits_{\|x\|=1}\|F(rS_1,\ldots,
rS_n)x\|= 1.$
\item[(iii)] For every  $\epsilon\in (0,1)$ there is $\delta\in (0,1)$
such that
$$
F(rS_1,\ldots, rS_n)^* F(rS_1,\ldots, rS_n)\geq (1-\epsilon)I \quad
\text{ for any } \ r\in (\delta, 1).
$$
\item[(iv)] There exist constants $c(r)\in (0,1]$, $r\in (0,1)$,
with $\lim_{r\to 1}c(r)=1$ such that
$$
F(rS_1,\ldots, rS_n)^* F(rS_1,\ldots, rS_n)\geq c(r)I.
$$
\item[(v)] There is $\delta \in (0,1)$ such that
$ F(rS_1,\ldots, rS_n)^* F(rS_1,\ldots, rS_n)$ is invertible for any
$r\in (\delta, 1)$ and
$$\lim_{r\to 1} \left\|\left[F(rS_1,\ldots, rS_n)^*
F(rS_1,\ldots, rS_n)\right]^{-1}\right\|=1.
$$
\end{enumerate}
\end{theorem}

\begin{proof}The equivalence of (i)  with (ii) is clear  if one takes into account the
inequality $$\|F(rS_1,\ldots, rS_n)\|\leq \|F \|_\infty,\quad r\in
[0,1).$$  Since  the equivalence of (ii)  with (iii) is
straightforward, we leave it the the reader. To prove the
implication $(ii)\implies (iv)$, define
$$
\mu(r):=\inf_{\|x\|=1}\|F(rS_1,\ldots, rS_n)x\|,\quad  r\in [0,1),
$$
and note that  $0\leq \mu(r)\leq \|F(rS_1,\ldots, rS_n)\|\leq
\|F\|_\infty=1$ and
$$
\|F(rS_1,\ldots, rS_n)x\|\geq \mu(r) \|x\|\quad \text{ for any } \
x\in F^2(H_n)\otimes \cE.
$$
Since the latter inequality is equivalent to $$F(rS_1,\ldots,
rS_n)^* F(rS_1,\ldots, rS_n)\geq \mu(r)^2I, $$  if (ii) holds, then
$\lim_{r\to1}\mu(r)=1$ and (iv) follows.  Conversely, assume that
(iv) holds. Then we have

$$
\|F(rS_1,\ldots, rS_n)x\|^2\geq c(r) \|x\|^2\quad \text{ for any } \
x\in F^2(H_n)\otimes \cE,
$$
which implies
$$
c(r)^{1/2}\leq\inf_{\|x\|=1}\|F(rS_1,\ldots, rS_n)x\|\leq
\|F(rS_1,\ldots, rS_n)\|\leq \|F\|_\infty=1.
$$
Since $\lim_{r\to 1} c(r)=1$, we deduce item (ii).

It remains to prove that $(iv)\leftrightarrow (v)$. First, assume
that  condition (iv) holds. Note that the inequality
\begin{equation}\label{F*F}
F(rS_1,\ldots, rS_n)^* F(rS_1,\ldots, rS_n)\geq c(r)I
\end{equation}
is equivalent to \begin{equation} \label{F*Fx}
 \|F(rS_1,\ldots,
rS_n)^* F(rS_1,\ldots, rS_n)x\|\geq c(r)\|x\|\quad \text{ for any }
\ x\in F^2(H_n)\otimes \cE.
\end{equation}
Indeed, if \eqref{F*F} holds, then
$$
\|F(rS_1,\ldots, rS_n)^* F(rS_1,\ldots, rS_n)x\| \|x\|\geq
\left<F(rS_1,\ldots, rS_n)^* F(rS_1,\ldots, rS_n) x,x\right>\geq
c(r)\|x\|^2,
$$
which proves one implication. Conversely, if \eqref{F*Fx} holds,
then, by squaring, we  deduce that
$$[F(rS_1,\ldots, rS_n)^* F(rS_1,\ldots, rS_n)]^2\geq c(r)^2I.
$$
Hence, we obtain relation \eqref{F*F}. Now, denote
\begin{equation} \label{dr}
d(r):=\inf_{\|x\|=1} \|F(rS_1,\ldots, rS_n)^* F(rS_1,\ldots,
rS_n)x\|,\quad r\in (0,1),
\end{equation}
and note that $ 0<c(r)\leq d(r)\leq 1$. Hence, using \eqref{F*Fx}
and condition (iv), we deduce that the positive operator
$F(rS_1,\ldots, rS_n)^* F(rS_1,\ldots, rS_n)$ is invertible and
$$
\frac{1}{c(r)}\geq \|[F(rS_1,\ldots, rS_n)^* F(rS_1,\ldots,
rS_n)]^{-1}\|=\frac{1}{d(r)}\geq 1.
$$
Since $\lim_{r\to 1}c(r)=1$, we obtain item (v). Conversely, assume
now  that condition (v) holds. Since $\|[F(rS_1,\ldots, rS_n)^*
F(rS_1,\ldots, rS_n)]^{-1}\|=\frac{1}{d(r)}$,where $d(r)$ is given
by \eqref{dr}, we have $\lim_{r\to 1} d(r)=1$ On the other hand, due
to \eqref{dr}, we also have
$$
\|F(rS_1,\ldots, rS_n)^* F(rS_1,\ldots, rS_n)x\| \geq d(r)
\|x\|\quad \text{ for any } \ x\in F^2(H_n)\otimes \cE,
$$
which, as proved above, is equivalent to $F(rS_1,\ldots, rS_n)^*
F(rS_1,\ldots, rS_n)\geq d(r)I$. Since $\lim_{r\to 1} d(r)=1$, we
deduce (iv) and complete the proof.
\end{proof}

Now we consider several examples of bounded free holomorphic
functions with the radial  infimum property. Another  notation is
necessary. If $\omega, \gamma\in \FF_n^+$, we say that  $\omega \geq
\gamma$ if there is $\sigma\in \FF_n^+ $ such that $\omega=
\gamma\sigma$.

\begin{example} \label{EXAMP} Let  $A_{(\alpha)}\in B(\cE, \cG)$ and let $F,G,\varphi$
 be      free
holomorphic function on $[B(\cH)^n]_{1}$ having  the  following
forms:
   \begin{enumerate}
   \item[(i)]
$F(X_1,\ldots, X_n)=\sum_{|\alpha|=m} X_\alpha\otimes A_{(\alpha)}$,
where $\sum_{|\alpha|=m} A_\alpha^* A_\alpha=I_\cE$ and $m\in\NN$;
\item[(ii)]
$G(X_1,\ldots, X_n)= \sum_{k=1}^n \sum_{|\beta|=k, \beta\geq
g_k}X_\beta \otimes A_{(\beta)}$, where $\sum_{k=1}^n
\sum_{|\beta|=k, \beta\geq g_k}A_{(\beta)}^* A_{(\beta)}=I$ and
$g_1,\ldots, g_n$ are the generators of the free semigroup
$\FF_n^+$;
\item[(iii)] $\varphi(X_1,\ldots, X_n)=\sum_{k=0}^\infty  a_kX_2^k X_1$,
where  $a_k\in \CC$ with $\sum_{k=0}^\infty |a_k|^2=1$.
\end{enumerate}
Then, $F,G,\varphi$ have  the radial infimum property.
\end{example}

\begin{proof} Since $S_1,\ldots, S_n$ satisfy  the relation
$S_j^*S_i=\delta_{ij} I$ for $i,j=1,\ldots, n$, one can easily see
that $\{S_\alpha\}_{|\alpha|=m}$ is a sequence of isometries with
orthogonal ranges. Consequently, we have
$$
F(rS_1,\ldots, rS_n)^* F(rS_1,\ldots, rS_n)=r^m I.
$$
Applying Theorem \ref{radial}, we deduce that $F$ has the radial
infimum property. Similarly, one can prove that
\begin{equation*}
\begin{split}
G(rS_1,\ldots, rS_n)^* G(rS_1,\ldots, rS_n)&= \sum_{k=1}^n
\sum_{|\beta|=k, \beta\geq g_k} r^{2|\beta|}
A_{(\beta)}^*A_{(\beta)}\\
&\geq r^{2n}\sum_{k=1}^n \sum_{|\beta|=k, \beta\geq
g_k}A_{(\beta)}^* A_{(\beta)}=r^{2n} I.
\end{split}
\end{equation*}
 Consequently, $G$ has the radial infimum property. Finally, note
that $\{S^k_2 S_1\}_{k=0}^\infty$ is a sequence of isometries with
orthogonal ranges and, consequently
$$
\varphi(rS_1,\ldots, rS_n)^* \varphi(rS_1,\ldots, rS_n)=
r^2\sum_{k=0}^\infty r^{2k} |a_k|^2\leq \sum_{k=1}^\infty |a_k|^2=1.
$$
Hence $\varphi$ is a bounded free holomorphic function and, taking
into account that $\lim\limits_{r\to 1} r^2\sum_{k=0}^\infty r^{2k}
|a_k|^2=1$, Theorem \ref{radial} shows that $\varphi$ has the radial
infimum property.
\end{proof}

\begin{proposition} If $F,G$ are bounded  free holomorphic
functions with   the radial infimum property, then so is their
product $FG$. If, in addition, $\|F\|_\infty=\|G\|_\infty$, then
$\left[\begin{matrix} F & 0\\ 0& G\end{matrix} \right]$  has the
radial infimum property.
\end{proposition}
\begin{proof} Without loss of generality, one can assume that
$\|F\|_\infty=\|G\|_\infty=1$. According to Theorem \ref{radial},
there exist constants $c(r), d(r)\in (0,1]$, $r\in (0,1)$, with
$\lim\limits_{r\to 1}c(r)=\lim\limits_{r\to 1}d(r)=1$ such that
$$
F(rS_1,\ldots, rS_n)^* F(rS_1,\ldots, rS_n)\geq c(r)I\ \text{ and }
\ G(rS_1,\ldots, rS_n)^* G(rS_1,\ldots, rS_n)\geq c(r)I.
$$
Hence, we deduce that
$$
G(rS_1,\ldots, rS_n)^*F(rS_1,\ldots, rS_n)^* F(rS_1,\ldots,
rS_n)G(rS_1,\ldots, rS_n)\geq c(r)d(r)I.
$$
Applying again Theorem \ref{radial}, we conclude that the product
$FG$ has the radial infimum property. To prove the second part of
this proposition, note that
$$
\left[\begin{matrix} F(rS_1,\ldots, rS_n)^* F(rS_1,\ldots, rS_n) &
0\\ 0& G(rS_1,\ldots, rS_n)^* G(rS_1,\ldots,
rS_n)\end{matrix}\right] \geq \min\{c(r), d(r)\} I
$$
and $\lim\limits_{r\to 1} \min\{c(r), d(r)\}=1$. Applying again
Theorem \ref{radial}, we complete the proof.
\end{proof}

The next result will provide several classes of free holomorphic
functions with the radial infimum property.

\begin{theorem}\label{properties}
Let $F:[B(\cH)^m]_1\to B( \cH)\bar\otimes_{min} B(\cE, \cG)$ be  a
bounded  free holomorphic function, and let $\varphi:[B(\cH)^n]_1
\to [B(\cH)^m]_1$ be an inner free holomorphic function . Then the
following statements hold.
\begin{enumerate}
\item[(i)] If $F$ is inner and $\widetilde F\in \cA_m\bar\otimes_{min} B(\cE, \cG)$, then
 $F$ has the radial infimum property.
\item[(ii)] If $F$ is inner,   $\widetilde F\in \cA_m\bar \otimes_{min} B(\cE, \cG) $, and
$\widetilde \varphi=(\widetilde \varphi_1,\ldots, \widetilde
\varphi_m)$ is in  $\cA_n \bar\otimes_{min} B(\CC^m, \CC)$ with
$\widetilde \varphi_1$ non-unitary if $m=1$, then the composition
$F\circ \varphi$ has the radial infimum property.
\item[(iii)]
If  $F$ has the   radial infimum property and $\varphi$ is
homogeneous of degree $q\geq 1$, then $F\circ \varphi$ has the
radial infimum property.
\item[(iv)] If    $A:=I\otimes A_0$, $A_0\in B(\cE,\cG)$, with
$\|A_0\|<1$, then  $F$ has the  the radial infimum property if and
only if the fractional transform $\Psi_A[F]$ has the radial infimum
property.
\end{enumerate}
\end{theorem}

\begin{proof}  According to \cite{Po-holomorphic}, the model  boundary
function $\widetilde F$  is the limit of $F(rS_1',\ldots, rS_m')$ in
the operator norm, as $r\to1$, where $S_1',\ldots, S_m'$ are the
left creation operators on  the full Fock space  $F^2(H_m)$, with
$m$ generators. Consequently, for any $\epsilon\in (0,1)$ there
exists $\delta\in (0,1)$ such that
$$
\|F(rS_1',\ldots,rS_m')-\widetilde F\|<\epsilon\quad \text{ for any
}\ r\in (\delta, 1).
$$
Hence, and due to the fact that $\widetilde F$ is an isometry, we
deduce that, for any $r\in (\delta, 1)$ and $x\in F^2(H_m)\otimes
\cE$,
\begin{equation*}
\begin{split}
\left<F(rS_1',\ldots,rS_m')^*F(rS_1',\ldots,rS_m')x,x\right>^{1/2}&=\|F(rS_1',\ldots,rS_m')\|\\
&\geq \|\widetilde Fx\|-\|F(rS_1',\ldots,rS_m')-\widetilde F x\|\\
&\geq \|x\|-\epsilon \|x\|=(1-\epsilon)\|x\|.
\end{split}
\end{equation*}
 Consequently,
 $$F(rS_1',\ldots,rS_m')^*F(rS_1',\ldots,rS_m')\geq
 (1-\epsilon)^2 I\ \text{ for any } \ r\in(\delta,1)
 $$
  and, due to Theorem
 \ref{radial}, $F$ has the radial infimum property. Therefore, item
 (i) holds.

 To prove (ii), note first that, due to Theorem \ref{inner},
 $ F\circ \varphi$ is inner. Since $\widetilde F\in \cA_m\bar \otimes_{min} B(\cE, \cG) $, and
$\widetilde \varphi\in \cA_n\bar \otimes_{min} B(\CC^m, \CC)$,
Theorem \ref{more-prop} implies that $\widetilde {F\circ \varphi} $
is in $\cA_n\bar \otimes B(\cE,\cG)$. Applying now item (i) to $
F\circ \varphi$, we deduce part (ii).

Now, we prove (iii).  Since $\varphi:=[\varphi_1,\ldots, \varphi_m]$
is homogeneous of degree $m\geq 1$, we deduce that each $\varphi_j$
is a homogeneous noncommutative polynomial  of degree $q$.
Therefore, $\widetilde \varphi_j =\varphi_j(S_1,\ldots, S_n)$ and
\begin{equation}
\label{phi-al} \varphi_\alpha(rS_1,\ldots,
rS_n)=r^{q|\alpha|}\varphi_\alpha(S_1,\ldots, S_n),\qquad \alpha\in
\FF_m^+,
\end{equation}
where $S_1,\ldots, S_n$ are the left creation operators on  the full
Fock space  $F^2(H_n)$.
As in the proof of Theorem \ref{inner}, we have
\begin{equation*}
(U\otimes I)\widetilde {F\circ\varphi}=(I_\cL\otimes \widetilde
F)(U\otimes I),
\end{equation*}
where $\cL$ is a separable  Hilbert space and $U:F^2(H_n)\to
\cL\otimes F^2(H_m)$ is a unitary operator. Hence, we deduce that
\begin{equation}
\label{U-inter2} (U\otimes I)\widetilde
{F_r\circ\varphi}=(I_\cL\otimes \widetilde F_r)(U\otimes I),\qquad
r\in (0, 1),
\end{equation}
where $F_r(X):=F(rX)$, \ $X\in [B(\cH)^m]_{1/r}$. Since $F_r$ is a
bounded free holomorphic  function on $[B(\cH)^m]_{1/r}$, we have
\begin{equation} \label{2-form}
\widetilde F_r=F_r(S_1',\ldots, S_m')\quad  \text{ and }\quad
\widetilde {F_r\circ\varphi}=F_r(\varphi_1(S_1,\ldots, S_m),\ldots,
\varphi_m(S_1,\ldots, S_n)).
\end{equation}
Since $F$ has the radial infimum property,  for any $\epsilon\in
(0,1)$ there is $\delta\in (0,1)$ such that
$$
F_r(S_1',\ldots, S_m')^* F_r(S_1',\ldots, S_m')\geq (1-\epsilon)I
\quad \text{ for any } \ r\in (\delta, 1).
$$
Hence, and using relations \eqref{U-inter2} and \eqref{2-form}, we
obtain
\begin{equation}
\label{Tilda}
 (\widetilde {F_r\circ\varphi})^* \widetilde
{F_r\circ\varphi} \geq (1-\epsilon)I \quad \text{ for any } \ r\in
(\delta, 1).
\end{equation}
On the other hand, due to  relation \eqref{phi-al}, we have
\begin{equation*}
\begin{split}
 \widetilde {F_r\circ\varphi}&=F(r\varphi_1(S_1,\ldots,
S_n),\ldots, r\varphi_m(S_1,\ldots, S_n))\\
&=F(\varphi_1(r^{1/q}S_1,\ldots, r^{1/q}S_n),\ldots,
\varphi_m(r^{1/q}S_1,\ldots, r^{1/q}S_n))\\
&=(F\circ \varphi)(r'S_1,\ldots, r'S_n),
\end{split}
\end{equation*}
where $r':=r^{1/q}$. Now  inequality \eqref{Tilda} becomes
$$
(F\circ \varphi)(r'S_1,\ldots, r'S_n)^*(F\circ
\varphi)(r'S_1,\ldots, r'S_n) \geq (1-\epsilon)I \quad \text{ for
any } \ r'\in (\delta^{1/k}, 1).
$$
Applying Theorem \ref{radial}, we conclude that $F\circ \varphi$ has
the radial infimum property.

To prove item (iv), assume that $F$ has the radial infimum property.
Applying  Theorem \ref{frac-trans} to $\Psi_A[F_r]$,  $r\in (0,1)$,
we obtain
\begin{equation*}
\begin{split}
I- &\Psi_A[F](rS')^* \Psi_A[F](rS')  \\
&\qquad\qquad  =D_{A}[I-F(rS')^*A]^{-1}[I-F(rS')^*F(rS')][I-A^*
F(rS')]^{-1} D_{A},
\end{split}
\end{equation*}
 where $rS':=(rS_1',\ldots, rS_m')$. Since $F$ has the radial infimum property,
  there exist constants $c(r)\in (0,1]$, $r\in (0,1)$,
with $\lim_{r\to 1}c(r)=1$ such that
$$
F(rS')^* F(rS')\geq c(r)I.
$$
Note also that, since $\|A\|<1$ and $\|F(rS')\|\leq 1$, we have
\begin{equation*}
\begin{split}
\|[I-F(rS')^*A]^{-1}\|&\leq 1+\|F(rS')^*A\|+
\|F(rS')^*A\|^2+\cdots\\
&\leq 1+\|A\|+\|A\|^2+\cdots\\&=\frac{1}{1-\|A\|}
\end{split}
\end{equation*}
Using all these relations, we deduce that
$$ \Psi_A[F](rS')^* \Psi_A[F](rS')\geq \left[1-\frac{1-c(r)}{(1-\|A\|)^2}\right] I.
$$
Since $\lim_{r\to 1}c(r)=1$,  Theorem \ref{radial} shows that
$\Psi_A[F]$ has the radial infimum property. Now, using the fact
that $\Psi_F[\Psi_A[F]]=F$, one can prove the converse. The proof is
complete.
\end{proof}

We remark that under the hypothesis of Theorem \ref{properties}, if
$\widetilde F\in \cA_m\bar \otimes_{min} B(\cE, \cG) $ with
$\|F\|_\infty=1$, then $F$ is inner if and only if it
 has the radial infimum property.
This suggests the following open question. Is there any bounded free
holomorphic function with the radial infimum property  so that its
boundary function is not in the noncommutative disc algebra ?

\begin{corollary} Any free holomorphic automorphism of the
noncommutative ball $[B(\cH)^n]_1$ has the radial infimum property.
\end{corollary}
\begin{proof} According to \cite{Po-automorphism}, if $\Psi\in
Aut([B(\cH)^n]_1)$, the automorphism group of  all free holomorphic
functions on $[B(\cH)^n]_1$, then its boundary function $\widetilde
\Psi=[\widetilde \Psi_1,\ldots, \widetilde \Psi_n]$ is an isometry
and $\widetilde \Psi_i\in \cA_n$, the noncommutative disc algebra.
Applying Theorem \ref{properties}, part (i), we deduce that $\Psi$
has the radial infimum property. The proof is complete.
\end{proof}

 \bigskip

\section{Factorizations and  free holomorphic  versions of classical inequalities.
   }

In this section we study  the class of free holomorphic functions
with the radial infimum property   in connection with factorizations
and noncommutative generalizations of  Schwarz's lemma and Harnack's
double inequality from complex analysis.

If $A,B\in B(\cK)$ are selfadjoint  operators, we say that $A<B$ if
$B-A$ is positive and invertible, i.e., there exists a constant
$\gamma>0$ such that $\left<(B-A)h,h\right>\geq \gamma\|h\|^2$ for
any $h\in \cK$. Note that  $C\in B(\cK)$ is a strict contraction
($\|C\|<1$) if and only if $C^*C<I$.

\begin{theorem} \label{Schwarz}
 Let $F, \Theta$, and $ G$ be free holomorphic functions
on $[B(\cH)^n]_1$ with coefficients in $ B(\cE, \cG)$, $B(\cY,\cG)$,
and $B(\cE,\cY)$, respectively, such that
$$F(X)=\Theta (X) G(X),\qquad X\in [B(\cH)^n]_1.
$$
Assume that   $F$ is bounded with $\|F\|_\infty\leq 1$ and $\Theta$
  has the radial infimum property with $\|\Theta\|_\infty=1$. Then $\|G\|_\infty\leq
1$,
\begin{equation*}
 F(X)F(X)^*\leq \Theta(X)\Theta(X)^*, \qquad X\in [B(\cH)^n]_1,
\end{equation*}
and
\begin{equation*}
 \|F(X)\|\leq \|\Theta(X)\|, \qquad X\in [B(\cH)^n]_1.
\end{equation*}
If,  in addition, $\|G(0)\|<1$ and $X_0\in [B(\cH)^n]_1$, then:
\begin{enumerate}
\item[(i)]
$F(X_0) F(X_0)^* < \Theta (X_0) \Theta^*(X_0)$ if and only if
$\Theta (X_0) \Theta^*(X_0)>0$;
\item[(ii)] $\|F(X_0)\|<\|\Theta(X_0)\|$ if and only if $G(X_0)\neq 0$.
\end{enumerate}
\end{theorem}

\begin{proof}
 Since
$\Theta$ has the radial infimum property and $\|\Theta\|_\infty=1$,
Theorem \ref{radial} shows that there exist constants $c(r)\in
(0,1]$, $r\in (0,1)$, with $\lim_{r\to 1}c(r)=1$  and such that
\begin{equation*}
 \Theta(rS_1,\ldots, rS_n)^* \Theta(rS_1,\ldots, rS_n)\geq
c(r)I.
\end{equation*}
Consequently, taking into account that $F(rS_1,\ldots,
rS_n)=\Theta(rS_1,\ldots, rS_n) G(rS_1,\ldots, rS_n)$ for any $r\in
[0,1)$, we deduce that
$$
\|\Theta(rS_1,\ldots, rS_n)^* F(rS_1,\ldots, rS_n)y\|\geq c(r)
\|G(rS_1,\ldots, rS_n)y\|
$$
for any $y\in F^2(H_n)\otimes \cE$. Since $\|\Theta(rS_1,\ldots,
rS_n)\|\leq \|\Theta\|_\infty = 1$ and $\|F(rS_1,\ldots, rS_n)\|\leq
1$, the inequality above implies
$$
c(r) \|G(rS_1,\ldots, rS_n)\|\leq 1\quad \text{ for any } \ r\in
[0,1).
$$
 Using the fact the map $r\mapsto
\|G(rS_1,\ldots, rS_n)\|$ is increasing and that $\lim_{r\to
1}c(r)=1$, we deduce that $\lim_{r\to 1} \|G(rS_1,\ldots,
rS_n)\|\leq 1$. Hence, $G$ is bounded  and
$$
\|G\|_\infty =\lim_{r\to 1} \|G(rS_1,\ldots, rS_n)\|\leq 1.
$$
Consequently,  $G(X)G(X)^*\leq I$ and
\begin{equation*}
 F(X)F(X)^*=\Theta (X) G(X) G(X)^* \Theta(X)^* \leq
\Theta(X)\Theta(X)^*, \qquad X\in [B(\cH)^n]_1.
\end{equation*}
Hence,  we have $F(X)F(X)^*\leq \Theta(X)\Theta(X)^*$  for all $
X\in [B(\cH)^n]_1$.

To prove the second part of this theorem, assume that $\|G(0)\|<1$.
 According to Corollary \ref{not-strict},
we have $\|G(X)\|<1$ for any $X\in [B(\cH)^n]_1$. Since $F(X)=\Theta
(X) G(X)$,  $X\in [B(\cH)^n]_1$, we deduce that
\begin{equation}
\label{T-T}
 \Theta(X) \Theta(X)^*-F(X) F(X)^*\geq
(1-\|G(X)\|^2)\Theta(X) \Theta(X)^*.
\end{equation}
Since $\|G(X)\|<1$, we have $(1-\|G(X))\|^2)\Theta(X)
\Theta(X)^*\geq 0$. Note also that if  $X_0\in [B(\cH)^n]_1$ is such
that $\Theta(X_0) \Theta(X_0)^*>0$ then relation \eqref{T-T} implies
$\Theta(X_0) \Theta(X_0)^*-F(X_0) F(X_0)^*>0$. The converse is
obviously true.

To prove item (ii), note that  when  $\|G(X_0)\|<1$ and $G(X_0)\neq
0$, we have
$$
\|F(X_0)\|=\|\Theta(X_0)G(X_0)\|\leq \|\Theta(X_0)\|\|G(X_0)\|<
\|\Theta(X_0)\|.
$$
Consequently, since $F(X_0)=\Theta (X_0) G(X_0)$, we deduce that
 $\|F(X_0)\|<\|\Theta(X_0)\|$ if and only if $G(X_0)\neq 0$.
This completes the proof.
\end{proof}

\begin{proposition} \label{disc} Let $F, \Theta$, and $ G$ be free holomorphic functions
on $[B(\cH)^n]_1$ with coefficients in $ B(\cE, \cG)$, $B(\cY,\cG)$,
and $B(\cE,\cY)$, respectively, such that
$$F(X)=\Theta (X) G(X),\qquad X\in [B(\cH)^n]_1.
$$
Assume that: \begin{enumerate}
\item[(i)]
$\|F\|_\infty\leq 1$ and $\widetilde F\in \cA_n\bar
\otimes_{min}B(\cE, \cG)$;
\item[(ii)]
$\Theta$ is inner and $\widetilde \Theta \in \cA_n\bar
\otimes_{min}B(\cY, \cG)$.
\end{enumerate}
Then $\|G\|_\infty\leq 1$, $\widetilde G\in \cA_n\bar
\otimes_{min}B(\cE, \cY)$, and
\begin{equation*}
 F(X)F(X)^*\leq \Theta(X)\Theta(X)^*, \qquad X\in [B(\cH)^n]_1.
\end{equation*}
If, in addition, $F$ is inner, then so is $G$.
\end{proposition}

\begin{proof} Since $\Theta$ is inner, i.e.,
$\widetilde\Theta^* \widetilde\Theta=I$, and $\widetilde \Theta \in
\cA_n\bar \otimes_{min}B(\cY, \cG)$, Theorem \ref{properties}
implies that $\Theta$ has the radial infimum property. Now, due to
Theorem \ref{Schwarz}, $G$ is bounded and $\|G\|_\infty\leq 1$.
Consequently, inequality  $ F(X)F(X)^*\leq \Theta(X)\Theta(X)^*$ for
all  $ X\in [B(\cH)^n]_1$.

On the other hand, since $\widetilde F\in \cA_n \bar
\otimes_{min}B(\cE, \cG)$ and $\widetilde \Theta \in \cA_n \bar
\otimes_{min}B(\cY, \cG)$, according to \cite{Po-holomorphic},
\cite{Po-pluriharmonic}, we have

\begin{equation}
\label{wide} \widetilde\Theta =\lim_{r\to 1}\Theta(rS_1,\ldots,
rS_n) \ \text{ and } \ \widetilde F=\lim_{r\to 1} F(rS_1,\ldots,
rS_n),
\end{equation}
in the operator norm topology.   Since $\|G\|_\infty\leq 1$,  its
boundary function $\widetilde G=\text{\rm SOT-}\lim_{r\to 1}
G(rS_1,\ldots, rS_n)$ exists. Now, for any $r\in [0,1)$, we have
\begin{equation*}
\Theta(rS_1,\ldots, rS_n)^*F(rS_1,\ldots, rS_n)=\Theta(rS_1,\ldots,
rS_n)^* \Theta(rS_1,\ldots, rS_n) G(rS_1,\ldots, rS_n).
\end{equation*}
Taking the  SOT-limit in  this equality and using the fact that
$\|\Theta(rS_1,\ldots, rS_n)\|\leq 1$ and $\|F(rS_1,\ldots,
rS_n)\|\leq 1$, we deduce that $\widetilde \Theta^* \widetilde
F=\widetilde G$.  Now,  due to  relation \eqref{wide}, we have
$$
\widetilde \Theta^* \widetilde F=\lim_{r\to 1}\Theta(rS_1,\ldots,
rS_n)^* F(rS_1,\ldots, rS_n) \ \text{ and } \ \lim_{r\to
1}\Theta(rS_1,\ldots, rS_n)^* \Theta(rS_1,\ldots, rS_n)=I
$$
in the operator norm. Consequently, since
\begin{equation*}
\begin{split}
\|\widetilde G- G(rS_1,\ldots, rS_n)\|&\leq \|\widetilde G-
\Theta(rS_1,\ldots, rS_n)^*  F(rS_1,\ldots, rS_n)\|\\
&\quad +\|\Theta (rS_1,\ldots, rS_n)^* \Theta (rS_1,\ldots,
rS_n)-I\|\|G(rS_1,\ldots, rS_n)\|
\end{split}
\end{equation*}
and $\|G(rS_1,\ldots, rS_n)\|\leq 1$, we deduce that $ \widetilde
G=\lim_{r\to 1}  G(rS_1,\ldots, rS_n)$ in the operator norm
topology.  Hence,  we deduce that $\widetilde G\in \cA_n \bar
\otimes_{min}B(\cE, \cY)$. The proof is complete.

If in addition, $F$ is inner, then relation $\widetilde F=\widetilde
\Theta \tilde G$ implies
$$
I=\widetilde F^* \widetilde F=\widetilde G^* \widetilde \Theta^*
\widetilde \Theta \widetilde G=\widetilde G^* \widetilde  G,
$$
which proves that $G$ is inner.
\end{proof}

We remark that the second part of Theorem \ref{Schwarz} holds also
under the hypothesis of Proposition \ref{disc}.

 In \cite{Po-holomorphic}, \cite{Po-automorphism}, and
\cite{Po-unitary}, we obtained analogues of Schwarz lemma   for free
holomorphic functions.  We mention the following. Let
$F(X)=\sum_{\alpha\in \FF_n^+} X_\alpha\otimes A_{(\alpha)}$, \
$A_{(\alpha)}\in B(\cE, \cG)$,  be a free holomorphic function on
$[B(\cH)^n]_1$ with $\|F\|_\infty\leq 1$ and $F(0)=0$. Then $$
\|F(X)\| \leq \|X\|,\qquad \text{ for any } \ X\in [B(\cH)^n]_1.
$$

Note that Theorem \ref{Schwarz} can  be seen as a generalization of
Schwarz's lemma. Let us consider a few important  particular  cases.

\begin{corollary}\label{SCH}
If $F:[B(\cH)^n]_{1}\to
  B( \cH)\bar\otimes_{min} B(\cE, \cG)$  is a bounded
   free
holomorphic function with \ $\|F\|_\infty\leq 1$ and  representation
$$F(X_1,\ldots, X_n)=\sum_{k=m}^\infty \sum_{|\alpha|=k}
X_\alpha\otimes A_{(\alpha)},
$$
 where $m=1,2,\ldots$, then
$$
F(X_1,\ldots, X_n) F(X_1,\ldots, X_n)^*\leq \sum_{|\beta|=m} X_\beta
X_\beta^*\otimes I_\cG
$$
for any $X:=(X_1,\ldots, X_n)\in [B(\cH)^n]_1$.
 If, in addition,
$\left\|\sum_{|\beta|=m} A_{(\beta)} A_{\beta)}^*\right\|<1$, then
the inequality  above is strict for any $X\neq 0$.
\end{corollary}
\begin{proof}
As in the the proof of Theorem 3.4 from \cite{Po-holomorphic}, we
 have the Gleason  type  factorization $F=\Theta G$, where $\Theta$ and $G$ are free
holomorphic functions given by
$$\Theta (X_1,\ldots, X_n):=[X_\beta\otimes I_\cG:\ |\beta|=m]\quad
\text{ and } \quad G(X_1,\ldots, X_n)=\left[\begin{matrix}
\Phi_{(\beta)}(X_1,\ldots, X_n)\\
:\\|\beta|=m\end{matrix}\right].
$$
Due to Section 3 (see Example \ref{EXAMP}),  $\Theta$ is inner and
has the radial infimum property. Applying now Theorem \ref{Schwarz},
we deduce that
$$ F(X_1,\ldots, X_n)F(X_1,\ldots, X_n)^*\leq \Theta(X_1,\ldots, X_n)\Theta(X_1,\ldots, X_n)^*=
\sum_{|\beta|=m} X_\beta X_\beta^*\otimes I_\cG
$$
for any $(X_1,\ldots, X_n)\in [B(\cH)^n]_1$. On the other hand,
since
$$
\|G(0)\|=\left\|\left[\begin{matrix}
A_{(\beta)}\\
:\\|\beta|=m\end{matrix}\right]\right\|=\left\|\sum_{|\beta|=m}
A_{(\beta)} A_{\beta)}^*\right\|<1,
$$
we can use the second part of Theorem \ref{Schwarz}, to complete the
proof.
\end{proof}

We remark that  Corollary \ref{SCH} implies  the version of Schwarz
lemma obtained in \cite{Po-holomorphic} and, when $m=1$, the
corresponding result from \cite{HKMS}.

\begin{corollary} \label{Sch-strong} If   $F:[B(\cH)^n]_{1}\to
  B( \cH)\bar\otimes_{min} B(\cE, \cG)$  is  a bounded
   free
holomorphic function
 with \ $\|F\|_\infty\leq 1$ and $\|F(0)\|<1$, then

\begin{equation*}
\begin{split}
D_{F(0)^*} [I-F(X)F(0)^*]^{-1} [I-F(X)F(X)^*]&[I-F(0)F(X)^*]^{-1}
D_{F(0)^*}\\
&\geq \left(I-\sum_{|\alpha|=1} X_\alpha X_\alpha^*\right)\otimes
I_\cG
\end{split}
\end{equation*}
for any $X:=(X_1,\ldots, X_n)\in [B(\cH)^n]_1$.
\end{corollary}

\begin{proof} According to Theorem \ref{frac-trans}, the mapping
$G:=\Psi_{F(0)}[F]$ is a bounded free holomorphic function with
$\|G\|_\infty\leq 1$ and $G(0)=\Psi_{F(0)}[F(0)]=0$. Applying
Corollary \ref{SCH} to $G$ (when $m=1$), we deduce that
$G(X)G(X)^*\leq XX^*\otimes I_\cG$. On the other hand, using
relation \eqref{two-ident}, we have
$$
I-G(X)G(X)^*=D_{F(0)^*} [I-F(X)F(0)^*]^{-1}
[I-F(X)F(X)^*][I-F(0)F(X)^*]^{-1} D_{F(0)^*}.
$$
Now, one can easily complete the proof.
\end{proof}
We remark that  Corollary \ref{Sch-strong} can be seen as an
extension on Corollary \ref{SCH} (case $m=1)$ to the case when
$\|F(0)\|<1$.

When   dealing with free holomorphic functions with scalar
coefficients, Theorem \ref{Schwarz} can be improved, as follows.
\begin{theorem} \label{Schwarz2}
 Let $f, \theta$, and $ g$ be free holomorphic functions
on $[B(\cH)^n]_1$ with scalar coefficients   such that:
\begin{enumerate}
\item[(i)] $f(X)=\theta (X) g(X),\quad X\in [B(\cH)^n]_1$;
 \item[(ii)] $f$ is
bounded with $\|f\|_\infty\leq 1$;
 \item[(iii)] $\theta$  has  the
radial infimum property and $\|\theta\|_\infty=1$.
\end{enumerate}
 Then $\|g\|_\infty\leq 1$ and, consequently,
\begin{equation*}
 f(X)f(X)^*\leq \theta(X)\theta(X)^*, \qquad X\in [B(\cH)^n]_1,
\end{equation*}

\begin{equation*}
 \|f(X)\|\leq \|\theta(X)\|, \qquad X\in [B(\cH)^n]_1,
\end{equation*}
and  $$\|g(0)\|\leq 1.$$

Moreover,
\begin{enumerate}
\item[(a)] $f(X_0)f(X_0)^*<\theta(X_0)\theta(X_0)^*$ for some
$X_0\in [B(\cH)^n]_1$ if and only if $\theta(X_0)\theta(X_0)^*>0$
and $g$ is not a constant $c$ with $|c|=1$.

\item[(b)] $\|f(X_0)\|=\|\theta(X_0)\|$ for some
$X_0\in [B(\cH)^n]_1$ if and only if either $\theta(X_0)=0$ or
$f=c\theta$ for some constant $c$ with $|c|=1$.

\item[(c)] If $|g(0)|=1$, then $f=c\theta$ for some constant $c$ with $|c|=1$.
\end{enumerate}
\end{theorem}

\begin{proof}
Due to Proposition \ref{strict}, if $g:[B(\cH)^n]_1 \to B(\cH)$  is
a non-constant   free holomorphic function with $\|g\|_\infty\leq
1$, then $\|g(X)\|<1$ for any $X\in[B(\cH)^n]_1$. Using this result
and Theorem \ref{Schwarz}, in the particular case when
$\cE=\cG=\cY=\CC$,  one can complete the proof.
\end{proof}

 We remark that in the particular case when $n=1$ and
 $\theta(z)=z$, we recover   Schwarz's lemma (see
 \cite{Co}).

In what follows  we obtain generalizations of some classical results
from complex analysis  for certain  classes of free holomorphic
functions with positive real parts and of the form $F=I+\Theta
\Gamma$.

\begin{theorem}  \label{FTT}
Let $F$, $\Theta$, and $\Gamma$ be free holomorphic functions
on $[B(\cH)^n]_1$ with coefficients in $B(\cE)$, $B(\cG, \cE)$, and
$B(\cE,\cG)$, respectively.  If
\begin{enumerate}
\item[(i)]
$\Re F\geq 0$,
\item[(ii)] $\Theta $  has the radial infimum property,
$\|\Theta\|_\infty= 1$, and $\|\Theta(0)\|<1$,
    \item[(iii)]
$ F=I+\Theta \Gamma $,
\end{enumerate}
then
\begin{equation*}
 [I-F(X)][I-F(X)^*]\leq [I+F(X)]
\Theta(X)\Theta(X)^*[I+F(X)^*]
\end{equation*}
and
$$
  \|F(X)\|\leq
\frac{1+\|\Theta(X)\|}{1-\|\Theta(X)\|}
$$
for any $X\in [B(\cH)^n]_1$.
\end{theorem}

\begin{proof}
Since $\Re F(X)\geq 0$, $X\in [B(\cH)^n]_1$, its noncommutative
Cayley transform $G:=(F-I)(I+F)^{-1}$ is in the unit ball of
$H^\infty_{\bf ball}(B(\cE)$, thus $\|G(X)\|\leq 1$. Due to item
(iii), we have $G=\Theta \Gamma (I+F)^{-1}$. Now, since $\Theta $
has the radial infimum property and  $\|\Theta\|_\infty= 1$, we can
apply Theorem \ref{Schwarz} to $G$ and obtain
 $G(X)G(X)^*\leq \Theta(X)\Theta(X)^*$ for all  $X\in [B(\cH)^n]_1$.
 Hence, we deduce that
 $$[I+F(X)]^{-1} [F(X)-I][F(X)^*-I][I+F(X)^*]^{-1}\leq \Theta(X)
 \Theta(X)^*,
 $$
 which is equivalent to
\begin{equation*}
 [I-F(X)][I-F(X)^*]\leq [I+F(X)]
\Theta(X)\Theta(X)^*[I+F(X)^*].
\end{equation*}
The latter inequality implies
$$
\|F(X)\|-1\leq \|I-F(X)\|\leq \|\Theta(X)\|(1+\|F(X)\|),
$$
which leads to
\begin{equation}
\label{F<F}
 \|F(X)\|(1-\|\Theta(X)\|)\leq 1+\|\Theta(X)\|.
\end{equation}
Since $\|\Theta\|_\infty= 1$ and $\|\Theta(0)\|<1$, the maximum
principle for  free holomorphic functions with operator-valued
coefficients (see Theorem \ref{maxmod}) implies that
$\|\Theta(X)\|<1$. Now, inequality  \eqref{F<F} implies the desired
inequality.
\end{proof}

Taking into account  Theorem \ref{Schwarz2} and Theorem \ref{FTT},
we can make the following observation.

 \begin{remark} In the scalar case,  when  $\cE=\cG=\CC$,   the first
 inequality in Theorem \ref{FTT} is strict if and only if
 $\Theta(X)\Theta(X)^*>0$ and $F$ is not of the form
 $F=(I+\eta \Theta)(I-\eta \Theta)^{-1}$ for some constant $\eta$
 with $|\eta|=1$.
 \end{remark}

Consider the set   $\left[H_{\bf ball}^\infty(B(\cE)\right]_{<1}$
(resp. $ H_{\bf ball}^{>0}(B(\cE)$) of all bounded free holomorphic
functions on $[B(\cH)^n]_1$ with coefficients in $B(\cE)$ such that
$\|F(X)\|<1$ (resp. $\Re F(X)>0$) for any $X\in [B(\cH)^n]_1$. We
remark that the restriction to   $ H_{\bf ball}^{>0}(B(\cE)$ of the
noncommutative Cayley transform, defined by
$$\text{\boldmath{$\cC$}} [F] := [F -1][1+ F ]^{-1},
$$
  is a bijection
$\text{\boldmath{$\cC$}}:    H_{\bf ball}^{>0}(B(\cE)\to
 \left[H_{\bf ball}^\infty(B(\cE)\right]_{<1}$ and
  $\text{\boldmath{$\cC$}}^{-1}[G] =[I+G ][I-G
]^{-1}.$ Indeed, taking into account Theorem 1.5 from
\cite{Po-free-hol-interp}, it is enough to show that  $F\in  H_{\bf
ball}^{>0}(B(\cE)$ if and only if $G \in \left[H_{\bf
ball}^\infty(B(\cE)\right]_{<1}$, where  $F=[I+G ][I-G ]^{-1}.$ To
this end, note that
\begin{equation*}
\begin{split}
2\Re &F(rS_1,\ldots, rS_n)\\
&= [I-G(rS_1,\ldots, rS_n)^*]^{-1}\left[I-G(rS_1,\ldots,
rS_n)^*G(rS_1,\ldots, rS_n)\right][I-G(rS_1,\ldots, rS_n)]^{-1}
\end{split}
\end{equation*}
for any $r\in[0,1)$. Consequently, $\Re F(rS_1,\ldots, rS_n)>0$ for
any $r\in[0,1)$ if and only if $\|G(rS_1,\ldots, rS_n)\|<1$ for any
$r\in[0,1)$. Using the noncommutative  Poisson transform, we deduce
that $\Re F(X)>0$ for any $X\in [B(\cH)^n]_1$ if and only if
$\|G(X)\|<1$ for any $X\in [B(\cH)^n]_1$, which proves our
assertion.

In what follows we need the following result.

\begin{lemma}\label{inverse}
Let $F:[B(\cH)^n]_1 \to B(\cH)\bar \otimes_{min} B(\cE)$ be a free
holomorphic function with coefficients in $B(\cE)$. Then there is a
free holomorphic function $G:[B(\cH)^n]_1\to B(\cH)\bar
\otimes_{min} B(\cE)$ such that
$$
F(X)G(X)=G(X)F(X)=I,\qquad X\in [B(\cH)^n]_1,
$$
if and only if $F(rS_1,\ldots, rS_n)$ is an invertible operator for
any $r\in [0,1)$. Moreover, in this case,
$$G(rS_1,\ldots,
rS_n)=F(rS_1,\ldots, rS_n)^{-1}, \qquad r\in [0,1),
$$
where $S_1,\ldots, S_n$ are the left creation operators.
\end{lemma}

\begin{proof}
One implication is obvious. Assume that $F(rS_1,\ldots, rS_n)$ is an
invertible operator for any $r\in [0,1)$.  First we prove that
$F(rS_1,\ldots, rS_n)^{-1}$ is in $F_n^\infty\bar \otimes B(\cE)$,
the weakly closed algebra generated by the spatial tensor product.
Since $F$ is  a free holomorphic function, $F(rS_1,\ldots, rS_n)$ is
in $\cA_n\bar \otimes_{min} B(\cE)$ for any $r\in [0,1)$.  In
particular, we have
$$
F(rS_1,\ldots, rS_n) (R_i\otimes I)=(R_i\otimes I) F(rS_1,\ldots,
rS_n),\qquad i=1,\ldots, n,
$$
where $R_1,\ldots, R_n$ are the right creation operators.  Hence, we
deduce that
$$
(R_i\otimes I) F(rS_1,\ldots, rS_n)^{-1}=F(rS_1,\ldots, rS_n)^{-1}
(R_i\otimes I)\quad i=1,\ldots, n.
$$
According to \cite{Po-analytic},  the commutant of $\{R_i\otimes I,\
i=1,\ldots, n\}$ is equal to $F_n^\infty\bar \otimes B(\cE)$.
Consequently,  $F(rS_1,\ldots, rS_n)^{-1}$ is in $F_n^\infty\bar
\otimes B(\cE)$ and has a unique Fourier representation
\begin{equation*}
F(rS_1,\ldots, rS_n)^{-1}\sim\sum_{\alpha\in \FF_n^+}
S_\alpha\otimes r^{|\alpha|} B_{(\alpha)}(r)
\end{equation*}
for some operators $B_{(\alpha)}(r)\in B(\cE)$. We prove  now  that
the operators $B_{(\alpha)}(r)$, $\alpha\in \FF_n^+$, don't depend
on $r\in [0,1)$. Assume that $F$ has the representation
$$F(X_1,\ldots, X_n):=\sum_{k=0}^\infty\sum_{|\alpha|=k}
X_\alpha\otimes A_{\alpha)},\qquad (X_1,\ldots, X_n)\in
[B(\cH)^n]_1,
$$
where the convergence is in the operator norm topology. Since
$$I=F(rS_1,\ldots, rS_n)F(rS_1,\ldots,
rS_n)^{-1}=\left(\sum_{\alpha\in \FF_n^+} S_\alpha\otimes
r^{|\alpha|} B_{(\alpha)}(r)
\right)\left(\sum_{k=0}^\infty\sum_{|\alpha|=k} r^{|\alpha|}
S_\alpha\otimes A_{\alpha)}\right)
$$
for any $\beta\in \FF_n^+$, we have
$$
\left< (S_\beta^*\otimes I_\cE)F(rS_1,\ldots, rS_n)F(rS_1,\ldots,
rS_n)^{-1} (1\otimes x), (1\otimes y)\right>= \sum_{\alpha,\omega\in
\FF_n^+, \alpha \omega=\beta} r^{|\beta|}\left<A_{(\alpha)}
B_{(\omega)}(r)x,y\right>
$$
for any $x,y\in \cE$. Therefore, $A_{(0)} B_{(0)}(r)=I$ and
\begin{equation}
\label{al-be}\sum_{\alpha,\omega\in \FF_n^+, \alpha \omega=\beta}
A_{(\alpha)} B_{(\omega)}(r)=0 \end{equation}
 if $|\beta|\geq 1$. Now, we proceed by induction. Note
that $B_{(0)}(r)=A_{(0)}^{-1}$ and assume that the operators
$B_{(\alpha)}(r)$ don't depend on $r\in [0,1)$ for any  $\alpha\in
\FF_n^+$ with $|\beta|\leq m$. We prove  that the property holds if
$|\beta|=m+1$. To this end, let $\beta:=g_{i_1} g_{i_2}\cdots
g_{i_m} g_{i_{m+1}}\in \FF_n^+$. Due to relation \eqref{al-be}, we
have
$$
A_{(0)}B_{(\beta)}+ A_{(g_{i_1})}B_{(g_{i_2}\cdots
g_{i_{m+1}})}+\cdots + A_{(g_{i_1}\cdots g_{i_{m}})} B_
{(g_{i_{m+1}})}+ A_{(\beta)} B_{(0)}=0.
$$
Hence and due to the induction hypothesis, we deduce that
$B_{(\beta)}(r)$ does not depend on $r\in [0,1)$. Thus  we can write
$B_{(\beta)}:=B_{(\beta)}(r)$ for any $\beta\in \FF_n^+$ and $
F(rS_1,\ldots, rS_n)^{-1}$ has the Fourier representation
$\sum_{\alpha\in \FF_n^+} S_\alpha\otimes r^{|\alpha|} B_{(\alpha)}
$ and the series $\sum_{k=0}^\infty\sum_{|\alpha|=k}
(sr)^{|\alpha|}S_\alpha\otimes
 B_{(\alpha)}$ converges in the operator norm topology for any
 $s,r\in [0,1)$.  Hence, we
 deduce that the map $G:[B(\cH)^n]_1\to B(\cH) \bar \otimes_{min}
B(\cE)$ defined by
$$
G(X_1,\ldots, X_n):=\sum_{k=0}^\infty\sum_{|\alpha|=k}
X_\alpha\otimes
 B_{(\alpha)},\qquad (X_1,\ldots, X_n)\in [B(\cH)^n]_1
 $$
is a free holomorphic function. Here the convergence is in the
operator norm topology. Due to    \eqref{al-be} and the similar
relation that can be deduced from the equation $F(rS_1,\ldots,
rS_n)^{-1}F(rS_1,\ldots, rS_n)=I$, one can easily see that
$F(X)G(X)=G(X)F(X)=I$ for any $X\in [B(\cH)^n]_1$. Moreover, we have
$$G(rS_1,\ldots, rS_n)=F(rS_1,\ldots, rS_n)^{-1}
$$
for any $r\in [0,1)$. The proof is complete.
\end{proof}

 The next result is a noncommutative extension of Harnack's double
 inequality.

\begin{theorem}\label{Har2} Let $F$, $\Theta$, and $\Gamma$ be free holomorphic functions
on $[B(\cH)^n]_1$ with coefficients in $B(\cE)$, $B(\cG, \cE)$, and
$B(\cE,\cG)$, respectively.  If
\begin{enumerate}
\item[(i)]
$\Re F>0$,
\item[(ii)] $\Theta $  has the radial infimum property,
$\|\Theta\|_\infty\leq 1$, and $\|\Theta(0)\|<1$,
    \item[(iii)]
$ F=I+\Theta \Gamma $  and $F^{-1}=I+\Theta L$ for some free
holomorphic function $L$ on $[B(\cH)^n]_1$,
\end{enumerate}
then
$$
\frac{1-\|\Theta(X)\|}{1+\|\Theta(X)\|}\leq \|F(X)\|\leq
\frac{1+\|\Theta(X)\|}{1-\|\Theta(X)\|}
$$
for any $X\in [B(\cH)^n]_1$.
\end{theorem}

\begin{proof} The inequality $ \|F(X)\|\leq
\frac{1+\|\Theta(X)\|}{1-\|\Theta(X)\|}$ is due to Theorem
\ref{FTT}. We prove  now the first inequality. Since $\Re F(X)>0$,
$X\in [B(\cH)^n]_1$, there exit constants $\gamma(r)\in (0,1)$ such
that
$$\Re F(rS_1,\ldots, rS_n) \geq \gamma(r)I,\quad r\in (0,1).
$$
Hence, we deduce that
$$
\|F(rS_1,\ldots, rS_n)^*x\|+\|F(rS_1,\ldots, rS_n)x\|\geq
2\gamma(r)\|x\|,\qquad x\in F^2(H_n)\otimes \cE),
$$
which shows that $F(rS_1,\ldots, rS_n)$ and $F(rS_1,\ldots, rS_n)^*$
are bounded below. Therefore,  the operator $F(rS_1,\ldots, rS_n)$
is invertible for all $r\in [0,1)$. Due to Lemma \ref{inverse},
there is a free holomorphic function $\Lambda:[B(\cH)^n]_1\to B(\cH)
\bar \otimes_{min} B(\cE)$ such that
$$
F(X)\Lambda(X)=\Lambda(X)F(X)=I,\qquad X\in [B(\cH)^n]_1,
$$
and $\Lambda(rS_1,\ldots, rS_n)=F(rS_1,\ldots, rS_n)^{-1}$ for all $
r\in [0,1)$. Since
$$
\Re \Lambda(rS_1,\ldots, rS_n)=[F(rS_1,\ldots, rS_n)^{-1}]^*[\Re
F(rS_1,\ldots, rS_n)]F(rS_1,\ldots, rS_n)^{-1}
$$
and $\Re F(rS_1,\ldots, rS_n)>0$, we deduce that $\Re
\Lambda(rS_1,\ldots, rS_n)>0$. Therefore $\Re \Lambda>0$. Due to
item (iii), we have $\Lambda=I+\Theta L$ for some free holomorphic
function $L$ on $[B(\cH)^n]_1$. Applying now Theorem \ref{FTT} to
$\Lambda$, we obtain
$$
\|\Lambda(X)\|\leq \frac{1+\|\Theta(X)\|}{1-\|\Theta(X)\|}
$$
for any $X\in [B(\cH)^n]_1$. Since $\Lambda(X)=F(X)^{-1}$, we have
$\|F(X)\|\geq \frac{1}{\|\Lambda(X)\|}$. Combining these
inequalities, we deduce that
$$
\frac{1-\|\Theta(X)\|}{1+\|\Theta(X)\|}\leq \|F(X)\|
$$
for any $X\in [B(\cH)^n]_1$. The proof is complete.
\end{proof}

We remark that when $n=1$ and $\cE=\cG=\CC$, then  the condition
$F^{-1}=I+\Theta L$ in Theorem \ref{Har2} is redundant, so  we can
drop it.

\begin{corollary}\label{Har1} Let $F$ be  a free holomorphic function
on $[B(\cH)^n]_1$   with coefficients in $B(\cE)$ and  standard
representation
$$
F(X_1,\ldots,X_n)=I+\sum_{k=m}^\infty \sum_{|\alpha|=k}
X_\alpha\otimes A_{(\alpha)},\quad (X_1,\ldots,X_n)\in [B(\cH)^n]_1,
$$
where $m=1,2\ldots$. If \ $\Re F\geq 0$, then
$$
\frac{1-\|\sum_{|\beta|=m}X_\beta
X_\beta^*\|^{1/2}}{1+\|\sum_{|\beta|=m}X_\beta
X_\beta^*\|^{1/2}}\leq \|F(X)\|\leq
\frac{1+\|\sum_{|\beta|=m}X_\beta
X_\beta^*\|^{1/2}}{1-\|\sum_{|\beta|=m}X_\beta X_\beta^*\|^{1/2}}
$$
for any $X\in [B(\cH)^n]_1$.

\end{corollary}

\begin{proof} First, we consider the case when $\Re F>0$.
 As in the proof of Theorem 2.4 from
\cite{Po-holomorphic}, we have a   decomposition $ F=I+\Theta
\Gamma$, where
$$
\Theta(X_1,\ldots, X_n):=[X_\beta\otimes I :\ |\beta|=m]\quad \text{
and }
\quad \Gamma(X_1,\ldots, X_n):=\left[\begin{matrix} \Phi_{(\beta)}(X_1,\ldots, X_n)\\
:\\
|\beta|=m
\end{matrix} \right]
$$
 are free holomorphic functions. Due to Section 3 (see Example
 \ref{EXAMP}), $\Theta$ is inner and has the radial infimum
 property.
 On the other hand, due to the proof of  Theorem \ref{Har2}, $X\mapsto F(X)^{-1}$ exists as a
 free holomorphic function on $[B(\cH)^n]_1$. Since
 $$I=F(X)^{-1}  F(X) =F(X)^{-1} (I+\Theta(X)
 \Gamma(X))
 $$
 we deduce that  $F(X)^{-1}=I-F(X)^{-1} \Theta(X) \Gamma(X)$. Taking
 into account that $\Theta$ is a homogeneous polynomial of degree $m$,
 it is easy to see that $X\mapsto F(X)^{-1} \Theta(X) \Gamma(X)$ is
 a free holomorphic function so that each monomial in its standard
 representation has degree greater than  or equal to $m$. This
 implies that $F^{-1}$ has a decomposition of the form $I+\Theta L$
 for some free holomorphic function $L$. Since we are under the
 hypotheses of
 Theorem \ref{Har2},  we can apply  this theorem  to $F$
 and obtain the desired inequalities.
 In the case when $\Re F\geq 0$, the map
 $G_\epsilon:=I+\frac{1}{1+\epsilon} (F-I)$, $\epsilon>0$, has the
 property  $\Re G_\epsilon>0$,  so that we can  use the first part of
 the proof and obtain the corresponding inequalities. Taking the
 limit as $\epsilon \to 0$, we complete the proof.
\end{proof}

From Corollary \ref{Har1}, we can deduce the following remarkable
particular case, which should be compared to Theorem 1.4 from
\cite{Po-hyperbolic}.

\begin{corollary}\label{har1}  If $F$ is   a free holomorphic function
on $[B(\cH)^n]_1$ with coefficients in $B(\cE)$ such that $F(0)=I$
and  $\Re F\geq 0$, then
$$
\frac{1-\|X\|}{1+\|X\|}\leq \|F(X)\|\leq \frac{1+\|X\|}{1-\|X\|},
\qquad X\in [B(\cH)^n]_1.
$$
\end{corollary}

\bigskip

\section{ Noncommutative Borel-Carath\'eodory theorems }

In this section, we obtain Borel-Carath\'eodory type results for
free holomorhic functions with operator-valued coefficients.

We start with  a Carath\'eodory type result for free holomorhic
functions which admit  factorizations  $F=\Theta \Gamma$, where
$\Theta$  is an inner function with the radial infimum property.

\begin{theorem}\label{Cara}
Let $F$, $\Theta$, and $\Gamma$ be free holomorphic functions on
$[B(\cH)^n]_1$ with coefficients in $B(\cE)$, $B(\cG, \cE)$, and
$B(\cE, \cG)$, respectively.  If

\begin{enumerate}
\item[(i)] $\Re F(X)\leq I$ for any $X\in [B(\cH)^n]_1$,
\item[(ii)] $\Theta $ has
     the radial infimum property, $\|\Theta\|_\infty = 1$, and
     $\|\Theta(0)\|<1$,
    \item[(iii)] $F=\Theta \Gamma$,
\end{enumerate}
     then
    $$
    \|F(X)\|\leq \frac{2\|\Theta(X)\|}{1- \|\Theta(X)\|},\qquad X\in
    [B(\cH)^n]_1.
    $$
    \end{theorem}

\begin{proof} Since $G:=I-F=I-\Theta \Gamma$ has the property that
$\Re G\geq 0$, we can apply Theorem \ref{FTT} to $G$ and obtain
\begin{equation*}
 [I-G(X)][I-G(X)^*]\leq [I+G(X)]
\Theta(X)\Theta(X)^*[I+G(X)^*],\qquad X\in [B(\cH)^n]_1.
\end{equation*}
Consequently, we  deduce that
\begin{equation*}
\begin{split}
\|F(X)\|&=\|I-G(X)\|\leq \|(G(X)+ I\|
\|\Theta(X)\|\\
&=\|2I-F(X)\|\|\Theta(X)\| \leq (2+\|F(X)\|) \|\Theta(X)\|.
\end{split}
\end{equation*}
Hence, we have
$$
\|F(X)\|(1-\|\Theta (X)\|)\leq 2\|\Theta (X)\|.
$$
We recall that, since $\|\Theta\|_\infty= 1$ and $\|\Theta(0)\|<1$,
the maximum principle for  free holomorphic functions with
operator-valued coefficients (see Theorem \ref{maxmod}) implies that
$\|\Theta(X)\|<1$. Now we can complete the proof.
\end{proof}

\begin{corollary}\label{Har3} Let $F$ be  a free holomorphic function
on $[B(\cH)^n]_1$   with coefficients in $B(\cE)$ and  standard
representation
$$
F(X_1,\ldots,X_n)=\sum_{k=m}^\infty \sum_{|\alpha|=k}
X_\alpha\otimes A_{(\alpha)},\quad (X_1,\ldots,X_n)\in [B(\cH)^n]_1,
$$
where $m=1,2\ldots$. If \  $\Re F\leq I$, then
$$
  \|F(X)\|\leq
\frac{2\|\sum_{|\beta|=m}X_\beta
X_\beta^*\|^{1/2}}{1-\|\sum_{|\beta|=m}X_\beta X_\beta^*\|^{1/2}},
\qquad X\in [B(\cH)^n]_1.
$$
\end{corollary}

\begin{proof}  As in the proof of Theorem 2.4 from
\cite{Po-holomorphic}, we have a  Gleason type decomposition $
F=\Theta \Gamma$, where
$$
\Theta(X_1,\ldots, X_n):=[X_\beta\otimes I :\ |\beta|=m]\quad \text{
and }
\quad \Gamma(X_1,\ldots, X_n):=\left[\begin{matrix} \Phi_{(\beta)}(X_1,\ldots, X_n)\\
:\\
|\beta|=m
\end{matrix} \right]
$$
 are free holomorphic functions.
 Since $\Theta$ is inner with the radial infimum
 property and $\Theta(0)=0$, we apply Theorem \ref{Cara} and complete the proof.
\end{proof}

From Corollary \ref{Har3}, we can deduce the following particular
case.

\begin{corollary}\label{Car1}  If $F$ is   a free holomorphic function
on $[B(\cH)^n]_1$ with coefficients in $B(\cE)$ such that $F(0)=0$
and  $\Re F\leq I$, then
$$
\|F(X)\|\leq \frac{2\|X\|}{1-\|X\|},\qquad X\in [B(\cH)^n]_1.
$$
\end{corollary}

The next result is a generalization of the Borel-Carath\' eodory
theorem, mentioned in the introduction, for free holomorphic
functions with operator-valued coefficients.

\begin{theorem}
\label{borel-cara} Let $F:[B(\cH)^n]_\gamma^-\to
B(\cH)\bar\otimes_{min} B(\cE)$ be a free holomorphic function with
coefficients in $B(\cE)$ and let $r\in (0,\gamma)$. Then
$$
\sup_{\|X\|=r}\|F(X)\|\leq
\frac{2r}{\gamma-r}A(\gamma)+\frac{\gamma+r}{\gamma-r}\|F(0)\|,
$$
where $A(\gamma):=\sup_{\|y\|=1}\left<\Re F(\gamma S_1,\ldots,
\gamma S_n)y,y\right>$  and $S_1,\ldots, S_n$ are the left creation
operators.
\end{theorem}

\begin{proof}
If $F$ is constant, i.e., $F=F(0)$, then the inequality holds due to
the fact that $$\Re F(0)\geq -\|F(0)\| I_{\cH\otimes \cE}. $$
 Assume
that $F$  is not constant and $F(0)=0$. First we show that
\begin{equation}
\label{A} A(\gamma)>0.
\end{equation}
Indeed, if we assume that $A(\gamma)\leq 0$, then
$$
\Re F(\gamma S_1,\ldots, \gamma S_n)\leq \Re F(0)=0.
$$
Applying the noncommutative Poisson transform at $[\frac{t}{\gamma}
X_1,\ldots,\frac{t}{\gamma} X_n]$, where $0\leq t<\gamma$ and
$(X_1,\ldots, X_n)\in [B(\cH)^n]_1^-$, we obtain
$$
\Re F(t X_1,\ldots, t X_n)=P_{[\frac{t}{\gamma}
X_1,\ldots,\frac{t}{\gamma} X_n]}\Re F(\gamma S_1,\ldots, \gamma
S_n) \leq \Re F(0)=0
$$
for any $t\in[0,\gamma)$ and $(X_1,\ldots, X_n)\in [B(\cH)^n]_1^-$.
According to Theorem 2.9 from \cite{Po-pluriharmonic}, we deduce
that $F=F(0)$, which contradicts our assumption. Therefore,
inequality \eqref{A} holds.

Since $\Re F(\gamma S_1,\ldots, \gamma S_n)\leq A(\gamma) I$, we can
use again the noncommutative Poisson transform to  deduce that $\Re
F(X)\leq A(\gamma) I$  for  $X\in [B(\cH)^n]_\gamma^-$. Now, let
$\epsilon>0$ and define the free holomorphic function on a
noncommutative ball $[B(\cH)^n]_s$ with $s>\gamma$,  by
$$
\varphi_\epsilon(X):=2[A(\gamma)+\epsilon] I_{\cH\otimes
\cE}-F(X),\qquad X\in[B(\cH)^n]_s.
$$
Note that, for any $y\in \cH\otimes \cE$ and $X\in[B(\cH)^n]_s$, we
have
\begin{equation}\label{var}
\begin{split}
\|\varphi_\epsilon(X) y\|^2&= 4[A(\gamma)+\epsilon]^2\|y\|^2-
4[A(\gamma)+\epsilon] \left< \Re F(X)y,y\right>+\|F(X)y\|^2\\
&\geq 4\epsilon[A(\gamma)+\epsilon] \|y\|^2 +\|F(X)y\|^2\\
&\geq 4\epsilon[A(\gamma)+\epsilon] \|y\|^2.
\end{split}
\end{equation}
Similar calculations show that
$$
\|\varphi_\epsilon(X)^* y\|^2\geq 4\epsilon[A(\gamma)+\epsilon]
\|y\|^2.
$$
Replacing $X$  by $(tS_1, \ldots, tS_n)$, \  $t\leq s$, and taking
$y\in F^2(H_n)\otimes \cE$  in the inequalities above,  we deduce
that $\varphi_\epsilon (tS_1,\ldots, tS_n)$ and $\varphi_\epsilon
(tS_1,\ldots, tS_n)^*$ are bounded below and, consequently,
invertible for any $t\in [0,s)$.

Applying Lemma \ref{inverse} to  $\varphi_\epsilon$, we deduce that
 there   is  a free holomorphic
function  $\psi_\epsilon$ on $[B(\cH)^n]_s$ such that
$$\varphi_\epsilon(X)
\psi_\epsilon(X)=\psi_\epsilon(X)\varphi_\epsilon(X)=I,\qquad X\in
[B(\cH)^n]_s.
$$
Using  relation \eqref{var} and replacing $y$ with
$\psi_\epsilon(X)y$, we obtain that
$$
\|y\|^2=\|\varphi_\epsilon(X) \psi_\epsilon(X)y\|\geq
\|F(X)\psi(X)y\|^2+4\epsilon[A(\gamma)+\epsilon]\|\psi_\epsilon(X)y\|^2.
$$
Hence, we deduce that the map
\begin{equation}
\label{Gura} \Lambda_\epsilon (X):=F(X) \psi_\epsilon (X), \qquad
X\in [B(\cH)^n]_s,
\end{equation}
is a contractive free holomorphic function on $[B(\cH)^n]_s$. Since
$\Lambda_\epsilon (0)=0$, Theorem \ref{maxmod} implies that
$\|\Lambda_\epsilon (X)\|<1$. Hence,  and due to relation
\eqref{Gura}, we deduce that
\begin{equation}
\label{F(X)}
 F(X)=2[A(\gamma)+\epsilon] \Lambda_\epsilon
[I+\Lambda_\epsilon (X)]^{-1},
\end{equation}
which implies
\begin{equation*}
\begin{split}
\|F(X)\|&\leq  2[A(\gamma)+\epsilon] \|\Lambda_\epsilon(X)\|\left(
1+\|\Lambda_\epsilon(X)\|+\|\Lambda_\epsilon(X)\|^2+\cdots \right)\\
& \leq  2[A(\gamma)+\epsilon]\frac{ \|\Lambda_\epsilon(X)\|}
{1-\|\Lambda_\epsilon(X)\|}.
\end{split}
\end{equation*}

On the other hand, applying the Schwarz type lemma for free
holomorphic functions (see \cite{Po-holomorphic}) to
$\Lambda_\epsilon$, we deduce that
\begin{equation}
\label{Lae} \|\Lambda_\epsilon(X)\|\leq \frac{r}{\gamma}
\end{equation}
for any $X\in [B(\cH)^n]_\gamma$ with $\|X\|=r$, where
 $0\leq r<\gamma$. Combining this with the previous inequality, we
 obtain
 $$
 \|F(X)\|\leq \frac{2[A(\gamma)+\epsilon] r}{\gamma-r}.
 $$
 Taking $\epsilon\to 0$, we deduce that
 $$
 \sup_{\|X\|=r}\|F(X)\|\leq \frac{2 r}{\gamma-r}A(\gamma),
 $$
which proves the theorem when $F(0)=0$.

Now, we consider the  case when $F(0)\neq 0$.  Applying the  result
above to $F-F(0)$, we obtain
\begin{equation*}
\begin{split}
\sup_{\|X\|=r}\|F(X)-F(0)\|&\leq\frac{2 r}{\gamma-r}
\sup_{\|y\|=1}\left<\Re
(F(\gamma S_1,\ldots, \gamma S_n)-F(0))y,y\right>\\
&\leq \frac{2 r}{\gamma-r}[A(\gamma)+\|F(0)\|].
\end{split}
\end{equation*}
Consequently, we have
\begin{equation*}
\begin{split}
\sup_{\|X\|=r}\|F(X\|&\leq \sup_{\|X\|=r}\|F(X)-F(0)\|+\|F(0)\|\\
&=\frac{2r}{\gamma-r}A(\gamma)+\frac{\gamma+r}{\gamma-r}\|F(0)\|.
\end{split}
\end{equation*}
The proof is complete.
\end{proof}

We remark  that if $A(\gamma)\geq 0$ in Theorem \ref{borel-cara},
then we can deduce that
$$
\sup_{\|X\|=r}\|F(X)\|\leq
\frac{\gamma+r}{\gamma-r}[A(\gamma)+\|F(0)\|].
$$

A closer look at the proof of Theorem \ref{borel-cara} reveals
another Carath\'eodory type inequality. More precisely, applying
Corollary \ref{SCH} to the free holomorphic function
$\Lambda_\epsilon$, we deduce that
$$
\Lambda_\epsilon(X)\Lambda_\epsilon(X)^*\leq \frac{XX^*}{\gamma^2}
$$
for any $X\in [B(\cH)^n]_\gamma$ with $\|X\|=r$, where
 $0\leq r<\gamma$.
Using  now  relations \eqref{F(X)} and \eqref{Lae}, we obtain
\begin{equation*}
\begin{split}
F(X)F(X)^*&\leq  4[A(\gamma)+\epsilon]^2\frac{
\Lambda_\epsilon(X)\Lambda_\epsilon(X)^*}
{(1-\|\Lambda_\epsilon(X))^2\|}\\
&\leq \frac{4[A(\gamma)+\epsilon]^2}{(\gamma-r)^2}(XX^*\otimes
I_\cE).
\end{split}
\end{equation*}
Taking $\epsilon\to 0$, we deduce that $F(X)F(X)^*\leq
\frac{4A(\gamma)^2}{(\gamma-r)^2}(XX^*\otimes I_\cE)$. Now, in the
general case when  $F(0)$ is not necessarily $0$, we obtain the
following result

\begin{corollary} Under the hypotheses of Theorem \ref{borel-cara}, we
have
$$
[F(X)-F(0)][F(X)-F(0)]^*\leq
\frac{4A(\gamma)^2}{(\gamma-r)^2}\left(XX^*\otimes I_\cE\right)
$$
for any $X\in B(\cH)^n$ with $\|X\|=r$.
\end{corollary}

 \bigskip

\section{  Julia's lemma for holomorphic functions  on
 noncommutative balls
   }

In this section,   we provide a  noncommutative generalization  of
Pick's theorem for bounded free holomorphic functions. Using this
result  and   basic facts concerning  the involutive  free
holomorphic automorphisms of $[B(\cH)^n]_1$, we obtain
 a free analogue of Julia's lemma from complex analysis.

A map $F:[B(\cH)^n]_{1}\to [B(\cH)^n]_{1}$ is called free
biholomorphic  if $F$  is  free homolorphic, one-to-one and onto,
and  has  free holomorphic inverse. The automorphism group of
$[B(\cH)^n]_{1}$, denoted by $Aut([B(\cH)^n]_1)$, consists of all
free biholomorphic functions  of $[B(\cH)^n]_{1}$. It is clear that
$Aut([B(\cH)^n]_1)$ is  a group with respect to the composition of
free holomorphic functions.

Inspired by  the classical results of  Siegel \cite{Si} and Phillips
\cite{Ph} (see also \cite{Y}), we used,  in  \cite{Po-automorphism},
the theory of noncommutative characteristic functions for row
contractions (see \cite{Po-charact}) to find all the involutive free
holomorphic automorphisms of $[B(\cH)^n]_1$, which turn out to be of
the form
\begin{equation*}
 \Phi_\lambda(X_1,\ldots, X_n)=- \Theta_\lambda(X_1,\ldots, X_n), \qquad (X_1,\ldots,
 X_n)\in [B(\cH)^n]_1,
\end{equation*}
for some $\lambda=[\lambda_1,\ldots, \lambda_n]\in \BB_n$, where
$\Theta_\lambda$ is the characteristic function  of the row
contraction $\lambda$, acting as an operator from $\CC^n$  to $\CC$.

We recall that the characteristic function of the row contraction
$\lambda:=(\lambda_1,\ldots,\lambda_n)\in \BB_n$ is the boundary
function $\tilde{\Theta}_\lambda$, with respect to $R_1,\ldots,
R_n$, of the free holomorphic function
$\Theta_\lambda:[B(\cH)^n]_1\to [B(\cH)^n]_1$ given by
\begin{equation*}
  \Theta_\lambda(X_1,\ldots, X_n):=-{
\lambda}+\Delta_{ \lambda}\left(I_\cH-\sum_{i=1}^n \bar{{
\lambda}}_i X_i\right)^{-1} [X_1,\ldots, X_n] \Delta_{{\lambda}^*}
\end{equation*}
for $(X_1,\ldots, X_n)\in [B(\cH)^n]_1$, where
$\Delta_\lambda=(1-\|\lambda\|_2^2)^{1/2} I_\CC$ and
$\Delta_{\lambda^*}=(I_\cK-\lambda^*\lambda)^{1/2}$.  For
simplicity, we also used the notation $ {\lambda}:=[ {\lambda}_1
I_\cG,\ldots, {\lambda}_n I_\cG]$ for the row contraction acting
from $\cG^{(n)}$ to $\cG$, where $\cG$ is a Hilbert space.

In \cite{Po-automorphism}, we proved that if
  $\lambda:=(\lambda_1,\ldots, \lambda_n)\in
\BB_n\backslash \{0\}$ and  $\gamma:=\frac{1}{\|\lambda\|_2}$, then
$\Phi_\lambda:=-\Theta_\lambda$ is a free holomorphic function on
$[B(\cH)^n]_\gamma$ which has the following properties:
\begin{enumerate}
\item[(i)]
$\Phi_\lambda (0)=\lambda$ and $\Phi_\lambda(\lambda)=0$;
\item[(ii)] the identities
\begin{equation}\label{E1}
\begin{split}
I_{\cH}-\Phi_\lambda(X)\Phi_\lambda(Y)^*&= \Delta_{\lambda}(I- X
\lambda^*)^{-1}(I-  X   Y^*)(I-  \lambda  Y^*)^{-1}
\Delta_{\lambda},\\
I_{\cH\otimes \CC^n}-\Phi_\lambda(X)^*\Phi_\lambda(Y)&=
\Delta_{\lambda^*}(I-{X}^*\lambda)^{-1}(I-{X}^* Y)(I-{\lambda}^*
Y)^{-1} \Delta_{\lambda^*},
\end{split}
\end{equation}
hold  for all  $X$ and $Y$ in  $[B(\cH)^n]_\gamma$;

\item[(iii)] $\Phi_\lambda$ is an involution, i.e., $\Phi_\lambda(\Phi_\lambda(X))=X$
for any $X\in [B(\cH)^n]_\gamma$;
\item[(iv)] $\Phi_\lambda$ is a free holomorphic automorphism of the
noncommutative unit ball $[B(\cH)^n]_1$;
\item[(v)] $\Phi_\lambda$ is a homeomorphism of $[B(\cH)^n]_1^-$
onto $[B(\cH)^n]_1^-$.
\end{enumerate}

 Moreover, we
determined all the free holomorphic automorphisms of the
noncommutative ball $[B(\cH)^n]_1$ by  showing  that if $\Phi\in
Aut([B(\cH)^n]_1)$ and $\lambda:=\Phi^{-1}(0)$, then there is a
unitary operator $U$ on $\CC^n$ such that
$$
\Phi=\Phi_U\circ \Phi_\lambda,
$$
where $$  \Phi_U(X_1,\ldots X_n):=[X_1,\ldots, X_n]U , \qquad
(X_1,\ldots, X_n)\in [B(\cH)^n]_1.
$$

 The first result of this section is
following extension of Pick's theorem  (see
\cite{Pic},\cite{Cara2}), for bounded free holomorphic functions.
Let $M_{n\times m}$ be the set of all $n\times m$ matrices with
scalar coefficients.

\begin{theorem}
\label{Schw3} Let $F:[B(\cH)^n]_1\to [B(\cH)^m]_1^-$ be a free
holomorphic function  with $\|F(0)\|<1$ and  let  $a\in \BB_n$. Then
there exists a free holomorphic function $\Gamma:[B(\cH)^n]_1\to
B(\cH)\bar \otimes_{min} M_{n\times m}$ with $\|\Gamma\|_\infty \leq
1$ such that
$$
\Phi_{F(a)}(F(X))=\Phi_a(X) (\Gamma \circ\Phi_a) (X),\qquad X\in
[B(\cH)^n]_1,
$$
where $\Phi_a$ and $\Phi_{F(a)}$   are the corresponding free
holomorphic automorphisms of $[B(\cH)^n]_1$ and $[B(\cH)^m]_1$,
respectively. Consequently,
$$
 \Phi_{F(a)}(F(X))\Phi_{F(a)}(F(X))^*\leq  \Phi_a(X)\Phi_a(X)^*,\qquad X\in
[B(\cH)^n]_1,
$$
and
$$
\left\|\Phi_{F(a)}(F(X))\right\|\leq \|\Phi_a(X)\|,\qquad X\in
[B(\cH)^n]_1.
$$
\end{theorem}
\begin{proof}  Since $F$ is a free holomorphic function with
$\|F\|_\infty\leq 1$ and $\|F(0)\|<1$, Corollary \ref{not-strict}
implies  that $\|F(X)\|<1$ for any $X\in [B(\cH)^n]_1$.
 We know  that
$\Phi_a\in Aut([B(\cH)^n]_1)$ and $\Phi_{F(a)}\in
Aut([B(\cH)^m]_1)$. Due to  Section 1 and the properties of the free
holomorphic automorphisms of $[B(\cH)^m]_1$, the composition map
$G:=\Phi_{F(a)}\circ F\circ \Phi_a:[B(\cH)^n]_1\to [B(\cH)^m]_1$ is
a free holomorphic function with $G(0)=0$. Therefore, it has a
representation of the form
\begin{equation}
\label{iden-Gam} G(X_1,\ldots, X_n)=\sum_{k=1}^\infty
\sum_{|\alpha|=k} X_\alpha\otimes A_{(\alpha)}= [X_1,\ldots,
X_n]\Gamma(X_1,\ldots, X_n)
\end{equation}
  for any  $ [X_1,\ldots, X_n]\in [B(\cH)^n]_1$, for some matrices
  $A_{(\alpha)}\in M_{1\times m}$ and a free holomorphic
  function  $\Gamma$ with coefficients in $M_{n\times m}$.
  Since $\|G\|_\infty\leq 1$   with $G(0)<1$, and
  $\Theta(X):=[X_1,\ldots, X_n]$ is  inner and has the radial infimum
  property, Theorem  \ref{Schwarz} implies that $\|\Gamma\|_\infty\leq
  1$,
  \begin{equation*}
 G(X)G(X)^*\leq XX^*, \qquad X\in [B(\cH)^n]_1,
\end{equation*}
and
\begin{equation*}
 \|G(X)\|\leq \|X\|, \qquad X\in [B(\cH)^n]_1.
\end{equation*}
 Replacing   $X$ by  $ \Phi_a(X)$ in these inequalities and in relation  \eqref{iden-Gam},
  and using the fact that
$\Phi_a\circ \Phi_a=\text{\rm id}$, we complete the proof.
\end{proof}

\begin{corollary}\label{4delta}
 If $F:[B(\cH)^n]_1\to [B(\cH)^m]_1^-$ is a free
holomorphic function   with $\|F(0)\|<1$ and  $a\in \BB_n$, then,
for any $X\in [B(\cH)^n]_1$,

\begin{equation}\label{E3}
 \begin{split}
 \Delta_{F(a)}&[I- F(X)
F(a)^*]^{-1}[I-  F(X)   F(X)^*][I-  F(a)  F(X)^*]^{-1}
\Delta_{F(a)}\\
&\geq
 \Delta_{a}(I- X a^*)^{-1}(I-  X   X^*)(I- a
X^*)^{-1} \Delta_{a}
\end{split}
\end{equation}
and
\begin{equation}\label{E4}
 \begin{split}
\left\|[I-  F(a)  F(X)^*]\right. &\left.[I-  F(X)   F(X)^*]^{-1}[I-
F(X)
F(a)^*]\right\|\\
&\leq  \frac{\Delta_{F(a)}^2}{\Delta_a^2}
\left\|(I-a^*X)(I-XX^*)^{-1} (I-X a^*)\right\|.
 \end{split}
\end{equation}
\end{corollary}

\begin{proof}
The first inequality follows from Theorem \ref{Schw3} and  relation
\eqref{E1}.  Since $\|F(0)\|<1$, Corollary \ref{not-strict} implies
that $\|F(X)\|<1$ for any $X\in [B(\cH)^n]_1$. Note that each side
of inequality \eqref{E3} is a positive invertible  operator. It is
well-known that if $A,B$ are two positive  invertible operator such
that $A\leq B$ then $B^{-1}\leq A^{-1}$. Applying this result to
inequality \eqref{E3}, we deduce that
\begin{equation*}
\begin{split}
[I-  F(a)  F(X)^*]&[I-  F(X)   F(X)^*]^{-1}[I- F(X) F(a)^*]\\
&\leq \frac{\Delta^2_{F(a)}}{\Delta_a^2}(I-aX^*)(I-XX^*)^{-1} (I-X
a^*).
\end{split}
\end{equation*}
Hence, the second inequality follows. The proof is complete.
\end{proof}

 Let $F:[B(\cH)^n]_1\to [B(\cH)^m]_1^-$ be a free
holomorphic function. Let $\xi\in \partial\BB_n:=\{z\in \CC^n:\
\|z\|_2=1\}$ and assume that
$$
L:=\liminf_{z\to \xi} \frac{1-||F(z)||^2}{1-||z||^2}<\infty.
$$
Then there is a sequence $\{z_k\}_{k=1}^\infty\subset \BB_n$ such
that $\lim_{k\to \infty}z_k=\xi$ and  $\lim_{k\to\infty}F(z_k)=
\eta$ for some $\eta\in\partial \BB_m$, and
$$
\lim_{k\to \infty} \frac{1-\|F(z_k)\|^2}{1-\|z_k\|^2}= L.
$$

Now, we can present our first generalization of Julia's lemma for
free holomorphic functions on noncommutative balls.
\begin{theorem}
\label{Julia1} Let $F:[B(\cH)^n]_1\to [B(\cH)^m]_1^-$ be a free
holomorphic function with $\|F(0)\|<1$. Let
$\{z_k\}_{k=1}^\infty\subset \BB_n$ be a sequence such that
$\lim_{k\to \infty}z_k=\xi$, $\lim_{k\to\infty}F(z_k)= \eta$ for
some $\xi, \in\partial \BB_n$,  $\eta\in\partial \BB_m$, and
$$
\lim_{k\to \infty} \frac{1-\|F(z_k)\|^2}{1-\|z_k\|^2}= L <\infty.
$$
Then the following statements hold.
\begin{enumerate}
\item[(i)] For any $X\in [B(\cH)^n]_1$,
\begin{equation*}
 \begin{split}
\left\|[I-  \eta  F(X)^*]\right. &\left.[I-  F(X)   F(X)^*]^{-1}[I-
F(X)
\eta^*]\right\|\\
&\leq  L \left\|(I-\xi^*X)(I-XX^*)^{-1} (I-X \xi^*)\right\|.
\end{split}
\end{equation*}
\item[(ii)]

If $\beta>0$ and $X\in [B(\cH)^n]_1$ is such that
$$
(I- X \xi^*)(I-\xi X^*)<\beta (I-XX^*),
$$
then
$$
[I-F(X)\eta^*][I-\eta F(X)^*] <  \beta L[I-F(X)F(X)^*].
$$
\end{enumerate}
\end{theorem}

\begin{proof}
Due to  inequality \eqref{E4}, we have
\begin{equation*}
 \begin{split}
\left\|[I-  F(z_k)  F(X)^*]\right. &\left.[I-  F(X)
F(X)^*]^{-1}[I- F(X)
F(z_k)^*]\right\|\\
&\leq  \frac{1-\|F(z_k)\|^2}{1-\|z_k\|^2}
\left\|(I-z_k^*X)(I-XX^*)^{-1} (I-X z_k^*)\right\|.
 \end{split}
\end{equation*}
 Taking the limit as $k\to \infty$, we deduce item (i).
 To prove (ii), note first that,  using the same inequality
 \eqref{E4}, when $a=z_k$ and $X=0$, we obtain
 \begin{equation*}
 \begin{split}
\left\|[I-  F(z_k)  F(0)^*]\right. &\left.[I-  F(0) F(0)^*]^{-1}[I-
F(0)
F(z_k)^*]\right\|\\
&\leq  \frac {1-\|F(z_k)\|^2}{1-\|z_k\|^2}.
\end{split}
\end{equation*}
Taking $z_k\to \xi$ and due to the fact that $\|F(0)\|<1$, we deduce
that
$$L\geq\left\|[I-  \eta  F(0)^*][I-  F(0) F(0)^*]^{-1}[I-
F(0) \eta^*]\right\|>0.$$

Notice  also that, for any $X\in [B(\cH)^n]_1$, the following
inequalities are equivalent:
\begin{enumerate}
\item[(a)] $(I- X \xi^*)(I-\xi X^*)<\beta (I-XX^*)$,
\item[(b)] $ \|(I-\xi X^*)(I-XX^*)^{-1} (I-X\xi^*)\|<\beta$.
\end{enumerate}
Indeed, inequality (b) holds if and only if $\|(I-\xi
X^*)(I-XX^*)^{-1/2}\|<\beta^{1/2}$, which is equivalent to
$$
(I-XX^*)^{-1/2}(I-X\xi^*)(I-\xi X^*)(I-XX^*)^{-1/2}<\beta.
$$
The latter inequality is clearly equivalent to (a).

 Now, to prove (ii), we assume that
$$
(I- X \xi^*)(I-\xi X^*)<\beta (I-XX^*).
 $$
 Due to the equivalence of (a) with (b),  and  using
the  inequality   from (i) and  the fact that $L>0$, we obtain
$$
\left\|[I-  \eta  F(X)^*][I-  F(X)   F(X)^*]^{-1}[I- F(X)
\eta^*]\right\|<\beta L.
$$
Once again using the equivalence of (a) with (b) when $X$ is
replaced by $F(X)$, we obtain that
$$
[I-F(X)\eta^*][I-\eta F(X)^*] <  \beta L[I-F(X)F(X)^*].
$$
This completes the proof.
\end{proof}

We mention that, using unitary transformations in $B(\CC^n)$ and
$B(\CC^m)$, respectively, we can choose the coordinates such that
$\xi=(1,0,\ldots, 0)\in
\partial\BB_n$ and $\eta=(1,0,\ldots, 0)\in \partial\BB_m$, in Theorem
\ref{Julia1}. For $0<c<1$, we define the noncommutative  ellipsoid
$$
{\bf E}_c:=\left\{ (X_1,\ldots, X_n)\in B(\cH)^n:\
\frac{[X_1-(1-c)I][X_1^*-(1-c)]I]}{c^2}+ \frac{X_2X_2}{c}+\cdots
+\frac{X_nX_n}{c}<I\right\}
$$
with center at $((1-c)I,0,\ldots, 0)$.

Here is our second version of  Julia's lemma for free holomorphic
functions.
\begin{theorem}
\label{Julia2} Let $F:[B(\cH)^n]_1\to [B(\cH)^m]_1$ be a free
holomorphic function. Let $z_k\in \BB_n$ be such that $\lim_{k\to
\infty} z_k =(1,0,\ldots, 0)\in \partial\BB_n$, $\lim_{k\to \infty}
F(z_k) =(1,0,\ldots, 0)\in \partial\BB_m$, and
$$
\lim_{k\to \infty} \frac{1-\|F(z_k)\|^2}{1-\|z_k\|^2}= L <\infty. $$
 If
$F:=(F_1,\ldots, F_m)$, then the following statements hold.
\begin{enumerate}
\item[(i)] For any $X:=(X_1,\ldots, X_n)\in [B(\cH)^n]_1$,
$$
(I-F_1(X)^*)(I-F(X)F(X)^*)^{-1} (I-F_1(X))\leq L
(I-X_1^*)(I-XX^*)^{-1}(I-X_1).
$$
\item[(ii)]   If $0<c<1$, then
$$ F({\bf E}_c) \subset {\bf E}_\gamma,
\quad \text{where }\ \gamma :=\frac{Lc}{1+Lc-c}.
$$
\end{enumerate}
\end{theorem}

\begin{proof} As in the proof of Corollary \ref{4delta},
inequality \eqref{E3} implies
\begin{equation*}
\begin{split}
[I-  F(z_k)  F(X)^*]&[I-  F(X)   F(X)^*]^{-1}[I- F(X) F(z_k)^*]\\
&\leq \frac{1-\|F(z_k)\|^2}{1-\|z_k\|^2}(I-z_kX^*)(I-XX^*)^{-1} (I-X
z_k^*).
\end{split}
\end{equation*}
Taking the limit as $k\to\infty$, we obtain the inequality in item
(i). Now we prove item (ii).  Straightforward calculations reveal
that  $X=(X_1,\ldots, X_n)$ is in the noncommutative ellipsoid ${\bf
E}_c$ if and only if
\begin{equation}
\label{elips}
 (I-X_1)(I-X_1^*)<\frac{c}{1-c}(I-XX^*).
\end{equation}
According to the equivalence $(a)\leftrightarrow (b)$ (see the proof
of Theorem \ref{Julia1}), when $\xi=(1,0,\ldots,0)$ and
$\beta:=\frac{c}{1-c}$, the latter inequality is equivalent to
$$
\|(I-X_1^*)(I-XX^*)^{-1} (I-X_1)\|<\frac{c}{1-c},
$$
which is equivalent to
$$
(I-X_1^*)(I-XX^*)^{-1} (I-X_1)< \frac{c}{1-c} I.
$$
Usind the inequality from item (i), we obtain
$$
(I-F_1(X)^*)(I-F(X)F(X)^*)^{-1} (I-F_1(X))<\frac{Lc}{1-c}
I=\frac{\gamma}{1-\gamma} I,
$$
where $\gamma:=\frac{Lc}{1+Lc-c}$. As above, the latter inequality
is equivalent to
\begin{equation*}
 (I-F_1(X))(I-F_1(X)^*)<\frac{\gamma}{1-\gamma}(I-F(X)F(X)^*),
\end{equation*}
which is equivalent to $F(X)\in {\bf E}_\gamma$. This completes the
proof.
\end{proof}

We  introduce noncommutative Korany  type regions in $[B(\cH)^n]_1$.
For each  $\xi\in
\partial \BB_n$ and $\alpha>1$, we define
$$
{\bf D}_\alpha(\xi):=\left\{X\in B(\cH)^n:\ (I-X\xi^*)(I-\xi X^*)<
\frac{\alpha^2}{4} (1-\|X\|^2)(I-XX^*)\right\}. $$

 Note that if $\cH=\CC$, then ${\bf D}_\alpha(\xi)$ coincides with
 the Korany region (see \cite{Ru2})
 $$D_\alpha(\xi)=\left\{ z\in \CC^n: |1-\left<
 z,\xi\right>|<\frac{\alpha}{2}(1-|z|^2)\right\}.
 $$

\begin{corollary} \label{Korany} If $F$ is as in Theorem \ref{Julia2} and
$F(0)=0$, then
\begin{enumerate}
\item[(i)]
$ F({\bf E}_c) \subset {\bf E}_{cL}, \quad \text{  for  }\
0<c<\frac{1}{L};$
\item[(ii)]
$ F({\bf D}_\alpha) \subset {\bf D}_{\alpha\sqrt{L}}, \quad \text{
for }\ \alpha>1$, where ${\bf D}_\alpha={\bf
D}_\alpha(1,0\ldots,0)$.
\end{enumerate}
\end{corollary}

\begin{proof} Since $F(0)=0$, due to  Schwarz lemma for free
holomorphic functions, we have $\|F(X)\|\leq \|X\|$ for all $X\in
[B(\cH)^n]_1$. Consequently,
$$
L=\lim_{k\to \infty} \frac{1-\|F(z_k)\|^2}{1-\|z_k\|^2}\geq 1
$$
which implies $\gamma :=\frac{Lc}{1+Lc-c}\leq Lc$, therefore  ${\bf
E}_\gamma\subset {\bf E}_{cL}$. Due to Theorem \ref{Julia2}, we
deduce that $F({\bf E}_c)\subseteq {\bf E}_{cL}$ when
$0<c<\frac{1}{L}$.

To prove item (ii), let $X=(X_1,\ldots, X_n)\in {\bf D}_\alpha$,
i.e.,
$$
(I-X_1)(I-X_1^*)< \frac{\alpha^2}{4} (1-\|X\|^2)(I-XX^*).
$$
Applying Theorem \ref{Julia1}, part (ii),  when $\xi=(1,0,\ldots,
0)$ and $\beta=\frac{\alpha^2}{4} (1-\|X\|^2)$ we deduce that
$$
[I-F_1(X)][I- F_1(X)^*] <  \frac{L\alpha^2}{4} (1-\|X\|^2)
[I-F(X)F(X)^*].
$$
Since $\|F(X)\|\leq \|X\|$, we obtain
$$
[I-F_1(X)][I- F_1(X)^*] <  \frac{L\alpha^2}{4} (1-\|F(X)\|^2)
[I-F(X)F(X)^*],
$$
which shows that $F(X)\in {\bf D}_{\alpha\sqrt{L}}$ and completes
the proof.
\end{proof}

\bigskip

\section{ Pick-Julia   theorems  for   free  holomorphic   functions with
operator-valued coefficients
   }

In this section, we use fractional transforms  and a version of the
noncommutative Schwarz's lemma to obtain Pick-Julia theorems for
free holomorphic functions  $F$ with operator-valued coefficients
such that  $\|F\|_\infty\leq 1$ (resp. $\Re F \geq 0$). As  a
consequence, we obtain
 a Julia type lemma for free holomorphic
 functions with positive real parts.
  We also provide commutative versions
  of these results for
operator-valued multipliers of the Drury-Arveson space.

\begin{theorem}
\label{Pi-Ju} Let $F:[B(\cH)^n]_1\to B(\cH)\bar \otimes_{min}
B(\cE,\cG)$ be a free holomorphic function  with $\|F\|_\infty\leq
1$ and $\|F(0)\|<1$.  If  $z\in \BB_n$, then
$$
 \Psi_{F(z)}(F(X))\Psi_{F(z)}(F(X))^*\leq  \Phi_z(X)\Phi_z(X)^*\otimes I_\cG,\qquad X\in
[B(\cH)^n]_1,
$$
where $\Psi_{F(z)}$   is the fractional transform defined by
\eqref{frac} and $\Phi_z$  is the corresponding free holomorphic
automorphisms of $[B(\cH)^n]_1$. Moreover, we have
\begin{equation*}
 \begin{split}
 D_{F(z)^*}&[I- F(X)
F(z)^*]^{-1}[I-  F(X)   F(X)^*][I-  F(z)  F(X)^*]^{-1}
D_{F(z)^*}\\
&\geq
 \Delta_{z}(I- X z^*)^{-1}(I-  X   X^*)(I- z
X^*)^{-1} \Delta_{z}\otimes I_\cG
\end{split}
\end{equation*}
for any $z\in \BB_n$ and $X\in [B(\cH)^n]_1$.
\end{theorem}

\begin{proof} Since $\|F(0)\|<1$, Corollary \ref{not-strict} implies
that $\|F(X)\|<1$ for any $X\in [B(\cH)^n]_1$.  According to Theorem
\ref{more-prop} and Theorem \ref{frac-trans}, the mapping $X\mapsto
(\Psi_{F(z)}\circ F\circ \Phi_z)(X)$ is a bounded free holomorphic
function on $[B(\cH)^n]_1$  with $\|(\Psi_{F(z)}\circ F\circ
\Phi_z)(X)\|<1$ for any $X\in [B(\cH)^n]_1$. On the other hand,
since we have $(\Psi_{F(z)}\circ F\circ
\Phi_z)(0)=\Psi_{F(z)}(F(z))=0$, we can apply Corollary \ref{SCH}
and obtain
$$
(\Psi_{F(z)}\circ F\circ \Phi_z)(Y)[(\Psi_{F(z)}\circ F\circ
\Phi_z)(Y)]^*\leq YY^*\otimes I_\cG
$$
for any $Y\in [B(\cH)^n]_1$. Taking $Y=\Phi_z(X)$, $X\in
[B(\cH)^n]_1$,  and due to the identity $\Phi_z\circ
\Phi_z=\text{\rm id}$, we obtain
$$
 \Psi_{F(z)}(F(X))\Psi_{F(z)}(F(X))^*\leq  \Phi_z(X)\Phi_z(X)^*\otimes I_\cG,\qquad X\in
[B(\cH)^n]_1.
$$
Using now relations \eqref{formulas} and \eqref{E1}, we complete the
proof.
\end{proof}

We remark that under the conditions of Theorem \ref{Pi-Ju}, one can
show, as in the proof of  Theorem \ref{Schw3}, that there is a free
holomorphic function  $G$ with operator-valued coefficients and
$\|G\|_\infty\leq 1$ such that
$$
\Psi_{F(z)}[F(X)]=\Phi_z(X) (G\circ \Phi_z)(X),\quad X\in
[B(\cH)^n]_1.
$$

Our Pick-Julia type result for free holomorphic functions with
positive real parts is the following.
\begin{theorem}
\label{Pi-Ju2} If $G:[B(\cH)^n]_1\to B(\cH)\bar \otimes_{min}
B(\cE)$ is a free holomorphic function  with $\Re G > 0$, then
\begin{equation*}
 \begin{split}
 \Gamma(z)[I+G(z)^*][G(X)+G(z)^*]^{-1}& [\Re G(X)][G(z)+G(X)^*]^{-1}
 [I+G(z)] \Gamma(z)
  \\
&\geq
   (1-\|z\|_2^2)(I- X z^*)^{-1}(I-  X   X^*)(I- z
X^*)^{-1} \otimes I_\cE
\end{split}
\end{equation*}
for any $z\in \BB_n$ and $X\in [B(\cH)^n]_1$, where
$$
\Gamma(z):=2\left\{[I+G(z)]^{-1} [\Re
G(z)][I+G(z)^*]^{-1}\right\}^{1/2}.
$$
\end{theorem}

\begin{proof} According to the considerations preceding Lemma
\ref{inverse}, since $\Re G > 0$, the noncommutative Cayley
transform $F:=\text{\boldmath{$\cC$}}[G]:=(G-I)(I+G)^{-1}$ is a
bounded free holomorphic function with $\|F(X)\|<1$ for any $X\in
[B(\cH)^n]_1$. Due to Theorem \ref{Pi-Ju}, we obtain
\begin{equation}\label{DF}
 \begin{split}
 D_{F(z)^*}&[I- F(X)
F(z)^*]^{-1}[I-  F(X)   F(X)^*][I-  F(z)  F(X)^*]^{-1}
D_{F(z)^*}\\
&\geq
 \Delta_{z}(I- X z^*)^{-1}(I-  X   X^*)(I- z
X^*)^{-1} \Delta_{z}\otimes I_\cG.
\end{split}
\end{equation}
Note that
\begin{equation*}
 \begin{split}
I-  F(X)  F(z)^* &=
I-[I+G(X)]^{-1}[G(X)-I][G(z)^*-I][I+G(z)^*]^{-1}\\
&=[I+G(X)]^{-1}\left\{[I+G(X)][I+G(z)^*]-[G(X)-I][G(z)^*-I]\right\}[I+G(z)^*]^{-1}\\
&=2[I+G(X)]^{-1}[G(X)+ G(z)^*][I+G(z)^*]^{-1}
\end{split}
\end{equation*}
and, similarly,
$$
I-  F(X)  F(X)^*=2[I+G(X)]^{-1}[G(X)+ G(X)^*][I+G(X)^*]^{-1}
$$
for any $X\in [B(\cH)^n]_1$ and $z\in \BB_n$. Using these
identities, we deduce that
\begin{equation*}
\begin{split}
[I- F(X) F(z)^*]^{-1}&[I-  F(X)   F(X)^*][I-  F(z)  F(X)^*]^{-1}\\
&=\frac{1}{2}[I+G(z)^*][G(X)+G(z)^*]^{-1}[I+G(X)]\\
&\quad \times 2[I+G(X)]^{-1}[G(X)+ G(X)^*][I+G(X)^*]^{-1}\\
&\quad \times \frac{1}{2}[I+G(X)^*][G(z)+G(X)^*]^{-1}[I+G(z)]\\
&=\frac{1}{2}[I+G(z)^*][G(X)+G(z)^*]^{-1}[G(X)+
G(X)^*][G(z)+G(X)^*]^{-1}[I+G(z)].
\end{split}
\end{equation*}

Now,  since
\begin{equation*}
\begin{split}
 D_{F(z)^*}&=[I-F(z)F(z)^*]^{1/2}=\left[2[I+G(z)]^{-1}[G(z)+
 G(z)^*][I+G(z)^*]^{-1}\right]^{1/2}\\
 &=2 \left[[I+G(z)]^{-1}[\Re G(z) ][I+G(z)^*]^{-1}\right]^{1/2}
 \end{split}
\end{equation*}
the  inequality \eqref{DF}  implies  the inequality  of the theorem.
The proof is complete.
\end{proof}

 The next result is a Julia type lemma for free holomorphic
 functions with  scalar coefficients and positive real parts.

\begin{theorem}
\label{Julia3} Let $G:[B(\cH)^n]_1\to B(\cH) $ be  a free
holomorphic function  with $\Re G > 0$.  Let
$\{z_k\}_{k=1}^\infty\subset \BB_n$ be a sequence such that
$\lim_{k\to \infty}z_k=\xi\in \partial \BB_n$,
$\lim_{k\to\infty}|G(z_k)|= \infty$, and such that
$$
\lim_{k\to \infty}\frac{\Re
G(z_k)}{(1-\|z_k\|_2^2)|G(z_k)|^2}=M<\infty.
$$
Then $M>0$ and
$$
\Re G(X)\geq  \frac{1}{4M}(I- X \xi^*)^{-1}(I-  X   X^*)(I- \xi
X^*)^{-1}
$$
for any $X\in [B(\cH)^n]_1$.
\end{theorem}

\begin{proof}According to Theorem \ref{Pi-Ju2}, when $\cE=\CC$, we
have
$$
\Gamma(z_k)=\frac{2[\Re G(z_k)]^{1/2}}{|1+G(z_k)|}
$$
and
\begin{equation*}
\begin{split}
4[\Re G(z_k)][G(X)+G(z_k)^*]^{-1} &[\Re G(X)] [G(z_k)+G(X)^*]^{-1}\\
&\geq (1-\|z_k\|^2) (I- X z_k^*)^{-1}(I-  X   X^*)(I- z_k X^*)^{-1}.
\end{split}
\end{equation*}
Hence, we obtain

\begin{equation} \label{RR}
\begin{split}
\frac{4\Re G(z_k)}{(1-\|z_k\|^2) |G(z_k)|^2}\Re G(X) \geq
 \frac{1}{|G(z_k)|^2}[G(X)+G(z_k)^*] A(k)[G(z_k)+G(X)^*],
\end{split}
\end{equation}
where $A(k):=(I- X z_k^*)^{-1}(I-  X   X^*)(I- z_k X^*)^{-1}$.
Taking $X=0$  in  inequality  \eqref{RR}, we obtain
$$
\frac{4\Re G(z_k)}{(1-\|z_k\|^2) |G(z_k)|^2}\Re G(0) \geq
\frac{1}{|G(z_k)|^2}[G(0)+G(z_k)^*] [G(z_k)+G(0)^*],
$$
whence
$$
\frac{1}{|G(z_k)|^2}[G(0)+G(z_k)^*] [\Re G(0)]^{-1}[G(z_k)+G(0)^*]
\leq \frac{\Re G(z_k)}{(1-\|z_k\|_2^2)|G(z_k)|^2} I.
$$
Since $\lim_{k\to\infty}|G(z_k)|= \infty$ and  taking the limit in
the latter inequality,   we obtain
$$
[\Re G(0)]^{-1} \leq \lim_{k\to\infty} \frac{\Re
G(z_k)}{(1-\|z_k\|_2^2)|G(z_k)|^2} I=MI.
$$
Consequently, we have $M>0$.

Now, due to inequality \eqref{RR} and the fact that
\begin{equation*}
\begin{split}
\lim_{k\to\infty} \frac{1}{|G(z_k)|^2}[G(X)+G(z_k)^*]&
A(k)[G(z_k)+G(X)^*]\\
&=\lim_{k\to \infty}{A(k)}= (I- X \xi^*)^{-1}(I- X X^*)(I- \xi
X^*)^{-1},
\end{split}
\end{equation*}
we deduce that
$$
4M\Re G(X)\geq  (I- X \xi^*)^{-1}(I-  X   X^*)(I- \xi X^*)^{-1}
$$
for any $X\in [B(\cH)^n]_1$, which completes the proof.
\end{proof}

We recall (see \cite{Po-poisson}, \cite{Po-holomorphic},
\cite{Po-pluriharmonic}) that if $F$ is a contractive
($\|F\|_\infty\leq 1$) free holomorphic function  with coefficients
in $B(\cE)$, then its boundary function is in $F_n^\infty\bar\otimes
B(\cE)$.  Consequently, the evaluation map $\BB_n\ni z\mapsto
F(z)\in B(\cE )$ is a contractive  operator-valued multiplier of the
Drury-Arveson space (\cite{Dr}, \cite{Arv}), and  any such a
contractive multiplier has  this type of representation.

\begin{corollary}\label{multiplier} The following statements hold.
\begin{enumerate}
\item[(i)]
 If $F:[B(\cH)^n]_1\to  B(\cH)\bar\otimes_{min} B(\cE)$ is a free
holomorphic function   with $\|F\|_\infty\leq 1$ and  $\|F(0)\|<1$
then

\begin{equation*}
 \begin{split}
[I-  F(z)  F(w)^*] &[I-  F(w)   F(w)^*]^{-1}[I- F(w)
F(z)^*]\\
&\leq  \frac{ |(1-\left<w,z\right>|^2}{(1-\|z\|^2)
(1-\|w\|^2)}[I-F(z)F(z)^*] .
 \end{split}
\end{equation*}
for any $z,w\in \BB_n$.
\item[(ii)]
If $G:[B(\cH)^n]_1\to B(\cH)\bar\otimes_{min} B(\cE)$ is a free
holomorphic function  with $\Re G>0$, then
\begin{equation*}
 \begin{split}
 [G(z)+G(w)^*][\Re G(w)]^{-1} [G(w)+G(z)^*]
\leq  \frac{4 |(1-\left<w,z\right>|^2}{(1-\|z\|^2) (1-\|w\|^2)}\Re
G(z)
\end{split}
\end{equation*}
for any $z,w\in \BB_n$.
\end{enumerate}
\end{corollary}

\begin{proof} Taking $X=w\in \BB_n$ in Theorem \ref{Pi-Ju}, we
deduce that
\begin{equation*}
\begin{split}
[I-F(w)F(z)^*]^{-1}&[I-F(w)F(w)^*][I-F(z)F(w)^*]^{-1} \\
&\geq \frac{(1-\|z\|_2^2)(1-\|w\|_2^2)}{|1-\left<w,z\right>|^2}
[I-F(w)F(w)^*]^{-1}
\end{split}
\end{equation*}
for any $z,w\in \BB_n$. We recall that if $A,B$ are two positive
invertible operator such that $A\leq B$ then $B^{-1}\leq A^{-1}$.
Applying this result to the inequality above, we complete the proof
of item (i).

To prove part (ii), take  $X=w\in \BB_n$ in Theorem \ref{Pi-Ju2}. We
obtain
\begin{equation*}
 \begin{split}
 [I+G(z)^*][G(w)+G(z)^*]^{-1}& [\Re G(w)][G(z)+G(w)^*]^{-1}
 [I+G(z)]
  \\
&\geq
   \frac{(1-\|z\|_2^2)(1-\|w\|^2)}{4|1-\left<w,z\right>|^2}[I+G(z)^*][\Re G(z)]^{-1} [I+ G(z)]
\end{split}
\end{equation*}
Multiplying to the left by $[I+G(z)^*]^{-1}$ and to the right by
$[I+G(z)]^{-1}$,  and passing to inverses, as above, we obtain the
desired inequality. The proof is complete.
\end{proof}

 \bigskip

\section{ Lindel\" of inequality and  sharpened forms of the noncommutative von Neumann
inequality
   }

In this section, we provide  a noncommutative generalization of a
classical
 inequality due to Lindel\"of, which turns out to be
  sharper then the noncommutative von Neumann inequality.

\begin{theorem} \label{Lindelof}If  $F:[B(\cH)^n]_1\to [B(\cH)^m]_1^-$ is   a free
holomorphic function, then
\begin{equation*}
  \|F(X)\|\leq
\frac{\|X\|+\|F(0)\|}{1+\|X\|\|F(0)\|},\qquad X\in [B(\cH)^n]_1.
\end{equation*}
If, in addition, the boundary function of $F$ has its entries in the
noncommutative disc algebra $\cA_n$, then the inequality above  can
be extended to  any $ X\in [B(\cH)^n]_1^-$.
\end{theorem}

\begin{proof} First, we consider the case when $\|F(0)\|<1$.
Using the first inequality  of Corollary \ref{4delta}, in the
particular case when $a=0$, we obtain
\begin{equation*}
\begin{split}
\Delta_{F(0)}[I- F(X) F(0)^*]^{-1}[I-  F(X)   F(X)^*][I-  F(0)
F(X)^*]^{-1} \Delta_{F(0)} \geq I-XX^*.
\end{split}
\end{equation*}
Hence, we deduce that
\begin{equation} \label{F(0)}
\begin{split}
I-  F(X)   F(X)^*\geq \frac{1-\|X\|^2}{1-\|F(0)\|^2}[I- F(X)
F(0)^*][I-  F(0) F(X)^*].
\end{split}
\end{equation}
On the other hand,  since $\|F(0)\|<1$, the operator $I- F(X)
F(0)^*$ is invertible and
\begin{equation*}
\begin{split}
\|[I- F(X) F(0)^*]^{-1}\|&\leq
1+\|F(X)\|\|F(0)\|+\|F(X)\|^2\|F(0)\|^2 +\cdots\\
&=\frac{1}{1-\|F(X)\|\|F(0)\|}.
\end{split}
\end{equation*}
Similarly, we have $\|[I-  F(0) F(X)^*]^{-1}\|\leq
\frac{1}{1-\|F(X)\|\|F(0)\|}$ and   deduce that
\begin{equation*}
\begin{split}
[I-  F(0) F(X)^*]^{-1}[I- F(X) F(0)^*]^{-1}&\leq \|[I-  F(0)
F(X)^*]^{-1}\|\|[I- F(X) F(0)^*]^{-1}\| I\\
&\leq \frac{1}{(1-\|F(X)\|\|F(0)\|)^2}I.
\end{split}
\end{equation*}
Hence, we obtain
$$
[I- F(X) F(0)^*][I-  F(0) F(X)^*]\geq (1-\|F(X)\|\|F(0)\|)^2I,
$$
which combined with inequality \eqref{F(0)}, leads to
$$
F(X)F(X)^*\leq
\left(1-\frac{1-\|X\|^2}{1-\|F(0)\|^2}(1-\|F(X)\|\|F(0)\|)^2\right)I.
$$
This inequality implies
$$
\|F(X)\|^2\leq
1-\frac{1-\|X\|^2}{1-\|F(0)\|^2}(1-\|F(X)\|\|F(0)\|)^2,
$$
which  is equivalent to
$$
(1-\|F(X)\|^2)(1-\|F(0)\|^2)\geq (1-\|X\|^2)(1-\|F(X)\|\|F(0)\|)^2.
$$
Straightforward calculations  show that the latter inequality is
equivalent to
$$
\left( \|F(X)\|-\|F(0)\|\right)^2\leq
\|X\|^2\left(I-\|F(X)\|\|F(0)\|\right)^2.
$$
Hence, we obtain
$$ \|F(X)\|-\|F(0)\|\leq \|X\|-\|X\|\|F(X)\|\|F(0)\|,
$$
which is equivalent to
\begin{equation*}
 \|F(X)\|\leq \frac{\|X\|+\|F(0)\|}{1+\|X\|\|F(0)\|},\qquad X\in
[B(\cH)^n]_1.
\end{equation*}

Now, we consider the case when $\|F(0)\|=1$.  Applying  our result
above to  $\epsilon F$, where  $\epsilon\in (0,1)$, we get
\begin{equation*}
 \epsilon\|F(X)\|\leq \frac{\|X\|+\epsilon\|F(0)\|}{1+\epsilon\|X\|\|F(0)\|}.
\end{equation*}
Taking $\epsilon\to 0$, the result follows. Now, consider the case
when the boundary function of $F$ has its entries in the
noncommutative disc algebra $\cA_n$.   According to
\cite{Po-holomorphic}, we have
$$
F(X)=\lim_{r\to 1} F(rX_1,\ldots, rX_n),\qquad X=(X_1,\ldots,
X_n)\in [B(\cH)^n]_1^-,
$$
in the operator norm topology. Applying inequality  of this theorem
to the free holomorphic function $X\mapsto F(rX_1,\ldots, rX_n)$ and
taking $r\to 1$, we complete the proof.
\end{proof}

A few remarks are necessary.  First, notice that in the particular
case when $F(0)=0$, Theorem \ref{Lindelof} implies the
noncommutative Schwarz  type result. We also remark that if
$\|F(0)\|<1$, then
$$
\frac{\|X\|+\|F(0)\|}{1+\|X\|\|F(0)\|}<1,\qquad X\in [B(\cH)^n]_1.
$$
Therefore, the inequality   in Theorem \ref{Lindelof} is sharper
than the noncommutative von Neumann inequality, which gives only
$\|F(X)\|\leq 1$, when $\|F\|_\infty\leq 1$.

We recall that if $F:=(F_1,\ldots, F_m)$ is a contractive
($\|F\|_\infty\leq 1$) free holomorphic function, then the
evaluation map $\BB_n\ni z\mapsto F(z)\in  \BB_m$ is a contractive
 matrix-valued multiplier of the Drury-Arveson space and, moreover,
any such a contractive multiplier has  this kind of representation.
In particular, Theorem \ref{Lindelof} implies that
 $$
 \|F(z)\|\leq
\frac{\|z\|+\|F(0)\|}{1+\|z\|\|F(0)\|},\qquad z\in \BB_n.
$$
 for any contractive
   multiplier $F:\BB_n\to \BB_m$  of the Drury-Arveson space.

We consider now the particular case when $m=1$. Here is a sharpened
form of the noncommutative von Neumann inequality (see
\cite{Po-von})

\begin{corollary} If $f:[B(\cH)^n]_1\to B(\cH)$  is a nonconstant   free
holomorphic function with $\|f\|_\infty\leq 1$, then

$$
\|f(X)\|\leq \frac{\|X\|+|f(0)|}{1+\|X\||f(0)|}<1,\qquad X\in
[B(\cH)^n]_1.
$$
If, in addition,  $f$ is in the noncommutative disc algebra $\cA_n$,
then the left inequality holds for any $ X\in [B(\cH)^n]_1^-$.
\end{corollary}

Another consequence  of Theorem \ref{Lindelof} is the following.

\begin{corollary} \label{sharp}
Let  $F:[B(\cH)^n]_1\to [B(\cH)^m]_1^-$  be  a free holomorphic
function and let $z\in \BB_n$, then
$$
\|F(X)\|\leq
\frac{\|\Phi_z(X)\|+\|F(z)\|}{1+\|\Phi_z(X)\|\|F(z)\|},\qquad X\in
[B(\cH)^n]_1,
$$
where $\Phi_z$ is  the free holomorphic automorphism of the
noncommutative unit ball $[B(\cH)^n]_1$ associated $z\in \BB_n$.
\end{corollary}

\begin{proof} Applying  Theorem \ref{Lindelof} to the free
holomorphic function $F\circ \Phi_z:[B(\cH)^n]_1\to [B(\cH)^m]_1^-$,
we obtain
\begin{equation*}
 \|(F\circ \Phi_z)(Y)\|\leq \frac{\|Y\|+\|(F\circ \Phi_z)(0)\|}{1+\|Y\|
 \|(F\circ\Phi_z)(0)\|},\qquad Y\in
[B(\cH)^n]_1.
\end{equation*}
Taking into account that $\Phi_z(0)=z$, $\Phi_z\circ
\Phi_z=\text{\rm id}$, and setting $Y=\Phi_z(X)$, $X\in
[B(\cH)^n]_1$,  in the inequality above, we obtain the desired
inequality.
\end{proof}

\bigskip

\section{ Pseudohyperbolic metric on the unit ball of $B(\cH)^n$ and an invariant Schwarz-Pick lemma }

  The
pseudohyperbolic  distance  on the open unit disc $\DD:=\{z\in \CC:
\ |z|<1\}$  of the complex plane is defined   by
$$
d_1(z,w):= \left| \frac{z-w}{1-\bar z w}\right|,\quad z,w\in \DD.
$$
Some of the basic  properties of the pseudohyperbolic distance are
the following:

\begin{enumerate}
\item[(i)]
the pseudohyperbolic  distance is invariant under the conformal
automorphisms of $\DD$, i.e.,
$$
d_1(\varphi(z), \varphi(w))=d_1(z,w),\quad z,w\in \DD,
$$
for all $\varphi\in Aut(\DD)$;
\item[(ii)] the $d_1$-topology induced on the open disc is the usual
planar topology;

\item[(iii)]  any analytic function $f:\DD\to \DD$ is distance-decreasing,
i.e., satisfies
$$
d_1(f(z), f(w))\leq d_1(z,w),\quad z,w\in \DD.
$$
\end{enumerate}

 The analogue of the pseudohyperbolic
distance for the open unit ball of $\CC^n$,
$$\BB_n:=\{z=(z_1,\ldots, z_n)\in \CC^n:\ \|z\|_2<1\},$$
     is defined by
 $$
d_n(z,w)= \|\psi_z(w)\|_2,\qquad z,w\in \BB_n,
$$
where $\psi_z$ is the involutive automorphism of $\BB_n$ that
interchanges $0$ and $z$. This    distance has
 properties similar to those of  $d_1$ (see \cite{Ru2}, \cite{Zhu}).

 In what follows, we introduce   a pseudohyperbolic
metric  on the noncommutative ball $ [B(\cH)^n]_1$,  which
  satisfies  properties similar to those of the pseudohyperbolic  metric
  $d_1$ on the unit disc $\DD$
  and which is a noncommutative extension  of  $d_n$ on the open unit ball of $\CC^n$.
In particular, we  obtain a Schwarz-Pick lemma for free holomorphic
functions on $[B(\cH)^n]_1$   with respect to this  pseudohyperbolic
metric.

We recall  (\cite{Po-hyperbolic}) that  $A,B\in [B(\cH)^n]_1^-$ are
called Harnack equivalent (and denote
$A\overset{H}{{\underset{c}\sim}}\, B$)  if and only if there exists
$c\geq 1$ such that
\begin{equation*}
 \frac{1}{c^2}\text{\rm Re}\,p(B_1,\ldots, B_n)\leq
\text{\rm Re}\,p(A_1,\ldots, A_n)\leq c^2 \text{\rm
Re}\,p(B_1,\ldots, B_n)
\end{equation*}
for any noncommutative polynomial with matrix-valued coefficients
$p\in \CC[X_1,\ldots, X_n]\otimes M_{m\times m}$, $m\in \NN$, such
that $\text{\rm Re}\,p\geq 0$.
 The equivalence classes with respect to $\overset{H}\sim$ are
called Harnarck parts of $[B(\cH)^n]_1^-$. We proved  that the
  open unit ball $[B(\cH)^n]_1$ is a distinguished Harnack part of
  $[ B(\cH)^n]_1^-$, namely, the Harnack part of $0$.

Now, let $\Delta$ be a Harnarck part of $[B(\cH)^n]_1^-$ and define
the map ${\bf d}:\Delta\times \Delta\to [0,\infty)$ by setting
\begin{equation}
\label{Om} {\bf
d}(A,B):=\frac{\omega(A,B)^2-1}{\omega(A,B)^2+1},\qquad A,B\in
\Delta.
\end{equation}
where
$$\omega(A,B):=\inf\left\{ c \geq  1: \
A\,\overset{H}{{\underset{c}\sim}}\, B   \right\}, \qquad  A,B\in
\Delta.
$$

The first result of this section is the following.
\begin{theorem}
\label{pseudometric} Let $\Delta$ be a Harnarck part of
$[B(\cH)^n]_1^-$. Then the map $\bf{d}$ defined by relation
\eqref{Om} has the following properties:
\begin{enumerate}
\item[(i)] ${\bf d}$ is a bounded metric on $\Delta$;
\item[(ii)]  for any free
holomorphic automorphism $\Phi$ of the  noncommutative unit ball
$[B(\cH)^n]_1$,
$${\bf d}(X,Y)={\bf d}(\Phi(X), \Phi(X)),\qquad X,Y\in \Delta.
$$
\end{enumerate}
\end{theorem}

\begin{proof}

 According to Lemma 2.1 from \cite{Po-hyperbolic},
if  $\Delta$ is a Harnack part of $[B(\cH)^n]_1^-$ and  $A,B,C\in
\Delta$, then the following properties hold:
\begin{enumerate}
\item[(a)] $\omega(A,B)\geq 1$ and  $\omega(A,B)=1$ if and only if $A=B$;
\item[(b)] $\omega(A,B)=\omega(B,A)$;
\item[(c)] $\omega(A,C)\leq \omega(A,B) \omega(B,C)$.
\end{enumerate}
Part (c) can be used  to show that
 \begin{equation*}
{\bf d}(A,C)\leq {\bf d}(A,B)+ {\bf d}(B,C). \end{equation*}
 Indeed, define the
function $f:[1,\infty)\to [0,\infty)$ by
$f(x):=\frac{x^2-1}{x^2+1}$. Since $f'(x)=\frac{2x}{(x^2+1)^2}\geq
0$, we deduce that $f$ is increasing. Hence, and due to inequality
(c), we have
$$
f(\omega(A,C))\leq f\left(\omega(A,B) \omega(B,C)\right).
$$
Since $f(\omega(A,C))={\bf{d}}(A,C)$, it remains  to prove that
$$
f\left(\omega(A,B) \omega(B,C)\right)\leq f\left(\omega(A,B) \right)
+ f\left( \omega(B,C)\right).
$$
Setting $x:=\omega(A,B)$ and $y:=\omega(B,C)$, the  inequality above
is equivalent to
$$
\frac{x^2 y^2-1}{x^2 y^2+1}\leq \frac{x^2 -1}{x^2 +1}+ \frac{
y^2-1}{ y^2+1}.
$$
Straightforward  calculations reveal that the latter inequality is
equivalent to $$(x^2 y^2-1)(x^2 -1)(y^2-1)\geq 0, $$ which holds for
any $x,y\geq 1$. Using (a) and (b), one can deduce that ${\bf d}$ is
a metric.

Now, we prove part (ii). According to Lemma 2.3 from
\cite{Po-hyperbolic}, if $A$ and $B$ are in $[B(\cH)^n]_1^-$, $c\geq
1$,  and $\Psi\in Aut([B(\cH)^n]_1)$, then
$A\overset{H}{{\underset{c}\prec}}\, B$ if and only if
$\Phi(A)\overset{H}{{\underset{c}\prec}}\, \Phi(B)$. Consequently,
if $A,B\in \Delta$, then  we deduce that
$\omega(A,B)=\omega(\Phi(A), \Phi(B))$, which implies (ii). The
proof is complete.
\end{proof}

We introduced in \cite{Po-hyperbolic} a hyperbolic   ({\it
Poincar\'e-Bergman} \cite{Be}  type) metric $\delta$ on any Harnack
part $\Delta$ of $[B(\cH)^n]_1^-$  by setting
\begin{equation}\label{def-Hyper}
 \delta(A,B):=\ln \omega(A,B).
\end{equation}
We will use the properties of $\delta$  to deduce  the following
result concerning the pseudohyperbolic distance on the open
noncommutative ball $[B(\cH)^n]_1$.

\begin{theorem}\label{pseudo-t} The  pseudohyperbolic metric
 ${\bf d}:[B(\cH)^n]_1\times [B(\cH)^n]_1\to [0,\infty)$ has the following properties:

\begin{enumerate}
\item[(i)] for any $X,Y\in [B(\cH)^n]_1$, $${\bf d}(X,Y)=\tanh \delta(X,Y).$$
\item[(ii)]
${\bf d}|_{\BB_n\times \BB_n}$ coincides with the pseudohyperbolic
distance on $\BB_n$, i.e.,
$$
{\bf d}(z,w)= \|\psi_z(w)\|_2,\qquad z,w\in \BB_n,
$$
where $\psi_z$ is the involutive automorphism of $\BB_n$ that
interchanges $0$ and $z$.
\item[(iii)] the ${\bf d}$-topology
 coincides with the norm topology on  the open unit ball
$[B(\cH)^n]_1$;
\item[(iv)] for any $X, Y\in [B(\cH)^n]_1$,
 \begin{equation*} {\bf d}(X,Y):=\frac{\max\{\|\Gamma\|,\|\Gamma^{-1}\|\}-1}
{\max\{\|\Gamma\|,\|\Gamma^{-1}\|\}+1},
\end{equation*}
where $$\Gamma:=(C_XC_Y^{-1})^*(C_XC_Y^{-1}), \quad C_X:=(
\Delta_X\otimes I)(I-R_X)^{-1},$$
 and $R_X:=X_1^*\otimes
R_1+\cdots + X_n^*\otimes R_n$ is the reconstruction operator.
\end{enumerate}
\end{theorem}

\begin{proof} Part (i) follows from relations \eqref{Om} and
\eqref{def-Hyper}. Part (ii) follows from part (i) and the fact
that, according to \cite{Po-hyperbolic},
 $\delta|_{\BB_n\times \BB_n}$ coincides with the
Poincar\'e-Bergman distance on $\BB_n$, i.e.,
$$
\delta(z,w)=\frac{1}{2}\ln
\frac{1+\|\psi_z(w)\|_2}{1-\|\psi_z(w)\|_2},\qquad z,w\in \BB_n,
$$
where $\psi_z$ is the involutive automorphism of $\BB_n$ that
interchanges $0$ and $z$.

  Since  the $\delta$-topology
 coincides with the norm topology on  the open unit ball
$[B(\cH)^n]_1$, part (i) implies (iii). In \cite{Po-hyperbolic} , we
proved that
\begin{equation*}
\delta(A,B)=\ln \max \left\{ \left\|C_{A} C_{B}^{-1} \right\|,
  \left\|C_{B} C_{A}^{-1} \right\|\right\},\quad A, B\in [B(\cH)^n]_1,
\end{equation*}
where $C_X:=( \Delta_X\otimes I)(I-R_X)^{-1}$ and $R_X:=X_1^*\otimes
R_1+\cdots + X_n^*\otimes R_n$ is the reconstruction operator. Hence
and due to (i), part (iv) follows.
\end{proof}

  We showed in \cite{Po-hyperbolic} that if  $A,B\in [B(\cH)^n]_1^-$,
  then   $A\overset{H}{\sim}\, B$ if and only if $rA\overset{H}{\sim}\,
rB$ for any $r\in [0,1)$ and  \  $\sup_{r\in [0,
1)}\omega(rA,rB)<\infty$. Moreover, in this case,  the function
$r\mapsto \omega(rA,rB)$ is ancreasing on [0,1) and
$\omega(A,B)=\sup_{r\in [0, 1)}\omega(rA,rB)$. As a consequence, one
can see that if  $A\overset{H}{\sim}\, B$, then  the function
$r\mapsto {\bf d}(rA,rB)$ are increasing on $[0,1)$ and   $ {\bf
d}(A,B)=\sup_{r\in [0, 1)}{\bf d}(rA,rB). $

This result together with Theorem \ref{pseudo-t} can be used to
obtain  an explicit formula for the pseudohyperbolic metric  on any
Harnack part of $[B(\cH)^n]_1^-$.

Now we provide a Schwarz-Pick lemma for free holomorphic functions
on $[B(\cH)^n]_1$ with operator-valued coefficients, with respect to
the pseudohyperbolic metric.

\begin{theorem} \label{S-P} Let $F_j:[B(\cH)^n]_1\to B(\cH) \bar \otimes_{min}
B(\cE)$, $j=1,\ldots, m$, be free holomorphic functions with
coefficients in $B(\cE)$, and assume that $F:=(F_1,\ldots, F_m)$ is
a contractive free holomorphic function. If $X,Y\in [B(\cH)^n]_1$,
then
 $F(X)\overset{H}{\sim}\, F(Y)$ and
$$
{\bf d}(F(X), F(Y))\leq {\bf d}(X,Y),
$$
where ${\bf d}$ is the pseudohyperbolic  metric  defined on the
Harnack parts of $[B(\cH)^n]_1^-$.
\end{theorem}

\begin{proof} According to  the proof of Theorem 4.2 from \cite{Po-hyperbolic},
we have   $F(X)\overset{H}{\sim}\, F(Y)$ and $$\omega(F(X),
F(Y))\leq \omega(X,Y).$$ Using  the definition \eqref{Om} and the
fact that the function $f(x):=\frac{x^2-1}{x^2+1}$ is increasing on
the interval $[1, \infty)$, the result follows.
\end{proof}

  If $F:=(F_1,\ldots, F_m)$ is a
contractive  ($\|F\|_\infty\leq 1$) free holomorphic function  with
coefficients in $B(\cE)$, then   the evaluation map $\BB_n\ni
z\mapsto F(z)\in B(\cE)^{(m)}$ is a contractive  operator-valued
multiplier of the  Drury-Arveson space, and  any such a contractive
multiplier has  this type of representation.

\begin{corollary} \label{Sch-part}Let $F:=(F_1,\ldots, F_m)$ be
a contractive free holomorphic function  with coefficients in
$B(\cE)$.
 If $z,w\in \BB_n$, then $F(z)\overset{H}{\sim}\, F(w)$ and
$$
{\bf d}(F(z), F(w))\leq {\bf d}(z,w).
$$
\end{corollary}

\bigskip

       %


\begin{thebibliography}{99}

\bibitem{AK}  {\sc D.~Alpay and  D.S.~Kalyuzhnyi-Verbovetzkii},  Matrix-$J$-unitary
non-commutative rational formal power series, in: {\it The state
space method generalizations and applications}, pp. 49--113, Oper.
Theory Adv. Appl., {\bf 161}, Birkhäuser, Basel, 2006.

\bibitem{AF} {\sc T.~Ando and  K.~Fan},
Pick-Julia theorems for operators, {\it Math. Z.} {\bf 168} (1979),
23--34.



\bibitem{Arv}    {\sc W.B.~Arveson},
      {Subalgebras of $C^*$-algebras III: Multivariable operator theory,}
{\it Acta Math.} {\bf 181} (1998), 159--228.








 \bibitem{BB} {\sc J.A.~Ball and V.~Bolotnikov}, On bitangential interpolation
 problem for contractive-valued
 functions on the unit ball,
 {\it  Linear Algebra Appl.}
  {\bf 353} (2002), 107--147.



\bibitem{BGM1} {\sc J.~A.~Ball, G.~Groenewald, and T.~Malakorn},
 Conservative structured noncommutative multidimensional linear systems,
 {\it  The state space method generalizations and applications}, 179--223,
Oper. Theory Adv. Appl., {\bf 161}, Birkhäuser, Basel, 2006.



\bibitem{BGM2} {\sc J.~A.~Ball, G.~Groenewald, and T.~Malakorn},
  Bounded real lemma for sructured noncommutative multidimensional
  linear systems and robust control,
 {\it
Multilinear Systems and Signal Processing}  {\bf 17} (2006),
119--150.





\bibitem{Be} {\sc S.~ Bergman},
 {\it The kernel function and conformal mapping},
 Mathematical Surveys, No. V. American Mathematical Society,
Providence, R.I., 1970. x+257 pp.


\bibitem{Cara1} {\sc C.~Carath\'eodory}, \"Uber die Winkelderivierten
von beschr\"ankten analytischen Funktionen, {\it Sitzungsberg
Preuss. Acad.Phys.-Math.} {\bf 4}, 1--18.


  \bibitem{Cara2} {\sc C.~Carath\'eodory},{\em  Theory of Functions
  of a Complex Variable}, Chelsea Publishing Co., 1954.




\bibitem{Co} {\sc J.B.~Conway},
{\em Functions of one complex variable. I.} Second Edition. Graduate
Texts in Mathematics  {\bf 159}. { Springer-Verlag, New York}, 1995.



\bibitem{Cu} {\sc J.~Cuntz},
 Simple $C^*$--algebras generated by isometries,
 {\it  Commun. Math. Phys.}
 {\bf 57} (1977), 173--185.






\bibitem{D} {\sc K.~R.~Davidson}
Free semigroup algebras. A survey. {\it Systems, approximation,
singular integral operators, and related topics} (Bordeaux, 2000),
209--240, Oper. Theory Adv. Appl., 129, Birkhäuser, Basel, 2001.




 \bibitem{Dr} {\sc S.~Drury}, A generalization of von Neumann's inequality
 to the complex
 ball,
 {\it Proc. Amer. Math. Soc.} {\bf 68} (1978), 300--304.

\bibitem{KF1} {\sc K.~Fan}, Analytic functions of a proper
contraction, {\it Math. Z} {\bf 160} (1978), 275--290.

\bibitem{KF2} {\sc K.~Fan}, Julia's lemma for operators,
{\it Math. Ann.} {\bf 239} (1979), 241--245.





\bibitem{Ga} {\sc J.~Garnett},
 {\it  Bounded analytic functions},
 Academic Press, New York, 1981.

\bibitem{Gl} {\sc J.~Glicksberg}, Julia's lemma for function
algebras,  {\it Duke Math. J.} {\bf 43}(1976), 277--284.








\bibitem{Ha} {\sc L.A.~Harris}, Schwarz's lemma in normed linear
spaces, {\it Proc. Nat. Acad. Sci. USA} {\bf 62} (1969), 1014--1017.

\bibitem{Ha1} {\sc L.A.~Harris},
 Bounded symmetric homogeneous domains in infinite dimensional spaces,
 {\it Proceedings on Infinite Dimensional Holomorphy (Internat. Conf.,
 Univ. Kentucky, Lexington, Ky.}, 1973),
 pp. 13--40. {\it Lecture Notes in Math.}, Vol. 364, Springer, Berlin, 1974.






\bibitem{HKMS} {\sc J.W.~Helton, I.~Klep, S.~McCullough, and
N.~Slingled}, Noncommutative ball maps, {\it J. Funct. Anal.} {\bf
257} (2009),  47-87



\bibitem{Hi} {\sc E.~Hille}, {\em Analitic function theory II},
Boston: Ginn 1962.

\bibitem{K}
 {\sc D. S.~Kaliuzhnyi-Verbovetskyi}, Carath\' eodory interpolation on the non-commutative
 polydisk,
  {\it J. Funct. Anal.} {\bf  229} (2005), no. 2, 241--276.


\bibitem{KV} {\sc
 D. S.~Kaliuzhnyi-Verbovetskyi and V.~ Vinnikov},  Singularities of
rational functions and minimal factorizations: the noncommutative
and the commutative setting, {\it  Linear Algebra Appl.} {\bf 430}
(2009), no. 4, 869--889.


\bibitem{Kr} {\sc S.G.~Krantz}, {\it  Geometric function theory},
 Explorations in complex analysis.
 Cornerstones. Birkhäuser Boston, Inc., Boston, MA, 2006. xiv+314 pp.

\bibitem{J} {\sc G.~Julia}, Extension d'un lemma de Schwarz,
{\it Acta Math.} {\bf 42} (1920), 349--355.


\bibitem{MuSo1}  {\sc P.S.~Muhly and  B.~Solel},
 Tensor algebras over $C^*$-correspondences: representations,
dilations, and $C^*$-envelopes, {\it J. Funct. Anal.}
 {\bf 158} (1998),  389--457.




  \bibitem{MuSo2}  {\sc P.S.~Muhly and  B.~Solel},
Hardy algebras, $W^*$-correspondences and interpolation theory, {\it
Math. Ann.} {\bf 330} (2004),  353--415.

\bibitem{MuSo3}  {\sc P.S.~Muhly and  B.~Solel},
Schur Class Operator Functions and Automorphisms of Hardy Algebras,
{\it Documenta Math.} {\bf 13} (2008),  365--411.




      \bibitem{PPoS}  {\sc V.I.~Paulsen, G.~Popescu,  and D.~Singh},
         On Bohr's inequality,
        {\it Proc. London Math. Soc.}
        {\bf 85} (2002), 493--512.


\bibitem{Ph} {\sc R.S.~Phillips}, On symplectic mappings of
contraction operators, {\it Studia Math.} {\bf 31}(1968), 15--27.


\bibitem{Pic} {\sc G.~Pick},
 \"Uber die
 Beschr\"ankungen analytische Funktionen, welche durch vorgegebene
 Funktionswerte bewirkt werden,
 {\it Math. Ann.} (1916) {\bf 77}, 7--23.

\bibitem{Po-isometric} {\sc G.~Popescu}, Isometric dilations for infinite
sequences of noncommuting operators, {\it Trans. Amer. Math. Soc.}
{\bf 316} (1989), 523--536.

\bibitem{Po-charact} {\sc G.~Popescu}, Characteristic functions for infinite
sequences of noncommuting operators, {\it J. Operator Theory}
{\bf 22} (1989), 51--71.



      \bibitem{Po-von} {\sc G.~Popescu},
{Von Neumann inequality for $(B(H)^n)_1$,}
      {\it Math.  Scand.} {\bf 68} (1991), 292--304.




      \bibitem{Po-funct} {\sc G.~Popescu},
      {Functional calculus for noncommuting operators,}
       {\it Michigan Math. J.} {\bf 42} (1995), 345--356.



      \bibitem{Po-analytic} {\sc G.~Popescu},
      {Multi-analytic operators on Fock spaces,}
      {\it Math. Ann.} {\bf 303} (1995), 31--46.





      \bibitem{Po-poisson} {\sc G.~Popescu},
     {Poisson transforms on some $C^*$-algebras generated by isometries,}
       {\it J. Funct. Anal.} {\bf 161} (1999),  27--61.



\bibitem{Po-curvature} {\sc  G.~Popescu},
  Curvature invariant for Hilbert modules over free semigroup algebras,
   {\it Adv. Math.}
 {\bf 158} (2001), 264--309.



      \bibitem{Po-holomorphic} {\sc G.~Popescu},
      {Free holomorphic functions on the unit ball of $B(\cH)^n$},
      {\it J. Funct. Anal.}  {\bf 241} (2006), 268--333.



   \bibitem{Po-Bohr} {\sc G.~Popescu},
     Multivariable Bohr inequalities,
   {\it Trans. Amer.  Math. Soc.},
   {\bf 359} (2007), 5283--5317.




\bibitem{Po-free-hol-interp} {\sc G.~Popescu},
{ Free holomorphic functions and interpolation}, {\it Math. Ann.},
{\bf 342} (2008), 1-30.


      \bibitem{Po-pluri-maj} {\sc G.~Popescu},
      {Free pluriharmonic majorants and commutant lifting},
      {\it J. Funct. Anal.}  {\bf  255} (2008), 891-939.


 \bibitem{Po-hyperbolic2} {\sc G.~Popescu},
 {Hyperbolic geometry on the unit ball of $B(\cH)^n$ and dilation theory},
{\it Indiana Univ. Math. J.} {\bf  57} (2008), No.6, 2891-2930.

\bibitem{Po-pluriharmonic} {\sc G.~Popescu},
{Noncommutative transforms and free pluriharmonic functions},
 {\it Adv. Math.} {\bf 220} (2009), 831-893.


\bibitem{Po-automorphism} {\sc G.~Popescu},
{ Free holomorphic automorphisms of the unit ball of $B(\cH)^n$},
{\it J. Reine Angew. Math.}, to appear.




    \bibitem{Po-unitary} {\sc G.~Popescu},
      {Unitary invariants in multivariable operator theory},
      {\it Mem. Amer. Math. Soc.},  {\bf 200} (2009), No.941, 91
      pages.



\bibitem{Po-hyperbolic} {\sc G.~Popescu},
Noncommutative hyperbolic geometry on the unit ball of $B(H)^n$,
{\it J. Funct. Anal.}  {\bf 256} (2009), 4030-4070.




\bibitem{Po-domains} {\sc G.~Popescu},
Operator theory on noncommutative domains, {\it Mem.  Amer. Math.
Soc.}, to appear, 124 pages.

 \bibitem{Pota} {\sc V.P.~Potapov}, The multiplicative structure of
 $J$-contractive matrix functions,
 {\it Amer.Math. Soc. Transl.}(2) {\bf 15} (1960), 131--243.



 \bibitem{Ru1} {\sc W.~Rudin},
{\em Real and Complex Analysis}, { McGraw-Hill Book Co.} (1966)


\bibitem{Ru2} {\sc W.~Rudin},
{\em Function theory in the unit ball of \,$\CC^n$}, {
Springer-verlag, New-York/Berlin}, 1980.

\bibitem{Si} {\sc C.L.~Siegel}, Symplectic geometry,
{\it Amer. J. Math.} {\bf 65} (1943), 1--86.


\bibitem{Tit} {\sc E.C.~Titchmarsh}, {\em The Theory of Functions},
Oxford University Press, Oxford, 1939.



\bibitem{SzF-book} {\sc B.~Sz.-Nagy and C.~Foia\c{s}}, {\em Harmonic
Analysis of Operators on Hilbert Space}, North Holland, New York
1970.

\bibitem{V} {\sc D.V.~Voiculescu} Free analysis questions II: The
Grassmannian completion and the series expansions at the origin,
preprint.

 \bibitem{von}  {\sc J.~von Neumann},
      {Eine Spectraltheorie f\"ur allgemeine Operatoren eines unit\"aren
      Raumes,}
      {\it Math. Nachr.} {\bf 4} (1951), 258--281.


\bibitem{Y} {\sc N.J.~Young}, Orbits of the unit sphere of
$\cL(\cH,\cK)$ under symplectic transformations, {\it J. Operator
Theory} {\bf 11} (1984), 171--191.

\bibitem{Zhu} {\sc K.~Zhu}, {\it Spaces of holomorphic functions in the unit
ball},  Graduate Texts in Mathematics, 226. Springer-Verlag, New
York, 2005. x+271 pp.




       \end{thebibliography}
      \end{document}